\documentclass[11pt,reqno]{amsart}
\usepackage[margin=1.1in]{geometry}

\usepackage{amsmath,amsthm,amssymb}
\usepackage{mathrsfs} 
\usepackage{mathabx}
\usepackage{array} 
 
\usepackage{enumitem}

\usepackage{pgf,tikz,pgfplots}

\usetikzlibrary{arrows}

\usepackage{bm} % bold and blackboard bold math symbols
\usepackage{graphicx} % for easier grahics handling
\usepackage{float} % tells floats to stay [H]ere!
\usepackage{enumitem} % nice lists
\usepackage{fancyhdr} % nice headers
\usepackage{caption}  % to control figure and table captions
\usepackage{booktabs} % to create nice tables
\usepackage{tikz}
\usepackage[comma,sort&compress, numbers]{natbib}
\usepackage{subcaption}

\numberwithin{equation}{section} 

\usepackage[hidelinks, colorlinks=true, linkcolor = blue, urlcolor = blue]{hyperref} % to create hyperlinks
\usepackage{float} % tells floats to stay [H]ere!
\usepackage{enumitem} % nice lists
\usepackage{fancyhdr} % nice headers
\usepackage{caption}  % to control figure and table captions
\usepackage{booktabs} % to create nice tables
\usepackage{pgf,tikz}
\usetikzlibrary{arrows,automata,graphs}
\usepackage{algorithm2e}
\usepackage[numbers, comma]{natbib}
\usepackage{subcaption}

\newtheorem{thm}{Theorem}[section]
\newtheorem{corollary}[thm]{Corollary}
\newtheorem{proposition}[thm]{Proposition}
\newtheorem{lem}[thm]{Lemma}

\theoremstyle{definition}
\newtheorem{defn}[thm]{Definition}

\newtheorem{remark}[thm]{Remark}
\newtheorem{example}[thm]{Example}

\def\R{\mathbb{R}}

\def\P{\mathbb{P}}
\def\E{\mathbb{E}}

\def\Tr{\mathrm{Tr}}

\def\tr{\mathrm{tr}}
\def\Var{\mathrm{Var}}

\def\Unif{\mathrm{Unif}}

\def\dto{\stackrel{D}{\longrightarrow}}

\def\Spec{\mathrm{Spec}}

\usepackage{tikz} 

\tikzset{
    motif/.style={baseline=-0.5ex, scale=0.3}, % scale adjusts overall size
    node/.style={draw, circle, fill=black, inner sep=0pt, minimum size=1.85pt},
    edge/.style={->,> = latex'}
}

\newcommand{\cc}{
\tikz[motif]{
 \node[node] (1) at (0,0.58) {};
    \node[node] (2) at (-0.5,-0.29) {};
    \node[node] (3) at (0.5,-0.29) {};
    \draw[edge] (1) to (3) ;
    \draw[edge] (3) to (2) ;
    \draw[edge] (2) to (1) 
}}

\newcommand{\cac}{
\tikz[motif]{
 \node[node] (1) at (0,0.58) {};
    \node[node] (2) at (-0.5,-0.29) {};
    \node[node] (3) at (0.5,-0.29) {}; 
    \draw[edge] (1) to (2) ;
    \draw[edge] (2) to (3) ;
    \draw[edge] (3) to (1) 
}}

\newcommand{\lp}{
\tikz[motif]{
 \node[node] (1) at (0,0.58) {};
    \node[node] (2) at (-0.5,-0.29) {};
    \node[node] (3) at (0.5,-0.29) {}; 
    \draw[edge] (1) to (2) ; 
    \draw[edge] (3) to (1) 
}}

\newcommand{\cccc}{
\tikz[motif]{ 
    \node[node] (v0) at (0,0) {};
    \node[node] (v1) at ({-sqrt(3)/2},-0.5) {};
    \node[node] (v2) at ({-sqrt(3)/2}, 0.5) {};
    \node[node] (v3) at ({ sqrt(3)/2}, 0.5) {};
    \node[node] (v4) at ({ sqrt(3)/2},-0.5) {}; 
    \draw[->] (v0) -- (v1);
    \draw[->] (v1) -- (v2);
    \draw[->] (v2) -- (v0);
    \draw[->] (v0) -- (v3);
    \draw[->] (v3) -- (v4);
    \draw[->] (v4) -- (v0)    
    }}

\allowdisplaybreaks

\begin{document}

\title[Transitivity in Inhomogeneous Random Tournaments]{Transitivity in Inhomogeneous Random Tournaments }

\author[Chatterjee and Bhattacharya]{Sayak Chatterjee and Bhaswar B. Bhattacharya}
\address{Department of Statistics and Data Science\\ University of Pennsylvania\\ Philadelphia\\ PA 19104\\ United States}
\email{sayakc@wharton.upenn.edu} 
\address{Department of Statistics and Data Science\\ University of Pennsylvania\\ Philadelphia\\ PA 19104\\ United States} 
\email{bhaswar@wharton.upenn.edu}

\begin{abstract}  
Paired-comparison data are naturally represented by tournaments, where transitivity corresponds to the existence of a global ranking consistent with all pairwise outcomes. Accordingly, the classical \emph{Kendall--Smith coefficient of consistency} \cite{kendall1940method} measures deviations from transitivity in a tournament by counting the number of circular triads (directed $3$-cycles). In this paper, we characterize the fluctuations of the number of circular triads in inhomogeneous random tournaments and develop an inferential framework for the consistency coefficient. Specifically, we consider the $W$-random tournament model, where the comparison probabilities are determined by a \emph{tournamenton} $W$, the analogue of a graphon in the tournament setting. We show that, for a $W$-random tournament on $n$ vertices, the number of circular triads exhibits three different fluctuation regimes, determined by suitable notions of regularity and uniformity of $W$. We further develop a novel tournamenton multiplier bootstrap that consistently approximates the limiting distribution of the circular-triad count in the relevant asymptotic regime. Combining this with procedures for testing regularity and uniformity, we design an algorithm for constructing confidence intervals for the consistency coefficient that is asymptotically valid for all tournamentons.  We also obtain structural characterizations of tournamentons for which the limiting distribution of the number of circular triads exhibits specific degeneracies. These results can also be viewed through the lens of tournament quasirandomness and may be of independent interest. 
\end{abstract}


\maketitle

\section{Introduction}

Paired comparisons provide one of the most basic ways of eliciting preferences or relative strengths among a group of alternatives. They have a long history in statistics and psychometrics, where they form the basis for various classical models for ranking and choice, including the Bradley--Terry model, Thurstone's law of comparative judgment, and Luce's choice model \cite{thurstone1927law,bradley1952rank,luce1959model,david1988method,davidson1976bibliography}. They are also central to social choice theory, where collective preferences are often aggregated through pairwise majority comparisons and where preference cycles are closely connected to the Condorcet paradox, reflecting obstructions to consistent social ranking \cite{condorcet1785essay,arrow1951social,gehrlein2006condorcet}. More recently, paired comparisons have become a basic primitive in various machine-learning tasks, such as recommender systems and reinforcement learning from human feedback (see \cite{herbrich2007trueskill,burges2005learning,christiano2017deep,ouyang2022training,xiao2025theoretical,rafailov2023direct} among many others). {\it Tournaments}, that is, complete directed graphs, provide a natural representation of paired-comparison data. Specifically, given a collection of $n$ items indexed by $[n]:=\{1,2,\ldots,n\}$, the outcomes of all $\binom{n}{2}$ pairwise comparisons among them can be represented by a tournament $T_n$, where a directed edge $u\to v$ indicates that item $u$ defeats item $v$. It is often convenient to encode the edge directions of $T_n$ by its adjacency matrix $A_{T_n}=\bigl((a_{uv})\bigr)_{1\leq u,v\leq n}$, where 
\begin{equation*}
    a_{uv}
    =
    \begin{cases}
        1, & \text{if } u\to v,\\
        0, & \text{otherwise},
    \end{cases}
    \qquad
    a_{vu}=1-a_{uv} , 
\label{eq:adj}
\end{equation*}
for $1 \leq u < v \leq n$. By convention, the diagonal entries are set to $a_{uu}=\frac{1}{2}$, for $1\leq u\leq n$.

A tournament is said to be \emph{transitive} if its vertices can be ranked in such a way that every edge points from the higher-ranked vertex to the lower-ranked one. Such a global ranking is often the desired outcome of a paired-comparison experiment, since it provides a single coherent ordering of the alternatives for decision-making. In practice, however, pairwise comparison data often exhibit intransitivities. The simplest obstruction to transitivity occurs when there are three vertices $u,v,w\in[n]$ such that $u$ beats $v$, $v$ beats $w$, and $w$ beats $u$. This forms a directed 3-cycle, which \citet{kendall1940method} referred to as a {\it circular triad}. In the language of social choice, such a circular triad is precisely the local form of a Condorcet cycle. Circular triads, in fact, form the basic units of inconsistency in a tournament, since every directed cycle in a tournament contains a directed 3-cycle. Hence, a tournament is transitive if and only if it contains no directed 3-cycle. For this reason, \citet{kendall1940method} proposed using the number of circular triads as a fundamental measure for quantifying the failure of a tournament to admit a global ranking. Formally, the {\it Kendall-Smith coefficient of consistency} of a tournament $T_n$ is defined as
\begin{align}\label{eq:coefficientTn}
\zeta(T_n) :=
\begin{cases}
1-\dfrac{24 N_\triangle(T_n)}{n^3-n}, & \text{when } n \text{ is odd},\\[1.2ex]
1-\dfrac{24 N_\triangle(T_n)}{n^3-4n}, & \text{when } n \text{ is even},
\end{cases}
\end{align}
where $N_\triangle(T_n)$ denotes the number of circular triads in $T_n$. \citet{kendall1940method} showed that the maximum possible value of $N_\triangle(T_n)$, over all tournaments of size $n$, is $\frac{1}{24}(n^3-n)$ when $n$ is odd and $\frac{1}{24}(n^3-4n)$ when $n$ is even (see also \cite{connectedtournament,moon1968tournament}). This explains the normalization in \eqref{eq:coefficientTn}, which ensures that $\zeta(T_n)\in[0,1]$. Moreover, $\zeta(T_n)= 1$ (equivalently $N_\triangle(T_n) = 0$) if and only if $T_n$ is transitive. Thus, smaller values of $\zeta(T_n)$ indicate a larger number of circular triads and hence a greater degree of intransitivity, making it harder to summarize the comparisons by a single global ranking.

The number of circular triads in $T_n$ can be expressed in terms of its adjacency matrix $A_{T_ n}$ as follows: 
\begin{equation}
    N_\triangle(T_n)=\sum_{1\le u<v< w\le n}(a_{uv}a_{vw}a_{wu}+ a_{uw} a_{wv} a_{vu} ).
    \label{eq:N} 
\end{equation}
Note that the first summand in \eqref{eq:N} is 1 when the vertices $\{u, v, w\}$ form a clockwise 3-cycle $(\cc)$, while the second summand is 1 when they form an anti-clockwise 3-cycle $(\cac)$. 
%Using the relation $a_{vu}=1-a_{uv}$, for $1\le u<v\le n,$ $N_\triangle(T_n)$ can be alternatively expressed as follows: 
The distributional theory of $N_\triangle(T_n)$, or equivalently of $\zeta(T_n)$, is classical for the uniform random tournament model, in which the direction of each edge is determined independently by a fair coin flip, that is,
$$a_{uv} \sim \mathrm{Ber}\!\left(\tfrac{1}{2}\right)
\quad \text{and} \quad
a_{vu}=1-a_{uv},$$
independently for $1\leq u<v\leq n$. We will denote this model by $T(n, \frac{1}{2})$.  In the uniform model, \citet{kendall1940method} provided formulas for the first four moments of $N_\triangle(T_n)$, and later \citet{moran1947method} established the asymptotic normality of $N_\triangle(T_n)$ in the same model. The uniform model, however, assumes complete homogeneity among the vertices: every player is equally likely to defeat every other player. This assumption is often unrealistic in applications, since players typically have different latent qualities that affect the comparison probabilities. A natural way to incorporate such inhomogeneity is to allow the edge-orientation probabilities to depend on underlying node-specific latent variables. To this end, analogous to the role of graphons in generating inhomogeneous random graphs \cite{lovasz_book,bickel2009nonparametric,bickel2011method,exchangeablegraph}, the notion of \emph{tournamentons} \cite{exchangeablegraph,thornblad2018decomposition,thornblad2016tournamentlimits,hladky2025digraphons,sah2024intransitive,grzesik2023cycles,impartialdigraphs} provides a flexible nonparametric framework for modeling exchangeable inhomogeneous random tournaments.  Specifically, a {\it tournamenton} is a measurable function $W:[0,1]^2 \to [0,1]$ such that
$W(x,y)+W(y,x)=1$, for all $x,y\in[0,1]$. Tournamentons may be viewed as continuum limits of adjacency matrices of finite tournaments, using which one can generate inhomogeneous random tournaments as follows: 

\begin{defn}[$W$-random tournament model]
Given a tournamenton \(W:[0,1]^2\to[0,1]\), a \(W\)-\emph{random tournament} on the vertex set \([n]:=\{1,2,\ldots,n\}\) is obtained by orienting the edge \(u\to v\) with probability \(W(\eta_u,\eta_v)\), independently for \(1\le u<v\le n\), where \(\{\eta_u:1\le u\le n\}\) is an i.i.d. collection of \(\Unif([0,1])\) random variables. Equivalently, it is a tournament with adjacency matrix \(((a_{uv}))_{1\le u,v\le n}\), where
\[
a_{uv}\sim \mathrm{Ber}\bigl(W(\eta_u,\eta_v)\bigr)
\quad\text{and}\quad
a_{vu}=1-a_{uv},
\]
independently for \(1\le u<v\le n\). We denote this model by $T(n,W)$, and write $T_n \sim T(n,W)$ to indicate that $T_n$ is a random tournament sampled from this model. 
\label{definition:W} 
\end{defn}

Note that taking $W(x,y)=\frac{1}{2}$, for every $x,y \in [0, 1],$ in the above definition gives the uniform random tournament $T(n, \frac{1}{2})$. We will denote this tournamenton by $W_{1/2}$. Another important special case is the {\it Condorcet tournamenton}, where  
\begin{align}\label{eq:randomW}
W(x,y)= W_p(x, y) : = 
\begin{cases}p&\text{ if }x<y,\\
1-p&\text{ if }x>y, \\ 
\frac{1}{2}&\text{ if }x=y , \\
\end{cases} 
\end{align} 
for $p \in [0, 1]$. 
In this case, a realization of $T(n,W)$ represents the outcome of a Condorcet random tournament among $n$ players whose rankings are determined by a latent uniform random permutation (generated using i.i.d. uniform random variables), with the better player winning each match with a fixed probability $p \in [\frac{1}{2}, 1)$ \cite{frank1968stochastic,luczak1996evolution,saile2020robust,manurangsi2022generalized,kunisky2024inference,williams2010fixing,kim2017singleelimination}.\footnote{Note that in the classical Condorcet random tournament the players' ranking is fixed in advance, while in our model it is determined by the ranking induced by latent uniform random variables.} Definition \ref{definition:W} also includes a latent variable version of the celebrated Bradley-Terry-Luce (BTL) model \cite{bradley1952rank,luce1959model}, 
where 
$$W(x, y) = \frac{f(x)}{f(x)+f(y)},$$
for some reward function $f:[0, 1] \rightarrow \R$. More generally, the $W$-random tournament model can be viewed as a latent variable analogue of the generalized random model studied in \cite{manurangsi2022generalized,saile2020robust} (see Section \ref{sec:examples} for  more examples).

Given an inhomogeneous random tournament $T_n\sim T(n,W)$, taking expectation on  both sides of \eqref{eq:N} gives, 
\begin{align}\label{eq:Etriangle}
\mathbb E N_\triangle(T_n) & = \binom{n}{3} \int_{[0,1]^3}  \left( W(x,y)W(y,z)W(z,x) \mathrm dx \mathrm dy \mathrm dz + W(x,z)W(z,y)W(y,x) \right)  \mathrm dx \mathrm dy \mathrm dz  \nonumber \\ 
& = 2 \binom{n}{3} \int_{[0,1]^3} W(x,y)W(y,z)W(z,x) \mathrm dx \mathrm dy \mathrm dz. 
\end{align} 
Consequently, we define the {\it Kendall-Smith coefficient of consistency for a tournamenton $W$} as 
\begin{align}\label{eq:transitivityW}
\zeta(W) :=1-8 \int_{[0,1]^3} W(x,y)W(y,z)W(z,x) \mathrm dx \mathrm dy \mathrm dz . 
\end{align}
This is the population (continuum) analogue of \eqref{eq:coefficientTn}. In particular, if $T_n \sim T(n, W)$ is a $W$-random tournament, then $\zeta(T_n)$ is a natural empirical estimate of $\zeta(W)$. The maximality result of  \citet{kendall1940method} for $\zeta(T_n)$, adapted to the continuum setting \cite{grzesik2023cycles} implies that $\zeta(W) \in [0, 1]$. Moreover, $\zeta(W)=1$ if and only if $W$ is the {\it transitive tournamenton} $W_{\mathrm{tr}} (x, y) = \bm 1\{x < y\}$, up to a measure preserving transformation (see \cite[Theorem 5.4]{thornblad2018decomposition}); and $\zeta(W)=0$ (which corresponds to the maximal possible density of circular triads) if and only if $W$ is degree-regular \cite[Theorem 9]{grzesik2023cycles} (see Remark \ref{remark:degree} for the definition of degree-regularity).

To assess the statistical validity of $\zeta(T_n)$ in estimating the population measure $\zeta(W)$ it is essential to understand the fluctuations (asymptotic distribution) of $\zeta(T_n)$ (equivalently that of $N_{\triangle}(T_n)$) and to provide uncertainty quantification through confidence intervals. In this paper, we address these questions by establishing the following results:

\begin{itemize} 

\item {\it Asymptotic distribution of $N_{\triangle}(T_n)$}: 
Our first main result gives a complete description of the asymptotic distribution of $N_\triangle(T_n)$, where $T_n \sim T(n, W)$ is sampled from any given tournamenton $W$. To this end, we introduce the notion of $\triangle$-regularity for $W$. Intuitively, this means that the density of circular triads incident to any given ``vertex'' of $W$ is constant almost everywhere (see Definition~\ref{defn:trianglefunction}). Depending on whether this regularity condition holds, three distinct fluctuation regimes emerge (see Theorem \ref{thm:N}): 
\begin{itemize} 

\item If $W$ is not $\triangle$-regular, then $N_\triangle(T_n)$ has fluctuations of order $n^{\frac{5}{2}}$ and the limiting distribution is Gaussian (after appropriate centering and scaling). 

\item However, if $W$ is $\triangle$-regular, then the above  mentioned Gaussian limit of $N_\triangle(T_n)$ is degenerate. In this case, $N_\triangle(T_n)$ has fluctuations of order $n^2$ and the limiting distribution has, in general, a Gaussian component and another independent (non-Gaussian) component which is a (possibly) infinite weighted sum of centered chi-squared random variables. The weights are determined by the spectrum of a graphon constructed from the two-point conditional densities of circular triads in $W$ (see Definition \ref{defn:triangleW}).

\item Finally, we show that the \emph{uniform tournamenton} $W\equiv \frac{1}{2}$ is the only $\triangle$-regular tournamenton for which the asymptotic distribution is degenerate even at the $O(n^2)$ scale (excluding other trivial degeneracies). In this case, the fluctuations are of order $n^{\frac{3}{2}}$, and the limiting distribution is Gaussian, recovering the classical result of Moran~\cite{moran1947method}.

\end{itemize}

\item {\it Confidence interval for $\zeta(W)$}: Next, using the distributional results described above, we design an algorithm for constructing an asymptotically valid confidence interval for $\zeta(W)$. In this direction, we obtain the following results:  

\begin{itemize}

\item The first step in constructing a confidence interval for \(\zeta(W)\) is to estimate the quantiles of the limiting distribution of \(N_\triangle(T_n)\). When this limiting distribution is Gaussian, as in the non-\(\triangle\)-regular case, the quantiles can be obtained by consistently estimating the asymptotic variance. The $\triangle$-regular case is more delicate, since the limiting distribution in this regime has a non-Gaussian component. To address this, in Section \ref{sec:estimateN}, we introduce the \emph{tournamenton multiplier bootstrap}. On a high level, the method estimates the limiting distribution in the \(\triangle\)-regular regime by introducing external randomness through independent Gaussian multipliers and replacing the spectral weights in the limiting distribution with empirical eigenvalues of a matrix constructed from the observed tournament \(T_n\). We show that, conditional on $T_n$, the resulting estimated distribution converges to the desired limiting distribution under $\triangle$-regularity, without any additional assumptions on $W$ (see Theorem~\ref{thm:TnJ}).

\item The multiplier bootstrap alone, however, is not sufficient to construct a confidence interval for $\zeta(W)$, because we have no apriori knowledge regarding the $\triangle$-regularity or uniformity of $W$. To address this, we develop two intermediate tests: one for whether $W$ is $\triangle$-regular (see Section \ref{sec:r}), and another for whether $W$ is uniform, that is, $W\equiv \frac{1}{2}$ (see Section \ref{sec:uniform}). In particular, for the latter, we invoke the notion of {\it quasirandom forcing} tournaments from extremal combinatorics (see \cite{qt,localtournament,coregliano2017density,hancock2023no,kral2026sidorenko,noel2025forcing} and the references therein). This shows that certain fixed-size tournaments (such as a transitive tournament on four vertices \cite{coregliano2017density}) have the following intriguing property: if their density in a tournamenton matches the corresponding density in the uniform tournamenton $W_{1/2}$, then the tournamenton itself must be $W_{1/2}$ almost everywhere. Thus, by estimating the density of a single fixed-size tournament in $T_n$, one can construct a consistent test for the property that the underlying tournamenton is uniform (see Proposition \ref{prop:Huniform}). 

\item Finally, in Section \ref{sec:cWmethod}, combining the outcomes of the intermediate tests with the tournamenton multiplier bootstrap, we construct a confidence interval for the Kendall--Smith consistency coefficient $\zeta(W)$ that is asymptotically valid for all choices of $W$ (see Theorem~\ref{thm:Ln}). 
%Simulation results demonstrating the empirical performance of the proposed algorithm are presented in Section~\ref{sec:simulation}. 

\end{itemize}

\end{itemize}

%Throughout, we illustrate the general theory with various examples (see Section~\ref{sec:examples}) and in simulations (see Section~\ref{sec:simulation}). In the $\triangle$-regular regime, we also discuss conditions under which exactly one of the two components of the limiting distribution vanishes, that is, the limit is either Gaussian or has no Gaussian component. This leads us to a conjecture on the forcibility of the Condorcet random model which, if true, would imply that it is the only $\triangle$-regular tournamenton for which $N_\triangle(T_n)$ is asymptotically Gaussian (see Conjecture \ref{conjecture:W}).

%In other words, the Condorcet random model is the only $\triangle$-regular tournamenton for which $N_\triangle(T_n)$ is asymptotically Gaussian. We record this formally below: 

In the \(\triangle\)-regular regime, we also characterize the structure of tournamentons where exactly one of the two components of the limiting distribution, either the Gaussian or the non-Gaussian component, is degenerate. In particular, we show that the only \(\triangle\)-regular tournamenton for which \(N_\triangle(T_n)\) is asymptotically Gaussian corresponds to the Condorcet tournamenton  \eqref{eq:randomW}. This can also be equivalently formulated as a quasirandom-forcing type result for tournaments: if the density of circular triads passing through every pair of continuum vertices in a tournamenton is a constant, then it must be the Condorcet tournamenton \(W_p\), for some \(p\in[0,1]\), up to a measure preserving transformation. To the best of our knowledge, this is among the first forcing results for a non-uniform tournamenton, namely the Condorcet model.

\section{ Fluctuations of Directed 3-Cycles }

In this section, we present our results on the limiting distribution of circular triads in $W$-random tournaments. We begin by discussing the relevant preliminaries in Section~\ref{sec:W}. The distributional results for $N_\triangle(T_n)$ are formally stated in Section~\ref{sec:triangle}. The degeneracy properties of the limiting distribution are discussed in Section~\ref{sec:limitregular}. In Section~\ref{sec:examples} we compute the limiting distribution in various examples.

\subsection{Preliminaries} 
\label{sec:W}

A quantity that will play a central role in our analysis is the homomorphism density of a fixed directed graph $H=(V(H),E(H))$ (with no self-loops or multiple directed edges), in a tournamenton $W$, which is defined as follows:
\begin{equation}
    t(H,W)=\int_{[0,1]^{|V(H)|}}\prod_{(s,t)\in E(H)}W(x_s,x_t)\prod_{s\in V(H)}\mathrm{d}x_s.
    \label{eq:tHW}
\end{equation}
Note that, as in the graphon setting \cite{lovasz_book}, \eqref{eq:tHW} may be viewed as the natural continuum analogue of the homomorphism density of a fixed directed graph $H=(V(H),E(H))$ in a finite tournament $T=(V(T),E(T)),$ which is defined as: 
$$t(H,T)=\frac{|\mathrm{hom}(H,T)|}{|V(T)|^{|V(H)|}},$$
where $\mathrm{hom}(H,T)$ is the collection of homomorphisms, that is, directed-edge-preserving maps, from $V(H)$ into $V(T)$. It is also helpful to define the {\it empirical tournamenton} associated with $T$ as: 
\begin{align}\label{eq:Tn} 
W^{T}(x, y) = 
\begin{cases} 
1 &\text{ if } \lceil |V(T)| x\rceil \rightarrow \lceil |V(T)|y \rceil , \\
\frac{1}{2}&\text{ if } \lceil |V(T)| x\rceil = \lceil |V(T)|y \rceil , \\ 
0 &\text{ otherwise}, \\ 
\end{cases} 
\end{align}  
for $x, y \in [0, 1]$. 
In other words, we partition $[0,1]^2$ into $|V(T)|^2$ squares, each of side length $\frac{1}{|V(T)|}$. Then define $W^T(x,y)=\frac{1}{2}$, whenever $(x,y)$ lies in a diagonal square. For $(x,y)$ in an off-diagonal square define $W^T(x,y)=1$, if there is a directed edge $\lceil |V(T)|x \rceil \to \lceil |V(T)|y \rceil$ in $T$, and $W^T(x,y)=0$ otherwise. 
%The error term \eqref{} accounts for maps in which two vertices of \(H\) are sent to the same vertex of \(T\), or equivalently for the diagonal blocks of \(W^T\).

In the special case when $H$ is a circular triad, oriented either clockwise ($\cc$) or anti-clockwise ($\cac$), the homomorphism density \eqref{eq:tHW} simplifies to
\begin{equation*} 
t(\cc, W) = t(\cac, W)=\int_{[0,\,1]^3}W(x,\,y)W(y,\,z)W(z,\,x)\mathrm{d}x\mathrm{d}y\mathrm{d}z.
    \label{eq:triad_hom}
\end{equation*} 
Hence, when $T_n \sim T(n, W)$ is a $W$-random tournament, \eqref{eq:Etriangle} can be expressed as: 
\begin{align}\label{eq:EN}
   \E [N_\triangle(T_n)] 
   %& =\sum_{1\le u<v< w\le n} \E\left[ W(\eta_u, \eta_v) W(\eta_v, \eta_w) W(\eta_w, \eta_u) +  W(\eta_u, \eta_w) W(\eta_w, \eta_v)  W(\eta_v, \eta_u) \right]  \nonumber \\ 
   & = 2\binom{n}{3} t(\cc, W) . 
\end{align} 
We now introduce the notion of conditional homomorphism densities for circular triads and the notion of $\triangle$-{\it regularity}. Throughout, $\eta_1, \eta_2, \eta_3$ will denote i.i.d. $\Unif([0,1])$ random variables.

\begin{defn}[1-point conditional homomorphism density and $\triangle$-regularity] 
Given a tournamenton $W$, the 1-point conditional homomorphism density function for the circular triad (oriented either clockwise or anticlockwise) is defined as: For $x \in [0, 1]$, 
\begin{align}
    t^{\triangle}_{W}(x) &= \E\left[ W(\eta_1, \eta_2) W(\eta_2, \eta_3) W(\eta_3, \eta_1) \mid \eta_{1}=x \right].     \nonumber \\ 
&=\int_{[0,1]^2}W(x,y)W(y,z)W(z,x)\mathrm{d}y\mathrm{d}z .
    \label{eq:g}
\end{align} 
The tournamenton $W$ is said to be $\triangle$-{\it regular}, if the function $t^{\triangle}_{W}(x)$ is constant for almost every $x \in [0,\,1].$ (Note that if $t^{\triangle}_{W}(x)$ is constant almost everywhere, by integrating both sides of \eqref{eq:g} it follows that the constant must be $t(\cc, W)$.) 
\label{defn:trianglefunction}
\end{defn}

\begin{remark}
\label{remark:conditional}
If we interpret the interval $[0,1]$ as a continuum of vertices, then $t^{\triangle}_{W}(x)$ can be viewed as the density of circular triads when one of their vertices is mapped to the vertex $x \in [0,1]$ in the continuum. Hence, a tournamenton $W$ is $\triangle$-regular if the density of circular triads incident to almost every continuum vertex $x \in [0,1]$ is constant. Clearly, the uniform tournamenton $W_{1/2}$ is $\triangle$-regular. More generally, the Condorcet tournamenton \eqref{eq:randomW} is also $\triangle$-regular, for every $p \in [0, 1]$. In particular, in this case by a direct computation it can be shown that $t^{\triangle}_{W}(x) = \frac{1}{2} p (1-p)$, for all $x \in [0, 1]$. 
\end{remark}

The notion of $\triangle$-regularity is also related to the more familiar notion of degree regularity.

\begin{remark}[Degree regularity]
For a tournamenton $W$, define its \textit{out-degree} and \textit{in-degree} functions, respectively, by
\begin{align}\label{eq:directeddegree}
d^{\uparrow}_W(x):=\int_0^1 W(x,y)\,\mathrm{d}y
\quad\text{and}\quad d^{\downarrow}_W(x):=\int_0^1 W(z,x)\,\mathrm{d}z , 
\end{align} 
for $x\in[0,1]$. Note that $d^{\uparrow}_W(x)+d^{\downarrow}_W(x)=1$, for almost every $x\in[0,1]$. The out-degree function, in particular, is the continuum analogue of the normalized \emph{score sequence} (the fraction of wins of a player in a tournament), which is a fundamental object that has been extensively studied over the years (see \cite{harary1966theory,landau1953dominance,winston1983asymptotic,bassan2026tournament,moon1968tournament} and the references therein). A tournamenton $W$ is said to be \textit{degree-regular} if $d^{\uparrow}_W(x)$ (equivalently $d^{\downarrow}_W(x)$) is constant almost everywhere on $[0,\,1]$. 
Note that if $W$ is degree-regular, then the constant value of $d^{\downarrow}_W(x)$ (and consequently, $d^{\uparrow}_W(x)$) has to be equal to $\frac{1}{2}$, since $\int_0^1 d^{\uparrow}_W(x) \mathrm d x = \int_0^1  d^{\downarrow}_W(x) \mathrm d x = \frac{1}{2}$. Clearly, the uniform tournamenton $W_{1/2}$ is degree-regular. Another simple example of a degree-regular tournamenton, which we will revisit again later, is the empirical tournamenton (recall \eqref{eq:Tn}) associated with the circular triad itself (see Figure \ref{fig:cc}): 
\begin{align}
W^{\cc}(x, y) = 
\begin{cases} 
1 &\text{ if } (x, y) \in [\frac{1}{3}, \frac{2}{3}] \times  [0, \frac{1}{3}] \bigcup  [\frac{2}{3}, 1] \times  [\frac{1}{3}, \frac{2}{3}]  \bigcup  [0, \frac{1}{3}] \times  [\frac{2}{3}, 1]  , \\
 \frac{1}{2} & \text{ if } (x, y) \in [0, \frac{1}{3}]^2 \bigcup  [\frac{1}{3}, \frac{2}{3}]^2 \bigcup [\frac{2}{3}, 1]^2 , \\ 
0 &\text{ otherwise} . 
\end{cases} 
\label{eq:cc}
\end{align}  
This can also be interpreted as a `rock--paper--scissors' type block tournament model: Players are assigned to three groups, corresponding to rock, paper, and scissors. Players within the same group are evenly matched, so the direction of an edge between two players in the same group is assigned at random, with probability $\frac{1}{2}$ in each direction. Between different groups, the edge directions follow the usual cyclic dominance rule: rock beats scissors, scissors beats paper, and paper beats rock. In Lemma \ref{lem:deg_reg} (see Appendix \ref{sec:degreepf}) we show that degree-regularity is a sufficient condition for $\triangle$-regularity. The converse, however, is not necessarily true (see Remark \ref{remark:randomdegree}). Another important property of degree-regularity is that it characterizes tournamentons attaining the maximal value $t(\cc,W) = \frac{1}{8}$ (equivalently the minimal value $\zeta(W)=0$) \cite[Theorem~9]{grzesik2023cycles}. \label{remark:degree}
\end{remark}

\begin{figure}[h]
\vspace{-0.475in}
    %\begin{minipage}{0.35\linewidth}
        \centering
        \includegraphics[width=0.35\linewidth]{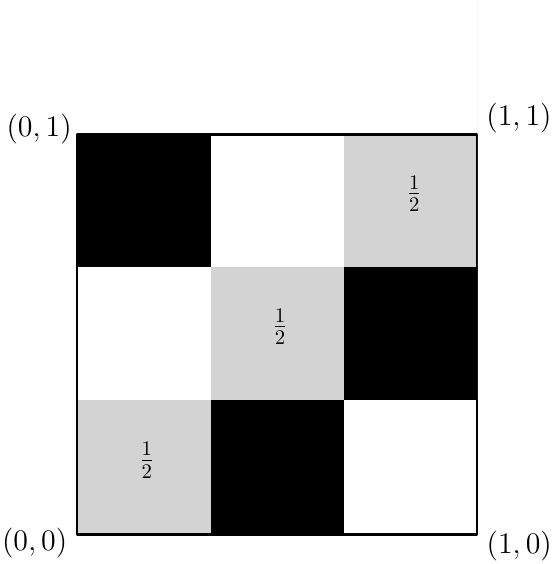}  \\ 
        %{(a)}
    %\end{minipage}
    \caption{\small{The empirical tournamenton \eqref{eq:cc}. The function takes the value $1$ on the black regions, $0$ on the white regions, and $\frac{1}{2}$ on the grey region. } } 
    \label{fig:cc}
\end{figure}

Next, we define the 2-point conditional homomorphism density of circular triads and properties of the associated kernel: 

\begin{defn}[2-point conditional homomorphism density] 
\label{defn:triangleW} 
Given a tournamenton $W$, the 2-point conditional homomorphism density function for the clockwise oriented circular triad $(\cc)$ is defined as: 
For $x, y \in [0, 1]$, 
\begin{align*}
    t_{W}(\cc, x, y) &= \E\left[ W(\eta_1, \eta_2) W(\eta_2, \eta_3) W(\eta_3, \eta_1) \mid \eta_{1}=x, \eta_2 = y \right].     \\ 
&= W(x,\,y)\int_0^1W(y,\,z)W(z,\,x)\mathrm{d}z  .
\end{align*}  
Similarly, the 2-point conditional homomorphism density function for the anti-clockwise oriented circular triad $(\cac)$ is defined as  
\begin{align*}
    t_{W}(\cac, x, y) &= \E\left[ W(\eta_1, \eta_3) W(\eta_3, \eta_2) W(\eta_2, \eta_1) \mid \eta_{1}=x, \eta_2 = y \right].     \\ 
&= W(y,\,x)\int_0^1W(x,\,z)W(z,\,y)\mathrm{d}z  =  t_{W}(\cc, y, x).
\end{align*}  
Finally, define {\it symmetrized $2$-point conditional homomorphism density} as: 
\begin{align}\label{eq:W_tr}
    \triangle_{W}(x,\,y) :=\frac{t_{W}(\cc, x, y) + t_{W}(\cac, x, y) }{2} . 
\end{align} 
\end{defn}

\begin{remark}  
Intuitively,  $t_{W}(\cc, x, y)$ (respectively, $t_{W}(\cac, x, y)$) can be interpreted as the homomorphism density of $\cc$ (respectively, $\cac$) incident on the continuum vertices $x, y \in [0, 1]$. It is also helpful for us to define the analogue of the 2-point conditional homomorphism density for the directed path of length 2. Specifically, for the oriented 2-path $(\lp)$ define: 
\begin{align}\label{eq:rpxy} 
    t_{W}(\lp, x, y) &= \E\left[ W(\eta_2, \eta_3) W(\eta_3, \eta_1) \mid \eta_{1}=x, \eta_2 = y \right].   \nonumber  \\ 
&= \int_0^1W(y,\,z)W(z,\,x)\mathrm{d}z.
\end{align}  
In this notation, \eqref{eq:W_tr} can be expressed as: 
\begin{align}
    \triangle_{W}(x,\,y) :=\frac{W(x, y) t_{W}(\lp, x, y) + W(y, x) t_{W}(\lp, y, x) }{2} . 
    \label{eq:Wp}
\end{align} 
From the definitions it is easy to check that (see Lemma \ref{lem:rp} in Appendix \ref{sec:degreepf}): 
\begin{align}\label{eq:tWrplp} 
t_{W}(\lp, x, y) - t_{W}(\lp, y, x) = d^{\uparrow}_W(y)-d^{\uparrow}_W(x), 
\end{align}
for $x, y \in [0, 1]$. Hence, the kernel $t_{W}(\lp, x, y)$ is symmetric if and only if $W$ is degree-regular. In this case, \eqref{eq:Wp} simplifies to $ \triangle_{W}(x,\,y) = \frac{1}{2} t_{W}(\lp, x, y)$. 
\end{remark}

Note that the kernel $\triangle_{W}$ is a symmetric function from $[0,\,1]^2$ to $[0, 1]$, that is, it is a graphon. Thus, it defines an operator $T_{\triangle_{W}}$ on $L^2[0,\,1]$ as follows:
$$(T_{\triangle_{W}} f)(x)=\int_0^1 \triangle_{W}(x,\,y)f(y)\mathrm{d}y.$$ 
for $f \in L^2[0,\,1]$. 
This is a symmetric Hilbert-Schmidt operator, thus it is compact and has
a discrete spectrum, that is, it has a countable multiset of non-zero real eigenvalues, which we denote by $\Spec(\triangle_{W}).$ Observe that 
\begin{equation}
    \sum_{\lambda\in\Spec(\triangle_{W})}\lambda^2=\int_{[0,\,1]^2}(\triangle_{W}(x,\,y))^2\mathrm{d}x\mathrm{d}y\le1 , 
    \label{eq:summable}
\end{equation}
since $\triangle_{W}$ is bounded by $1.$ Also, recalling \eqref{eq:g}, note that
\begin{align}\label{eq:Wdegree}
\int_0^1 \triangle_W(x,y)\,\mathrm{d}y = t_W^{\triangle}(x) , 
\end{align}
for all $x\in[0,1]$. Therefore, if $W$ is $\triangle$-regular, that is, if $t_W^{\triangle}(x)=t(\cc, W)$ almost everywhere on $[0,1]$, then $t(\cc, W)$ is an eigenvalue of $T_{\triangle_W}$ with corresponding eigenfunction $\phi\equiv \bm{1}$. In this case, we write $\Spec^-(\triangle_W)$ for the multiset obtained from $\Spec(\triangle_W)$ by decreasing the multiplicity of the eigenvalue $t(\cc, W)$ by $1$. 
%Then $\Spec^-(\triangle_{W})$ gives the spectrum of an operator defined on the orthogonal compliment of the space of constant functions in $L^2[0,\,1].$

\subsection{ Asymptotic Distribution of $N_{\triangle}(T_n)$}
\label{sec:triangle}

Throughout, $\cccc$ will denote the directed graph on 5 vertices obtained by joining two clockwise oriented 3-cycles at a vertex. The homomorphism density of this directed graph in a tournamenton $W$ can be expressed as (recall \eqref{eq:tHW}):
\begin{align}
    t(\cccc \,,\, W) & = \int_{[0,\,1]^5} W(x,\,y)W(y,\,z)W(z,\,x) W(x,\,y')W(y',\,z')W(z',\,x) \mathrm{d}x\mathrm{d}y\mathrm{d}z \mathrm{d}y'\mathrm{d}z ' \nonumber \\ 
    & = \int_0^1 t^{\triangle}_W(x)^2 \mathrm{d} x , 
    \label{eq:ccjoin}
\end{align} 
where $t_{W}^\triangle$ is as defined in \eqref{eq:g}. With this notation, we are now ready to state our result on the asymptotic distribution of the number of circular triads in $W$-random tournaments. The proof is given in Section \ref{sec:Npf}.

\begin{thm} Fix a tournamenton $W$ with $\zeta(W) \in [0, 1)$ and consider the $W$-random tournament $T_n \sim T(n, W)$. Then for $N_{\triangle}(T_n)$ as defined in \eqref{eq:N} the following hold: 
\begin{enumerate}

\item[$(1)$] For any $W$, 
\begin{equation}
    \frac{N_{\triangle}(T_n)- 2\binom{n}{3} t(\cc, W)  }{n^{\frac{5}{2}}} \dto N(0,\,\sigma^2_W) , 
    \label{eq:irreg}
\end{equation}
where 
\begin{align}
\sigma^2_W := t(\cccc \,,\, W) - t(\cc, \, W)^2 = \int_0^1 t^{\triangle}_W(x)^2 \mathrm{d} x  - \left( \int_0^1 t^{\triangle}_W(x) \mathrm{d} x \right)^2 . 
\label{eq:irvariance}
\end{align} 
Moreover, the distribution in \eqref{eq:irreg} is non-degenerate $($that is, $\sigma^2_W  > 0)$ if and only if $W$ is not $\triangle$-regular.

\item[$(2)$] If $W$ is $\triangle$-regular, then
\begin{equation}
    \frac{N_{\triangle}(T_n) - 2\binom{n}{3} t(\cc, W) }{n^{2}}\dto \tau_W \cdot Z+\sum_{\lambda\in\Spec^-(\triangle_{W})}\lambda(Z_\lambda^2-1) , 
    \label{eq:reg}
\end{equation}
where $Z, \{Z_\lambda\}_{\lambda\in\Spec^-(\triangle_{W})}$ are i.i.d. $N(0,\,1)$ and 
\begin{equation}
    \tau_W^2 : =\frac{1}{2} \int_{[0, 1]^2} W(x,\,y)W(y,\,x)(t_W( \lp, x, y )-t_W( \lp, y, x ))^2 \mathrm{d} x \mathrm{d} y .  
    \label{eq:tau_W}
\end{equation}  
Moreover, the limiting distribution in \eqref{eq:reg} is non-degenerate if and only if $W(x,\,y) \ne \frac{1}{2}$ on a set of positive measure. 

\item[$(3)$] If $W(x,\,y)=\frac{1}{2}$ almost everywhere, then 
\begin{align}\label{eq:Wuniform}
\frac{N_{\triangle}(T_n)-\frac{1}{4} \binom{n}{3} }{n^{\frac{3}{2}}}\dto N\left(0,\,\tfrac{1}{32}\right). 
\end{align} 

\end{enumerate}

\label{thm:N}
\end{thm}

Theorem \ref{thm:N} provides a complete characterization of the limiting distribution of $N_\triangle(T_n)$ in the $W$-random tournament model. Specifically, three distinct asymptotic regimes emerge (depending on the structure of $W$):

\begin{itemize} 

\item When $W$ is not $\triangle$-regular, the fluctuations of $N_{\triangle}(T_n)$ are of order $n^{\frac{5}{2}}$ and the limiting distribution is Gaussian (as in \eqref{eq:irreg}). 

\item When $W$ is $\triangle$-regular,  the fluctuations of $N_{\triangle}(T_n)$ are of order $n^{2}$ and the limiting distribution has, in general, a Gaussian component and another independent (non-Gaussian) component which is a (possibly) infinite weighted sum of centered chi-squared random variables (as in \eqref{eq:reg}). Note that although the second term in \eqref{eq:reg} involves an infinite sum, it converges in $L^2$ and almost surely due to \eqref{eq:summable}.

\item Interestingly, the limit in \eqref{eq:reg} is degenerate, that is, both the Gaussian and the non-Gaussian components have zero variances, if and only if $W$ is the uniform tournamenton $W_{1/2}$. In this case, the fluctuations of $N_{\triangle}(T_n)$ are of order $n^{\frac{3}{2}}$ and the limiting distribution is again Gaussian (as in \eqref{eq:Wuniform}), which is the classical result of \citet{moran1947method}. 

\end{itemize}
The proof of Theorem~\ref{thm:N} relies on the asymptotic theory of generalized $U$-statistics developed by \citet{janson1991asymptotic} (see also \cite[Chapter~11]{SJIII}). This allows us to decompose $N_{\triangle}(T_n)$, which is a generalized $U$-statistic of order 3 (see Lemma~\ref{lem:N1}), into projections onto orthogonal subspaces of 
$L^2([0,1])$ indexed by the subgraphs of the triangle (the complete graph on $3$ vertices). In particular, when $W$ is not $\triangle$-regular, the limiting distribution is governed by the first-order projections onto the individual latent variables $\eta_u$, for $1 \leq u \leq n$. In the $\triangle$-regular regime, the first-order projection is degenerate and the fluctuations are instead driven by the second-order projections: those onto pairs of latent variables $\eta_u, \eta_v$ and the edge indicator associated with the vertex pair $(u, v)$, for $1 \leq u < v \leq n$. Finally, when $W \equiv \frac{1}{2}$, both the first and second-order projections are degenerate, and the limiting distribution is determined by the third-order projection on the 2-star (that involve three latent variables $\eta_u, \eta_v, \eta_w$ the edge indicators associated with the vertex pairs $(u, v)$ and $(v, w)$, for $1 \leq u < v < w \leq n$).

\begin{remark} 
Observe that Theorem~\ref{thm:N} assumes $\zeta(W)<1$, or equivalently $t(\cc,W)>0$. This is without loss of generality, since $t(\cc,W)=0$ holds if and only if $W$ is weakly equivalent to the transitive tournamenton $W_{\mathrm{tr}}(x,y)=\bm 1\{x<y\}$ \cite[Theorem~5.4]{thornblad2018decomposition}.\footnote{Two tournamentons $W$ and $\tilde W$ are said to be weakly equivalent if there exists measure preserving maps $\psi, \phi: [0, 1] \rightarrow [0, 1]$ such that $W^\psi=\tilde W^\phi$, almost everywhere, where $W^\psi(x, y) := W(\psi(x), \psi(y))$ and $\tilde W^\phi(x, y) := \tilde W(\phi(x), \phi(y))$.} In this case, $\mathbb{E}[N_{\triangle}(T_n)]=0$ (by \eqref{eq:EN}), and hence $N_{\triangle}(T_n)=0$, almost surely. 
\label{remark:transitiveW}
\end{remark}

A special case for which Theorem~\ref{thm:N} simplifies considerably is when $W$ is degree-regular (recall Remark~\ref{remark:degree}). Note that degree-regularity implies $\triangle$-regularity (by Lemma~\ref{lem:deg_reg}), hence, the limit in \eqref{eq:irreg} is always degenerate. Furthermore, under degree-regularity, the kernel $t_W(\lp,x,y)$ is symmetric (by Lemma~\ref{lem:rp}), which implies, $\tau_W=0$ (recall \eqref{eq:tau_W}). Therefore, under degree regularity, we have the following result:

\begin{corollary} 
Suppose $W$ is a degree-regular tournamenton with $\zeta(W) \in [0, 1)$ and $T_n \sim T(n, W)$. Then the following hold: 
\begin{enumerate}

\item[$(1)$] If $W(x,\,y) \ne \frac{1}{2}$ on a set of positive measure, then 
\begin{equation*}
    \frac{N_{\triangle}(T_n)-\E[N_{\triangle}(T_n)] }{n^{2}}\dto \sum_{\lambda\in\Spec^-(\triangle_{W})}\lambda(Z_\lambda^2-1) , 
    %\label{eq:degreeregular}
\end{equation*}
where $\{Z_\lambda\}_{\lambda\in\Spec^-(\triangle_{W})}$ are i.i.d. $N(0,\,1)$. 

\item[$(2)$] If $W(x,\,y)=\frac{1}{2}$ almost everywhere, then \eqref{eq:Wuniform} holds. 

\end{enumerate}
\label{cor:regularN}
\end{corollary}

\subsection{Degeneracies of the Limiting Distribution}
\label{sec:limitregular}

We know from Theorem~\ref{thm:N} the precise conditions under which the limiting distributions in each regime are degenerate. Another natural question that arises in the \(\triangle\)-regular case is the following: under what conditions is exactly one of the two components of the limiting distribution in \eqref{eq:reg} degenerate, that is, either the Gaussian component or the non-Gaussian component is degenerate, but not both?

\begin{enumerate} 

\item {\it Degeneracy of the Gaussian component in \eqref{eq:reg}}: This happens if and only if (recall \eqref{eq:tau_W}): 
$$\tau_W^2 = \frac{1}{2} \int_{[0, 1]^2} W(x,\,y) (1 - W(x,\,y)) ( d^{\uparrow}(x) - d^{\uparrow}(y) )^2 \mathrm{d} x \mathrm{d} y  =0.$$ 
This is equivalent to the condition (recall \eqref{eq:tWrplp}):  For almost every $(x,y)\in [0, 1]$,  
\begin{align}\label{eq:varianceregular} 
\text{ either }  W(x,y) \in \{0,1\} \quad \text{ or }  \quad d_W^\uparrow(x)=d_W^\uparrow(y). 
\end{align} 
We refer to a tournamenton satisfying the above condition as $\triangle$-{\it cosymmetric}. 
Two extreme cases of the above are: (1) the kernel $t_W(\lp,x,y)$ is symmetric and (2) $W$ is $\{0, 1\}$-valued almost everywhere.  
In the first case, $W$ is degree-regular (by Lemma~\ref{lem:rp}). In the second case, $W$ is {\it random-free}, that is, once the latent uniform random variables $\{\eta_u\}_{1 \leq u \leq n}$ are chosen, the tournament is completely determined (there is no additional edge randomness).  There are also `mixed' examples, where $W$ is neither degree-regular nor random-free, but one still has $\tau_W^2 =0 $. For example, consider the 2-block tournamenton: 
$$W(x,y)=
\begin{cases}
\frac12, & (x,y)\in [0,\tfrac12) \times [0,\tfrac12) \bigcup [\tfrac12,1] \times [\tfrac12,1] ,\\
1, & (x,y)\in [0,\tfrac12) \times [\tfrac12,1] , \\ 
0, &  (x,y)\in [\tfrac12,1]  \times [0,\tfrac12) . 
\end{cases}
$$
In this case, 
$$
d_W^\uparrow(x)=
\begin{cases}
\frac34,& x\in [0,\tfrac12) ,\\ 
\frac14,& x\in [\tfrac12,1] .
\end{cases}
$$
Hence, $W$ is neither degree-regular nor random-free. Nevertheless, it is easy to check that \eqref{eq:varianceregular} still holds, and hence, $\tau_W^2=0$.

\item {\it Degeneracy of the non-Gaussian component in \eqref{eq:reg}}: This happens if and only if $\Spec^-(\triangle_{W})$ is empty. This is equivalent to the condition (recall \eqref{eq:Wp}): 
\begin{align}\label{eq:regularconstant}
\triangle_{W}(x, y)=\frac{W(x, y) t_{W}(\lp, x, y) + W(y, x) t_{W}(\lp, y, x) }{2} \text{ is constant},
\end{align}
almost everywhere on $[0, 1]^2$. We refer to a tournamenton satisfying the above condition as $\triangle$-{\it coregular}. From \eqref{eq:Wdegree} it follows that $\triangle$-coregularity implies $\triangle$-regularity. A simple example of $\triangle$-coregularity is the Condorcet tournamenton \eqref{eq:randomW}. Interestingly, it turns out that this is the only example, as shown in the following theorem: 
%that is, if a tournamenton is $\triangle$-coregular, then it must be equivalent to \eqref{eq:randomW}, for some $p \in [0, 1]$, up to measure-preserving transformations. We 
\end{enumerate}

\begin{thm} 
Suppose $W$ is $\triangle$-coregular, that is, \eqref{eq:regularconstant} holds. Then $W$ is weakly equivalent to the Condorcet tournamenton \eqref{eq:randomW}, for some $p \in [\frac{1}{2}, 1]$. 
\label{thm:W}
\end{thm}

The proof of Theorem \ref{thm:W} is given in Section \ref{sec:Wrandompf}. The result has both probabilistic and combinatorial interpretations:

\begin{itemize}

\item  Probabilistically, Theorem~\ref{thm:W} means that, in the \(\triangle\)-regular regime, the only \(W\)-random tournament models for which \(N_\triangle(T_n)\) is asymptotically Gaussian, arise from the Condorcet tournamenton \eqref{eq:randomW}, for some \(p\in(0,1)\). Note that this   includes fluctuations both in the $O(n^2)$ scale (when $p \in (0, 1)\backslash\{\frac{1}{2}\}$) as well as in the $O(n^{\frac{3}{2}})$ scale (when $p=\frac{1}{2}$). The boundary values \(p\in\{0,1\}\) are excluded, since in these cases \(N_\triangle(T_n)=0\), almost surely (recall Remark~\ref{remark:transitiveW}). 

\item Combinatorially, Theorem~\ref{thm:W} can be viewed as a quasirandom-forcing result for the Condorcet tournamenton. In other words, it shows that if \(W\) is \(\triangle\)-coregular, then \(W\) is forced to be equivalent to the tournamenton \(W_p\) in  \eqref{eq:randomW}, for some $p \in [0, 1]$. Although there is a growing body of work on quasirandom-forcing tournaments \cite{qt,localtournament,coregliano2017density,hancock2023no,kral2026sidorenko,noel2025forcing}, to the best of our knowledge, existing results primarily concern properties that force the uniform tournamenton $W_{1/2}$. Theorem~\ref{thm:W} appears to be the first result of this type that forces a non-uniform tournamenton, namely the Condorcet model. Intuitively, the theorem says that if the density of circular triads passing through every pair of continuum vertices \(x,y\in[0,1]\) is constant, then the tournamenton must be equivalent to \(W_p\). This may be viewed as a tournament analogue of the graph quasirandomness result of \cite[Theorem~4.2]{graphonK3}, which shows that if the density of triangles passing through every pair of continuum vertices in a graphon is constant, then the graphon must itself be constant (see also \cite{hladky2019limit} for consequences of this result for fluctuations of triangle counts in \(W\)-random graphs.) 
\end{itemize}

\begin{remark} 
Note that the limiting distribution in \eqref{eq:reg} is completely degenerate if and only if both the Gaussian and non-Gaussian components are degenerate. By Theorem~\ref{thm:N} (2), this occurs if and only if $W\equiv \frac12$ is the uniform tournamenton. This can also be stated as a quasirandomness-type result in the above terminology: a tournamenton $W$ with $\zeta(W)\in[0,1)$ is both $\triangle$-cosymmetric and $\triangle$-coregular if and only if $W\equiv \frac12$. The endpoint $\zeta(W)=1$ corresponds to the transitive tournamenton $W_{\mathrm{tr}}$ (recall Remark~\ref{remark:transitiveW}), which is also both $\triangle$-cosymmetric and $\triangle$-coregular. Thus, being both $\triangle$-cosymmetric and $\triangle$-coregular forces $W$ to be either uniform or transitive  (see Proposition~\ref{prop:half} for the formal statement).  
\label{remark:uniform} 
\end{remark}

\subsection{Examples} 
\label{sec:examples}

In this section we compute the limiting distribution of $N_\triangle(T_n)$ for some common choices of $W$, using Theorem \ref{thm:N}.

\begin{example}[Condorcet random model]  
Let $W=W_p$ be the Condorcet tournamenton \eqref{eq:randomW}. A direct calculation shows that 
$t^{\triangle}_{W}(x)=\frac{1}{2} p(1-p)$, for almost every $x\in[0,\,1].$ Hence, in this case, $W$ is $\triangle$-regular. Another straightforward calculation gives, 
$$t_{W}(\lp, x, y)=
\begin{cases}
(1-p)\{ (1-2p)(y-x)+p\} &\text{ if }x<y,\\
p\{(1-2p)(y-x)+(1-p)\}&\text{ if }x>y.
\end{cases} 
$$ 
Then the symmetrized 2-point conditional homomorphism density (recall \eqref{eq:Wp}) simplifies to $\triangle_{W}(x,\,y)=\frac{1}{2} p(1-p),$ almost everywhere. Hence, $\frac{1}{2} p(1-p)$ is the only nonzero eigenvalue of the operator $T_{\triangle_W}$. Thus, $\Spec^-(\triangle_W)$ is empty, and the second (non-Gaussian) term in the limiting distribution in \eqref{eq:reg} vanishes. Therefore, 
$$\frac{N_{\triangle}(T_n)-\binom{n}{3}p(1-p)}{n^2} \dto N\left(0,\, \tau_W^2 \right) , $$
where 
$$\tau_W^2=\frac{p(1-p)}{2}\int_{[0,\,1]^2}(t_{W}(\lp, x, y)-t_{W}(\lp, y, x))^2\mathrm{d}x\mathrm{d}y=\frac{p(1-p)(1-2p)^2}{12} > 0 , $$ 
whenever $p \in (0, 1) \backslash \{ \frac{1}{2} \}$. 
This recovers the classical result of \citet[Section 10.4]{janson1991asymptotic}. Note that for $p= \frac{1}{2}$, the limiting distribution is given by \eqref{eq:Wuniform}. Also, when $p = 1$ (and similarly for $p=0$), which corresponds to the transitive tournamenton $W_{\mathrm{tr}}(x, y) = \bm 1\{ x < y \}$, one has $t(\cc, W) = 0$. Hence, in this case, $\E[N_\triangle(T_n)] = 0$ (by \eqref{eq:EN}), which means that $N_\triangle(T_n) = 0$ almost surely (recall Remark \ref{remark:transitiveW}).   
\label{example:randomW}
\end{example}

Next, we consider the tournamenton in \eqref{eq:cc}, which is the empirical tournamenton associated with the directed cycle $\cc$.

\begin{example}[Rock--paper-scissors block model]  Suppose $W=W^{\cc}$ be the tournamenton in \eqref{eq:cc}. Clearly, $W$ is degree-regular, hence Corollary \ref{cor:regularN} applies. In particular, one can verify that $t^{\triangle}_{W}(x)=\frac{1}{8}$, for almost every $x\in[0,\,1]$. Furthermore, for $x, y \in[0, 1]$, 
$$t_{W}(\lp, x, y)=b_{ \lceil 3 x \rceil  \lceil 3 y \rceil }  ,  \quad \text{ where } \quad 
B = ((b_{uv}))_{1 \leq u, v \leq 3} = 
\begin{pmatrix}
\frac{1}{12}&\frac{1}{3}&\frac{1}{3}\\
\frac{1}{3}&\frac{1}{12}&\frac{1}{3}\\
\frac{1}{3}&\frac{1}{3}&\frac{1}{12}
\end{pmatrix}.$$ 
This implies, since $t_{W}(\lp, x, y)$ is symmetric, $\triangle_{W}(x,\,y)= \frac{1}{2} t_{W}(\lp, x, y)$. Further, one has $\Spec(\triangle_{W})=\{\frac{1}{8},\,-\frac{1}{24},\,-\frac{1}{24}\}$, and hence,  $\Spec^-(\triangle_{W})=\{-\frac{1}{24},\,-\frac{1}{24}\}$. Hence, by Corollary \ref{cor:regularN}, 
$$\frac{ N_{\triangle}(T_n)- \frac{1}{4} \binom{n}{3} }{n^2}\dto-\frac{1}{24}(Z_1^2-1)-\frac{1}{24}(Z_2^2-1) \stackrel{D}= \frac{1}{12}-\frac{1}{24} \chi^2_{(2)} , $$
where $Z_1,\,Z_2$ are independent $N(0, 1)$ random variables and $\chi^2_{(2)}$ denotes the chi-squared distribution with 2 degrees of freedom. 
\label{example:blockcycle}
\end{example}

Next, we consider a tournamenton with a circular-threshold structure, also known as the {\it carousel} tournamenton, which appears often in extremal combinatorics problems \cite{grzesik2023cycles}.

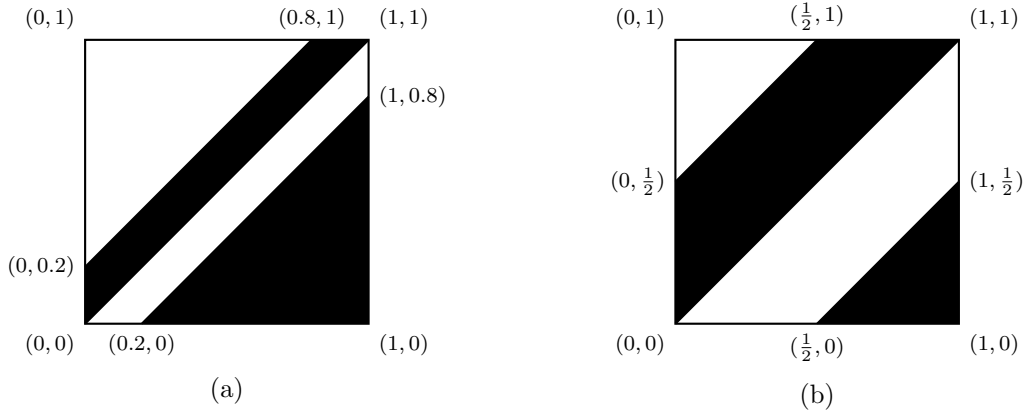
\begin{figure}[h]
%\vspace{-0.25in}
    \begin{minipage}{0.48\linewidth}
        \centering
\begin{tikzpicture}[scale=3.75]

    % black regions where W_{0.2}(x,y)=1
    \fill[black] (0,0) -- (1,1) -- (0.8,1) -- (0,0.2) -- cycle;
    \fill[black] (0.2,0) -- (1,0) -- (1,0.8) -- cycle;

    % unit square
    \draw[thick] (0,0) rectangle (1,1);

    % boundary lines between 0 and 1 regions
    \draw[thick] (0,0) -- (1,1);
    \draw[thick] (0,0.2) -- (0.8,1);
    \draw[thick] (0.2,0) -- (1,0.8);

    % labels for square vertices
    \node[below left]  at (0,0) {\scriptsize $(0,0)$};
    \node[below right] at (1,0) {\scriptsize $(1,0)$};
    \node[above left]  at (0,1) {\scriptsize $(0,1)$};
    \node[above right] at (1,1) {\scriptsize $(1,1)$};

    % labels for transition points
    \node[left]  at (0,0.2) {\scriptsize $(0,0.2)$};
    \node[below] at (0.2,0) {\scriptsize $(0.2,0)$};
    \node[above] at (0.8,1) {\scriptsize $(0.8,1)$};
    \node[right] at (1,0.8) {\scriptsize $(1,0.8)$};

\end{tikzpicture} \\ 
        \small{(a)}
    \end{minipage} 
        \begin{minipage}{0.48\linewidth}
        \centering
\begin{tikzpicture}[scale=3.75]

    % black regions where W_{0.5}(x,y)=1
    \fill[black] (0,0) -- (0,0.5) -- (0.5,1) -- (1,1) -- cycle;
    \fill[black] (0.5,0) -- (1,0) -- (1,0.5) -- cycle;

    % unit square
    \draw[thick] (0,0) rectangle (1,1);

    % boundary lines between 0 and 1 regions
    \draw[thick] (0,0) -- (1,1);
    \draw[thick] (0,0.5) -- (0.5,1);
    \draw[thick] (0.5,0) -- (1,0.5);

    % labels for square vertices
    \node[below left]  at (0,0) {\scriptsize $(0,0)$};
    \node[below right] at (1,0) {\scriptsize $(1,0)$};
    \node[above left]  at (0,1) {\scriptsize $(0,1)$};
    \node[above right] at (1,1) {\scriptsize $(1,1)$};

    % labels for transition points
    \node[left]  at (0,0.5) {\scriptsize $(0,\tfrac12)$};
    \node[below] at (0.5,0) {\scriptsize $(\tfrac12,0)$};
    \node[above] at (0.5,1) {\scriptsize $(\tfrac12,1)$};
    \node[right] at (1,0.5) {\scriptsize $(1,\tfrac12)$};

\end{tikzpicture} \\ 
        \small{(b)}
    \end{minipage}

    \caption{\small{The carousel tournamenton \eqref{eq:circularW} with (a) $\alpha=0.2$ and (b) $\alpha =0.5$. The function takes the value $1$ on the black regions and $0$ on the white regions. }}
    \label{fig:circular}
\end{figure}

\begin{example}[Carousel tournamenton]
Fix $0< \alpha <1$ and consider the tournamenton:  For $x, y \in [0, 1]$, 
\begin{align} 
W_{(\alpha)}(x,\,y)=\begin{cases}
    1&\text{if }x-y\in[-\alpha,\,0)\cup[\alpha,\,1],\\
    0&\text{if }x-y\in[-1,\,-\alpha)\cup[0,\,\alpha) . 
\end{cases}
\label{eq:circularW}
\end{align} 
Figure \ref{fig:circular} shows the function $W_{(\alpha)}$, for $\alpha=0.2$ and $\alpha=0.5$. 
By directly computing the function $t_{W_{(\alpha)}}^\triangle$ it can be shown that $W_{(\alpha)}$ is $\triangle$-regular if and only if $\alpha=\frac{1}{2}$.\footnote{In a random tournament sampled from $W_{(\alpha)}$ with $\alpha=0.5$, each player beats half of the other players in a circular structure, resembling horses arranged on a merry-go-round carousel. }

\begin{itemize}

\item $\alpha \ne \frac{1}{2}$: In this case, $W_{(\alpha)}$ is not $\triangle$-regular. Hence, applying  Theorem \ref{thm:N} gives 
\begin{equation*}
    \frac{N_{\triangle}(T_n)- \E[N_{\triangle}(T_n)]  }{n^{\frac{5}{2}}} \dto N(0,\,\sigma^2_{W_{(\alpha)}}) , 
    %\label{eq:irreg}
\end{equation*}
for some $\sigma^2_{W_{(\alpha)}} > 0$.

\item $\alpha = \frac{1}{2}$: Then $W_{(1/2)}$ is $\triangle$-regular, in fact, it is degree-regular (see Figure \ref{fig:circular} (b)). Hence, for $x, y \in [0, 1]$, 
\begin{align}
\triangle_{W_{(1/2)}}(x,\,y) =\frac{t_W(\lp, x, y)}{2} & = \frac{1}{2} \int_0^1W(y,\,z)W(z,\,x)\mathrm{d}z \nonumber \\ 
& =\frac{1}{2}\min\{|x-y|,\,1-|x-y|\} .  \nonumber 
\end{align} 
The operator associated with $\triangle_{W_{(1/2)}}(x,\,y)$ has (countably) infinitely many eigenvalues $\lambda_0 \{\lambda_{r}, \lambda_r' \}_{r \geq 0}$, with 
$$\lambda_0=\frac18,\qquad
\lambda_r=
\begin{cases}
-\dfrac{1}{2\pi^2 r^2} & \text{ if }  r \text{ is odd},\\
0 & \text{ if } r \text{ is even } \text{ and } r \neq 0 , 
\end{cases}
$$
and $\lambda_{r} = \lambda_r'$. Hence, by Corollary \ref{cor:regularN}, 
\begin{equation*}
    \frac{N_{\triangle}(T_n)-\E[N_{\triangle}(T_n)] }{n^{2}}\dto \sum_{s=0}^\infty \lambda_{2s+1} (Y_s- 2), 
\end{equation*}
where $\{Y_r\}_{r \geq 0}$ are i.i.d. $\chi^2_{(2)}$ random variables. 
\end{itemize}
\label{example:circular} 
\end{example}

\section{Tournamenton Multiplier Bootstrap}
\label{sec:estimateN}

Note that the asymptotic distribution of $N_\triangle(T_n)$ derived in Theorem~\ref{thm:N} depends on the tournamenton $W$. Hence, in order to use this result for statistical inference for $\zeta(W)$ (recall \eqref{eq:transitivityW}), we need to estimate the quantiles of the asymptotic distribution. When $W$ is not $\triangle$-regular, this can be done by consistently estimating $\sigma_W^2$ in \eqref{eq:irvariance}, since the limiting distribution of $N_\triangle(T_n)$ in this case is Gaussian (recall \eqref{eq:irreg}). However, when $W$ is $\triangle$-regular, this becomes more complicated, since the limiting distribution is non-Gaussian (recall \eqref{eq:reg}). In this section, we introduce the \textit{tournamenton multiplier bootstrap}, a method for estimating the quantiles of the limiting distribution in \eqref{eq:reg} based on the observed tournament $T_n$ itself and additional external randomness. For this, we need to define the empirical counterpart of the  2-point conditional homomorphism density defined in \eqref{eq:W_tr}.

\begin{defn}[Empirical 2-point conditional homomorphism density] 
Suppose $T_n$ is a tournament with vertex set $[n]:=\{1, 2, \ldots, n\}$ and adjacency matrix $A_{T_n} = ((a_{uv}))_{1 \leq u, v \leq n}$. Fix $1 \leq u, v \leq n$. Then empirical 2-point conditional homomorphism density for $\cc$ is defined as: 
\begin{align*}
    \hat{t}_{T_n}(\cc, u, v) & = \frac{a_{uv}}{n} \sum_{w=1}^n a_{vw} a_{wu} .
\end{align*}  
Similarly, the empirical 2-point conditional homomorphism density for $\cac$ is defined as
\begin{align*}
   \hat{t}_{T_n}(\cac, u, v) & = \frac{a_{vu}}{n} \sum_{w=1}^n a_{uw} a_{wv} = \hat{t}_{T_n}(\cc, v, u) .
\end{align*}  
Finally, define {\it symmetrized empirical $2$-point conditional homomorphism density} as: 
\begin{align*}
    \hat \triangle_{T_n}(u, v) :=\frac{ a_{uv} \sum_{w=1}^n a_{vw} a_{wu} + a_{vu} \sum_{w=1}^n a_{uw} a_{wv} }{2 n } . 
    %\label{eq:matrixtriangle}
\end{align*}  
\end{defn}

With the above definition we can now describe the multiplier bootstrap estimate of the limiting distribution in \eqref{eq:reg}. To this end, let $\hat \triangle_{T_n} = ((\hat \triangle_{T_n}(u, v)))_{1 \leq u, v \leq n}$ be the $n\times n$ symmetric matrix containing the symmetrized empirical $2$-point conditional homomorphism density for all pairs of vertices. Suppose $\hat\lambda_1\ge\hat\lambda_2\ge\ldots\ge\hat\lambda_n$ are the eigenvalues of $\frac{1}{n} \hat \triangle_{T_n}$ arranged in non-increasing order. Now, define  
\begin{equation}
    \hat X_\triangle(T_n) : =\sum_{s=2}^n \hat{\lambda}_s(J_s^2-1) , 
    \label{eq:XTn}
\end{equation}
where $\{J_s\}_{1 \leq s \leq n}$ are i.i.d. $N(0,\,1).$ Note that $\hat X_\triangle(T_n)$ depends only on the observed tournament $T_n$ and the Gaussian multipliers
$J_1, J_2, \ldots, J_n$ but not on the tournamenton $W$. In the next theorem, we show that $\hat X_\triangle(T_n)$, conditional on the observed tournament $T_n$, converges to the limiting distribution in \eqref{eq:reg}. 

\begin{thm} 
Suppose $W$ is a $\triangle$-regular tournamenton, with  $\zeta(W) \in [0, 1)$, and $T_n \sim T(n, W)$. 
Then, almost surely, as $n \rightarrow \infty$, 
\begin{align}\label{eq:TnJ}
\hat X_\triangle(T_n)|T_n \dto \tau_W \cdot Z+\sum_{\lambda\in\Spec^-(\triangle_{W})}\lambda(Z_\lambda^2-1), 
\end{align} 
where $Z, \{Z_\lambda\}_{\lambda\in\Spec^-(\triangle_{W})}$ are i.i.d. $N(0,\,1)$ and $\tau_W^2$ is as defined in \eqref{eq:tau_W}. 
\label{thm:TnJ}
\end{thm}

The proof of Theorem \ref{thm:TnJ} is given in Section \ref{sec:TnJpf}. It shows that the asymptotic distribution of $\hat X_\triangle(T_n)|T_n$ is the same as that in \eqref{eq:reg}. Hence, we can use the distribution of $\hat X_\triangle(T_n)|T_n$, which depends only on the observed tournament $T_n$, to approximate the quantiles of the limiting distribution $\bm{Z}(\mathcal{H},W)$. This allows us to construct a confidence interval for $\zeta(W)$ (recall \eqref{eq:transitivityW}), which we describe in the next section.

\section{Confidence Interval for the Consistency Coefficient } 
\label{sec:cW}

Note that although Theorem~\ref{thm:TnJ} provides a way to estimate the quantiles of the limiting distribution in \eqref{eq:reg}, it cannot be applied directly to construct a confidence interval for the consistency coefficient $\zeta(W)$ (recall \eqref{eq:transitivityW}), because it is not known a priori which of the three regimes in Theorem~\ref{thm:N} applies, since $W$ is unknown. To address this, we propose a \textit{testimation} strategy that first involves carrying out two intermediate tests for (1) whether or not $W$ is $\triangle$-regular, and (2) whether $W$ is uniform, that is, whether $W\equiv \frac{1}{2}$. Then based on the outcome of these tests we can construct the confidence interval by estimating the quantiles of the appropriate asymptotic distribution, either by plug-in estimation of the asymptotic variance (for \eqref{eq:irreg}) or by the multiplier bootstrap (for \eqref{eq:reg}). The rest of this section is organized as follows. In Section~\ref{sec:r}, we discuss the test for regularity. Then, in Section~\ref{sec:uniform}, we present the test for uniformity. In Section~\ref{sec:cWmethod}, we combine these tests with the multiplier bootstrap method described in Section \ref{sec:estimateN} to provide an algorithm for constructing the confidence interval. Simulations illustrating the empirical performance of the proposed algorithm are reported in Section~\ref{sec:simulation}.

\subsection{Testing for Regularity}
\label{sec:r}

Given a tournamenton $W$ and $T_n \sim T(n, W)$ the regularity testing problem can be formulated as: 
\begin{align}\label{eq:secondtest}
	H_{0} :W \text{ is } \triangle \text{-regular} \quad \text{ versus } \quad H_{1}: W \text{ is not } \triangle \text{-regular} . 
\end{align}
We know from Theorem \ref{thm:N} (1) that $W$ is $\triangle$-regular if and only if the asymptotic variance $\sigma_W^2 = 0$ (recall \eqref{eq:irvariance}).  Observe that  $\sigma_W^2$  can be consistently estimated from $T_n$ based on the following simple estimate: 
\begin{align}\label{eq:varianceestimate}
	\hat \sigma^2_{T_n} = t(\cccc \,,\, W^{T_n}) - t(\cc, \, W^{T_n})^2 , 
\end{align} 
where, for any fixed directed graph $H = (V(H), E(H))$,  $t(H, W^{T_n})$ is the homomorphism density of $H$ in the empirical tournamenton $W^{T_n}$ (recall \eqref{eq:Tn}). This can be expressed as  (recall \eqref{eq:tHW}):  
\begin{align*}%\label{eq:tHWTn}
    t(H, W^{T_n})=\frac{1}{n^{|V(H)|}}\sum_{\bm s\in [n]^{|V(H)|}} \prod_{(u, v)\in E(H)} a_{s_u s_v} , 
\end{align*} 
where $A_{T_n} = ((a_{st}))_{1 \leq s,t \leq n}$ is the adjacency matrix of the tournament $T_n$. 
%The following result shows that $\hat \sigma^2_{T_n}$ converges to zero at rate faster than $1/\sqrt n$, when $W$ is $\triangle$-regular. 

\begin{proposition}\label{prop:H01r} Suppose $\hat \sigma^2_{T_n}$ be as defined in 
\eqref{eq:varianceestimate}. Then the following hold: 
\begin{itemize}
 
\item[$(1)$] When $W$ is $\triangle$-regular, $n \hat \sigma^2_{T_n} = O_P(1)$. 

\item[$(2)$] When $W$ is not $\triangle$-regular, $\hat \sigma^2_{T_n} \overset{P} \rightarrow \sigma_W^2 > 0$. 

\end{itemize}
\end{proposition}

The proof of Proposition \ref{prop:H01r} is given in Section \ref{sec:H01rpf}. Now, consider the test function for the hypotheses \eqref{eq:secondtest}, 
\begin{align}\label{eq:testHW}
	\phi_n^{\mathrm{reg}} := \bm 1 \left\{ \hat \sigma^2_{T_n}> \frac{1}{\sqrt n} \right\} . 
\end{align} 
Proposition \ref{prop:H01r} implies that under $H_{0}$ as in \eqref{eq:secondtest}, $\mathbb{P}_{H_{0}}( \phi_n^{\mathrm{reg}} )\rightarrow 0$, and under $H_{1}$, $\mathbb{P}_{H_{1}}( \phi_n^{\mathrm{reg}} )\rightarrow 1$. Hence, the test \eqref{eq:testHW} is consistent for the regularity testing problem \eqref{eq:secondtest}.

\subsection{Testing for Uniformity}
\label{sec:uniform}

Given a tournamenton $W$ and $T_n \sim T(n, W)$ the uniformity testing problem can be formulated as follows: 
\begin{equation}
   H_0: W \equiv \tfrac{1}{2} \quad\text{ versus }\quad H_1: W \ne \tfrac{1}{2} .
    \label{eq:H01uniform}
\end{equation}
More precisely, we want to test the null hypothesis that $W(x, y) = \frac{1}{2}$ for almost every $x, y \in [0, 1]$ versus the alternative that $W(x, y) \ne \frac{1}{2}$ on a set of positive measure. Although uniformity of a tournamenton is a ``global'' property, it can still be tested using ``local'' features, such as the densities of fixed-size tournaments. For this, we invoke the notion of quasirandom-forcing tournaments:

\begin{defn}(\cite[Proposition 1]{hancock2023no}) 
A tournament $H = (V(H), E(H))$ is said to be \textit{quasirandom-forcing} if every tournamenton $W$ satisfying $$t(H, W) = \frac{1}{2^{|V(H)| \choose 2}}$$ is equal to $\frac{1}{2}$ almost everywhere. 
\label{defn:TW} 
\end{defn}

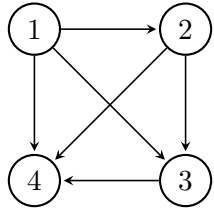
\begin{figure}
    \centering
    \begin{tikzpicture}[
            > = stealth, % arrow head style
            shorten > = 1pt, % don't touch arrow head to node
            auto,
            node distance = 10cm, % distance between nodes
            semithick % line style
        ]
        \tikzstyle{every state}=[
            draw = black,
            thick,
            fill = white,
            minimum size = 6mm
        ]
        \node[state] (1) at (0,2) {$1$};
        \node[state] (2) at (2,2) {$2$};
        \node[state] (3) at (2,0) {$3$};
        \node[state] (4) at (0,0) {$4$}; 
        \path[->] (1) edge node {} (2);
        \path[->] (1) edge node {} (3);
        \path[->] (1) edge node {} (4);
        \path[->] (2) edge node {} (3);
        \path[->] (2) edge node {} (4);
        \path[->] (3) edge node {} (4);
    \end{tikzpicture}
    \caption{\small{ $\Tr_4$: A transitive tournament on $4$ vertices which is quasirandom-forcing. }}
    \label{fig:2}
\end{figure}

It is known from results of \citet{coregliano2017density} that every transitive tournament with at least $r \ge4$ vertices is quasirandom-forcing. The smallest such tournament is a transitive tournament on $4$ vertices (see Figure \ref{fig:2}), which we will denote by $\Tr_4$. Let $t(\Tr_4,W)$ be the homomorphism density of $\Tr_4$ in $W$ (recall \eqref{eq:tHW}). A natural empirical estimate of $t(\Tr_4,W)$, based on the observed tournament $T_n$, is given by,  
\begin{align}\label{eq:trTn}
\hat t_{\Tr_4}(T_n) = \frac{1}{(n)_4} \sum_{1 \leq s_1 \ne s_2 \ne s_3 \neq s_4 \leq n} a_{s_1s_2} a_{s_1s_3} a_{s_1s_4} a_{s_2s_3} a_{s_2 s_4} a_{s_3s_4}, 
\end{align}
where $(n)_4 = n(n-1)(n-2)(n-3)$ and $A_{T_ n} = ((a_{st}))_{1 \leq s, t \leq n}$ is the adjacency matrix of $T_n$. Observe that $\hat t_{\Tr_4}(T_n)$ is the normalized count of the number of copies of $\Tr_4$ in the $W$-random tournament $T_n.$  Moreover, 
$\E [\hat t_{\Tr_4}(T_n)]= t(\Tr_4,\,W)$, where 
%$$t(\Tr_4,W)=\int_{[0,1]^4} W(x_1, x_2)W(x_1,x_3)W(x_1,x_4)W(x_2,x_3)W(x_2,x_4)W(x_3,x_4)\mathrm{d}x_1\mathrm{d}x_2\mathrm{d}x_3 \mathrm{d}x_4.$$ 
\begin{align*}%\label{eq:W4} 
t(\Tr_4,W)=\int_{[0,1]^4} \prod_{1 \leq a < b \leq 4} W(x_a, x_b) \prod_{a=1}^4 \mathrm{d} x_a .   
\end{align*}
The quasirandom-forcing property of $\Tr_4$ implies that $t(\Tr_4,\,W)=\frac{1}{64}$ if and only if $W=\frac{1}{2}$ almost everywhere on $[0,\,1]^2.$ 
%The next result establishes the concentration of $\hat t_{\Tr_4}(T_n) $, which is an unbiased estimate of $t(\Tr_4,\,W)$.  

\begin{proposition}\label{prop:Huniform}
Let $\hat t_{\Tr_4}(T_n)$ be as defined in \eqref{eq:trTn}. Then the following hold: 
\begin{itemize} 

\item[$(1)$]  For $H_0$ as in \eqref{eq:H01uniform}, $n ( \hat t_{\Tr_4}(T_n) - \frac{1}{64} ) = O_P(1)$. 
 
\item[$(2)$] For $H_1$ as in \eqref{eq:H01uniform}, $ \hat t_{\Tr_4}(T_n) \overset{P} \rightarrow t(\tr_4, W) \ne \frac{1}{64}$. 

\end{itemize}
\end{proposition}

The proof of Proposition \ref{prop:Huniform} is given in Section \ref{sec:Hunifpf}. Now, consider the test function for the hypothesis \eqref{eq:H01uniform},  
\begin{align}\label{eq:uniformtr}
	\phi_n^{\mathrm{unif}} := \bm 1 \left\{ \left| \hat t_{\Tr_4}(T_n) - \frac{1}{64} \right| > \frac{1}{\sqrt n} \right\} . 
\end{align}  
Proposition \ref{prop:Huniform} implies that under $H_{0}$ as in \eqref{eq:H01uniform}, $\mathbb{P}_{H_{0}}( \phi_n^{\mathrm{unif}} )\rightarrow 0$, and under $H_{1}$, $\mathbb{P}_{H_{1}}( \phi_n^{\mathrm{unif}} )\rightarrow 1$. Hence, the test \eqref{eq:uniformtr} is consistent for the uniformity testing problem \eqref{eq:H01uniform}.

\subsection{Constructing Confidence Intervals} 
\label{sec:cWmethod}

Using the tests for regularity and uniformity discussed above, we can now describe our algorithm for constructing the confidence interval for $\zeta(W)$ (recall \eqref{eq:transitivityW}). For this, it is convenient to use the following estimator
$$\hat\delta_n = 1-\frac{24N_\triangle(T_n)}{n(n-1)(n-2)}.$$
Note, from \eqref{eq:EN}, that $\E[\hat \zeta_n] = \zeta(W)$, hence, $\hat{\delta}_n$
is an unbiased estimate of $\zeta(W)$.\footnote{Note that $\hat{\delta}_n$ differs slightly from the classical Kendall--Smith coefficient $\zeta(T_n)$ in \eqref{eq:coefficientTn}, whose normalization depends on the parity of $n$. However, both normalizations are asymptotically equivalent. }
Now, fixing $\alpha \in (0, 1)$, our algorithm proceeds in the following steps: 
\begin{itemize}
\item[(S1)] If the regularity test $\phi^{\mathrm{reg}}$ (recall \eqref{eq:varianceestimate}) rejects $H_0$ in \eqref{eq:secondtest}, then define 
\begin{equation}
    L_n =\left[ \hat \zeta_n -   z_{\alpha/2} \frac{ 24 \hat{\sigma}_{T_n}}{\sqrt n} , \hat \zeta_n +  z_{\alpha/2} \frac{ 24 \hat{\sigma}_{T_n}}{\sqrt n}  \right],
    \label{eq:irreg_ci}
\end{equation}  
    where $ \hat{\sigma}_{T_n}$ is as defined in \eqref{eq:varianceestimate} and $z_{\alpha}$ is the $(1-\alpha)$-th quantile of standard Gaussian distribution. 
\item[(S2)] If the regularity test $\phi^{\mathrm{reg}}$ accepts $H_0$ in \eqref{eq:secondtest}, but the uniformity test $\phi^{\mathrm{unif}}$ (recall \eqref{eq:uniformtr}) rejects $H_0$ in \eqref{eq:H01uniform}, then define 
    \begin{equation}
    L_n = \left[ \hat \zeta_n +   \hat q_{1-\alpha/2, T_n} \frac{ 24 }{ n} , \hat \zeta_n +  \hat q_{\alpha/2, T_n} \frac{ 24 }{ n}  \right] , 
    \label{eq:reg_aprx_ci}
\end{equation}
where $\hat q_{\alpha, T_n}$ is the $(1-\alpha)$-th quantile of the random variable $\hat X_\triangle (T_n)|T_n$ (recall \eqref{eq:XTn}). (In implementation, $\hat q_{\alpha,T_n}$ is computed as the empirical  $(1-\alpha)$-th quantile of $\hat X_\triangle(T_n)\mid T_n$, based on $B$ bootstrap resamples of the Gaussian multipliers.) 
\item[(S3)] If the regularity test $\phi^{\mathrm{reg}}$ accepts $H_0$ in \eqref{eq:secondtest} and the uniformity test $\phi^{\mathrm{unif}}$ accepts $H_0$ in \eqref{eq:H01uniform}, then define $L_n = 0$. 
\end{itemize}

The next theorem shows that the interval $L_n$ (more precisely, the `mixture' of intervals/points) defined above has asymptotically valid coverage. The proof is given in Section \ref{sec:Lnpf}.

\begin{thm} 
Let $L_n$ be as defined above. Then, for any $W$ with $\zeta(W) \in [0, 1)$,  
$$\lim_{n \rightarrow \infty} \P( \zeta(W) \in L_n) 
= \begin{cases} 
1 - \alpha & \text{ when } W \ne \frac{1}{2} \text{ on a set of positive measure},\\
1 & \text{ when } W = \frac{1}{2} \text{ almost everywhere}.  \\
\end{cases} 
$$
\label{thm:Ln} 
\end{thm}

\begin{remark}
Observe from Theorem \ref{thm:Ln} that the asymptotic coverage of $L_n$ is $1-\alpha$ whenever $W \ne \frac{1}{2}$ on a set of positive measure. This corresponds to the intervals in \eqref{eq:irreg_ci} and \eqref{eq:reg_aprx_ci}. On the other hand, if both the regularity test $\phi^{\mathrm{reg}}$ (recall \eqref{eq:secondtest}) and the uniformity test $\phi^{\mathrm{unif}}$ (recall \eqref{eq:uniformtr}) accept their respective null hypotheses, then we know, with probability tending to one, that $W$ must be the constant function $\frac{1}{2}$, and hence, $\zeta(W)=0$. Consequently, we report the single point $L_n=0$ as our confidence ``interval'', which has asymptotic coverage equal to 1. Although this procedure is asymptotically valid, in finite samples both the regularity and uniformity tests may make errors, and reporting only a single point can easily lead to finite-sample miscoverage. For this reason, in practice it may be more stable to report an interval based on the asymptotic distribution in \eqref{eq:Wuniform} instead. Specifically, when both $\phi^{\mathrm{reg}}$ and $\phi^{\mathrm{unif}}$ accept their respective null hypotheses, we can report the following interval, based on the asymptotic distribution in \eqref{eq:Wuniform}:
\begin{equation}
L_n =
\left[
\hat \zeta_n - z_{\alpha/2} \frac{\sqrt{18}}{n^{\frac{3}{2}}},
\hat \zeta_n + z_{\alpha/2} \frac{\sqrt{18}}{n^{\frac{3}{2}}}
\right] . 
\label{eq:intervaluniform}
\end{equation}
If the above interval is used in step (S3) of the algorithm described above, then the asymptotic coverage of $L_n$ will be $1-\alpha$ for all tournamentons $W$ (including the case when $W \equiv \frac{1}{2}$).  
\end{remark}

\begin{remark}  
Observe that Propositions \ref{prop:H01r} and \ref{prop:Huniform} both show, under their respective null hypotheses, the corresponding test statistics are $O_P(1/n)$. This means that even if we replace the $1/\sqrt{n}$ rejection thresholds in \eqref{eq:testHW} and \eqref{eq:uniformtr} by $1/a_n$, for any diverging sequence $a_n$ such that $a_n/n \to 0$, the resulting tests will still be consistent for \eqref{eq:secondtest} and \eqref{eq:H01uniform}. Although any such choice is sufficient to guarantee the correct asymptotic coverage of the confidence interval, in practice, varying the thresholds leads to the following dichotomy:
\begin{itemize} 
\item For smaller thresholds, both tests begin to incur Type I errors: (1) the regularity test may identify $W$ as not $\triangle$-regular when it is, in fact, regular; and (2) the uniformity test may identify $W$ as not uniform when it is, in fact, uniform. This results in confidence intervals whose widths are of a larger order than that predicted by the asymptotic results. Such intervals typically still contain the true value of $\zeta(W)$, but will have conservative coverage. 
\item On the other hand, for larger thresholds, these tests begin to incur Type II errors. In this case, the resulting confidence intervals are narrower than they should be, which can fail to cover $\zeta(W)$ and lead to reduced coverage. 
\end{itemize}
Since undercoverage is typically more undesirable than conservative coverage, it is generally advisable in practice to use relatively smaller rejection thresholds. 
\end{remark}

\subsection{Numerical Experiments} 
\label{sec:simulation}

In this section, we illustrate the empirical performance of the proposed algorithm through simulations. To this end, we consider the following choices of $W$: 

\begin{itemize}

    \item[(1)] $W$ is the BTL tournamenton with reward function $f(x)=x,$ that is, $W(x,y)=\frac{x}{x+y}.$   In this case, $W$ is not $\triangle$-regular and $\zeta(W)\approx0.333$ (computed by numerical integration).

    \item[(2)] $W$ is the carousel tournamenton with $\alpha=0.8.$ (recall \eqref{eq:circularW}). 
    In this case, $W$ is also not $\triangle$-regular and a direct computation shows that $\zeta(W)=0.648.$
    
        \item[(3)] $W$ is the Condorcet tournamenton with parameter $p=0.3$ (recall  \eqref{eq:randomW}). In this case, $W$ is $\triangle$-regular and $\zeta(W)=1-4p(1-p) = 0.16.$
        
    \item[(4)] $W \equiv \frac{1}{2}$ is the uniform random tournamenton. 
\end{itemize}
%We consider the following 2 choices of $W$ which are not $\triangle$-regular. 
\begin{figure}
    \centering
    \includegraphics[width=0.45\linewidth]{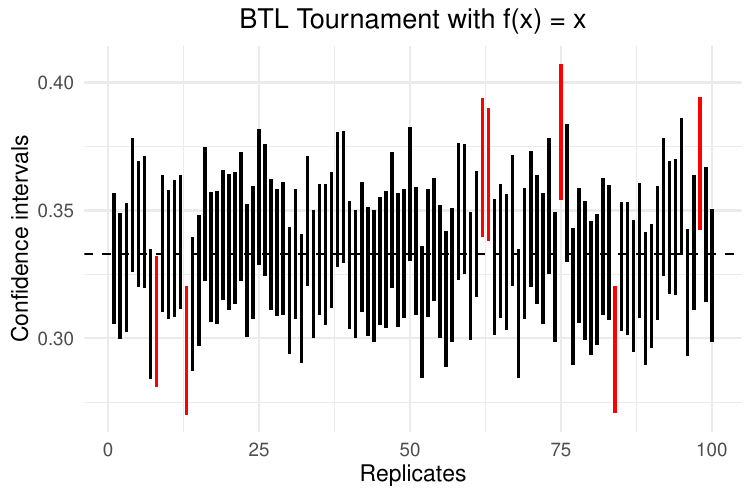} 
        \includegraphics[width=0.45\linewidth]{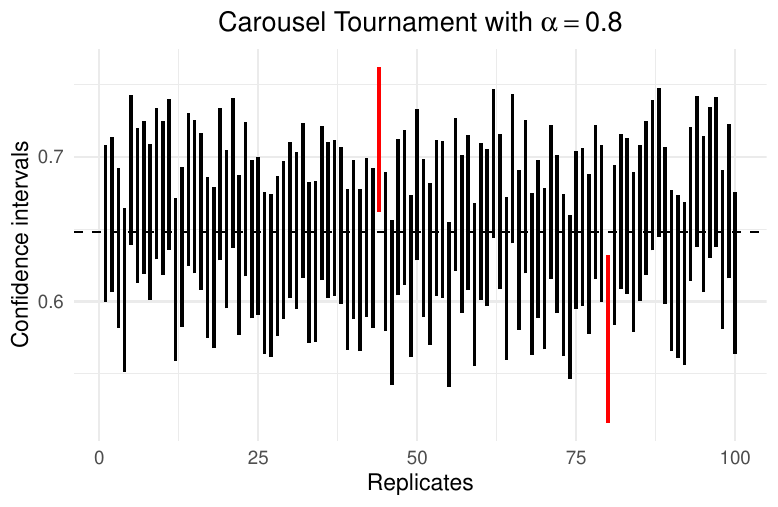} \\
          \small{\;\;\;\;\;\;\;\;(a)   \hspace{2.6in} (b)} 
    \caption{\small{Confidence intervals for (a) the BTL model with $f(x)=x$ (coverage: 93\%), and (b) 
    the carousel tournamenton with parameter $\alpha=0.8$ (coverage: 98\%). }}
    \label{fig:interval}
\end{figure} 
For each of the above choices of $W$, we construct 100 instances of $95\%$ confidence intervals for $\zeta(W)$. Each interval is computed based on a tournament $T_n \sim T(n, W)$, with $n=1000$. The nominal level is set to $0.95$. The quantiles in \eqref{eq:reg_aprx_ci} are calculated based on $B=1000$ multiplier bootstrap samples. Further, when both $\phi^{\mathrm{reg}}$ and $\phi^{\mathrm{unif}}$ accept their respective null hypotheses, we compute the interval in \eqref{eq:intervaluniform}. The results are shown in Figures \ref{fig:interval} and \ref{fig:intervalregular}.  The black dotted line represents the true value of the parameter $\zeta(W)$, and the intervals that do not contain $\zeta(W)$ are shown in red. In all cases, the empirical coverage of the intervals is very close to $0.95$, validating the asymptotic theory. It is worth recalling that the algorithm does not have any prior knowledge of whether $W$ is $\triangle$-regular or uniform. In each case, it tests for regularity and, if necessary, also for uniformity, and then constructs the interval accordingly.

\begin{figure}
    \centering
\includegraphics[width=0.45\linewidth]{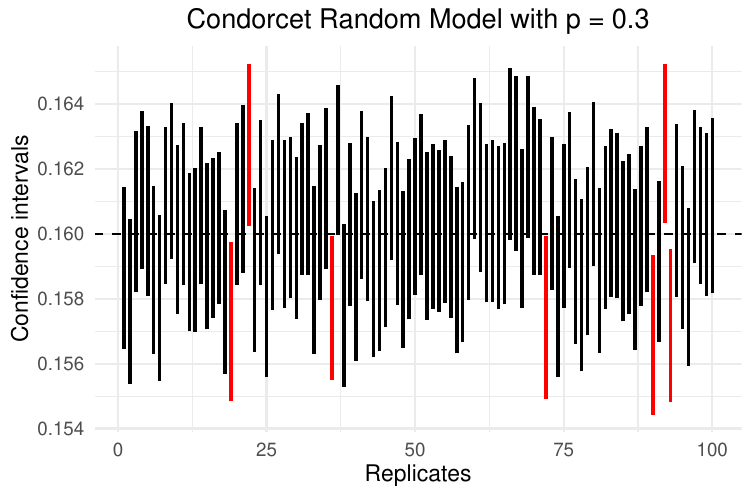}
    \includegraphics[width=0.45\linewidth]{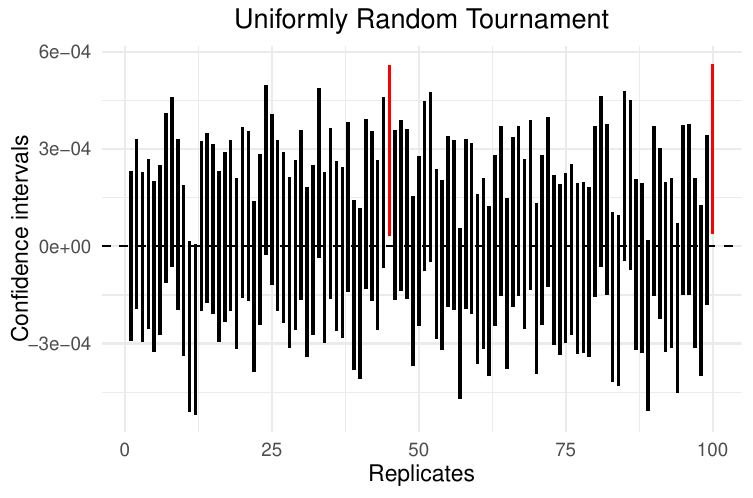} \\ 
          \small{\;\;\;\;\;\;\;\;(a)   \hspace{2.6in} (b)} 
    \caption{\small{Confidence interval for (a) Condorcet tournamenton with $p=0.3$ (coverage: 93\%), and (b) uniform random tournament $W \equiv \frac{1}{2}$ (coverage: 98\%). }} 
    \label{fig:intervalregular}
\end{figure}

\normalsize

\section{Proof of Theorem \ref{thm:N}} 
\label{sec:Npf}

\subsection{Preparations}

We begin with the following useful representation of $N_{\triangle}(T_n)$.

\begin{lem} Let $N_{\triangle}(T_n)$ be as defined in \eqref{eq:N}. Then 
\begin{align}\label{eq:NTn}
N_{\triangle}(T_n)=\sum_{1\le u<v< w\le n}Q(u,\,v,\,w) , 
%= \frac{n(n-1)(2n-1)}{12}-\frac{1}{2}\sum_{u=1}^n s_u^2, 
\end{align} 
where $Q(u,\,v,\,w)=\frac{1}{2}\left(a_{uv}a_{vw}+a_{vw}a_{wu}+a_{wu}a_{uv}+a_{vu}a_{wv}+a_{wv}a_{uw}+a_{uw}a_{vu}-1\right)$. 
\label{lem:N1}
\end{lem}

\begin{proof} 
Since $a_{uv}+a_{vu}=1$, for $1\le u<v\le n,$
\begin{align*}
    a_{uv}a_{vw}a_{wu}+a_{vu}a_{wv}a_{uw}&=a_{uv}a_{vw}a_{wu}+(1-a_{uv})(1-a_{vw})(1-a_{wu})\\
    &=1-(a_{uv}+a_{vw}+a_{wu})+(a_{uv}a_{vw}+a_{vw}a_{wu}+a_{wu}a_{uv}).
\end{align*}
Similarly,
\begin{align*}
    a_{uv}a_{vw}a_{wu}+a_{vu}a_{wv}a_{uw}=1-(a_{vu}+a_{wv}+a_{uw})+(a_{vu}a_{wv}+a_{wv}a_{uw}+a_{uw}a_{vu}).
\end{align*}
Adding the above identities and again using $a_{uv}+a_{vu}=1$ gives, 
$$a_{uv}a_{vw}a_{wu}+a_{vu}a_{wv}a_{uw}=Q(u,\,v,\,w).$$
Hence, recalling \eqref{eq:N}, the result in \eqref{eq:NTn} follows. 
\end{proof}

Now, suppose $\{\eta_u:1\le u\le n\}$ be i.i.d. $\Unif([0,\,1])$, as before, and $\{\vartheta_{uv}:1\le u<v\le n\}$ be another i.i.d. collection of $\Unif([0,\,1])$ random variables that is also independent of $\{\eta_u:1\le u\le n\}$. Define $\vartheta_{vu}=1-\vartheta_{uv}$ for $1\le u<v\le n.$ Then, from Lemma \ref{lem:N1}, we can write, 
\begin{align}\label{eq:Nuvw}
N_{\triangle}(T_n)=\sum_{1\le u<v< w\le n}f(\eta_u,\,\eta_v,\,\eta_w,\,\vartheta_{uv},\,\vartheta_{vw},\,\vartheta_{wu}) , 
\end{align} 
where
\begin{align}
    f(\eta_u,\,\eta_v,\,\eta_w, & \,\vartheta_{uv},\,\vartheta_{vw},\,\vartheta_{wu}) \nonumber \\ 
    &=\bm 1\{\vartheta_{uv}\le W(\eta_u,\,\eta_v)\}\bm 1\{\vartheta_{vw}\le W(\eta_v,\,\eta_w)\}\bm 1\{\vartheta_{wu}\le W(\eta_w,\,\eta_u)\} \nonumber \\ 
    & \hspace{0.5in} + \bm 1\{\vartheta_{vu}\le W(\eta_v,\,\eta_u)\}\bm 1\{\vartheta_{wv}\le W(\eta_w,\,\eta_v)\}\bm 1\{\vartheta_{uw}\le W(\eta_u,\,\eta_w)\}.
    \label{eq:f}
\end{align}
This is because, for every $1\le u<v\le n,$
$$a_{uv}=\bm 1\{\vartheta_{uv}\le W(\eta_u,\,\eta_v)\} \text{ and }  a_{vu}=1-a_{uv}=\bm 1\{\vartheta_{vu}\le W(\eta_v,\,\eta_u)\}.$$ 
Note also that the function $f$ is symmetric, in the sense that
$$f(\eta_u,\,\eta_v,\,\eta_w,\,\vartheta_{uv},\,\vartheta_{vw},\,\vartheta_{wu})=f(\eta_{\sigma(u)},\,\eta_{\sigma(v)},\,\eta_{\sigma(w)},\,\vartheta_{\sigma(u)\sigma(v)},\,\vartheta_{\sigma(v)\sigma(w)},\,\vartheta_{\sigma(w)\sigma(u)}) , $$
for every permutation $\sigma:\{u,\,v,\,w\}\to\{u,\,v,\,w\}.$ 
%Lemma \ref{lem:N1} ensures that $f$ is symmetric according to the notion of symmetry in \cite{janson1991asymptotic}. Now we recall some definitions and notations from \cite{janson1991asymptotic}.

\subsection{Orthogonal Decomposition of $N_{\triangle}(T_n)$}

The representation of $N_{\triangle}(T_n)$ in \eqref{eq:Nuvw} allows us to invoke the framework of generalized $U$-statistics developed in \cite{janson1991asymptotic}. We begin by recalling some notations and definitions from \cite{janson1991asymptotic}, adapted, in particular, to the case of $N_{\triangle}(T_n)$. Suppose $\{\eta_{1}, \eta_2, \eta_3\}$ and $\{\vartheta_{12}, \vartheta_{23}, \vartheta_{13}\}$ are i.i.d.\ sequences of $\mathrm{Unif}[0,1]$ random variables. Denote by $K_3$ the complete
graph on the set of vertices $\{1, 2, 3\}$ and let $G= (V(G), E(G))$
be a subgraph of $K_{3}$. Let $\mathcal{F}_{G}$ be the $\sigma$-algebra
generated by the collections $\{\eta_u\}_{u \in V(G)}$ and $\{\vartheta_{uv}\}_{(u, v) \in E(G)}$, and let $L^{2}(G)=L^2(\mathcal F_G)$  be the space of all square integrable random
variables that are functions of $\{\eta_u\}_{u \in V(G)}$ and $\{\vartheta_{uv}\}_{(u, v) \in E(G)}$. Now, consider the following subspace of $L^{2}(G)$: 
\begin{align*}
M_{G}:=\{Z\in L^{2}(G) : \mathbb{E}[ZV]=0\text{ for every }V\in L^{2}(H)\text{ such that }H\subset G\}. 
\end{align*}
(For the empty graph, $M_{\emptyset}$ is the space of all constants.) 
Equivalently,
%One important implication of the above definition is that, 
$Z\in M_{G}$ if and only if $Z\in L^{2}(G)$ and 
\begin{align*}%\label{eq:condition_MG}
\mathbb{E}\left[Z\mid \{ \eta_{u} : u \in V(H)\}, \{ \vartheta_{uv}: (u, v) \in E(H)\} \right] = 0,
\quad \text{for all } H\subset G.
\end{align*}
Then,  we have the orthogonal decomposition 
(see \cite[Lemma 1]{janson1991asymptotic}),  
\begin{align}\label{eq:LGH}
L^{2}(G)=\bigoplus_{H\subseteq G} M_H, 
\end{align}
that is, $L^{2}(G)$ is the orthogonal direct sum of $M_{H}$ for all 
subgraphs $H\subseteq G$. 
This %The above decomposition 
allows us to decompose a function $f \in L^{2}(G)$ as the sum of  its
projections onto $M_{H}$ for $H\subseteq G$.  To this end, for any closed subspace $M$ of $L^2(K_3)$, denote the orthogonal projection
onto $M$ by $P_M$.
Then, we have the  decomposition 
\begin{align*}%\label{eq:fH}
f=\sum_{H\subseteq G} f_{H},
\end{align*}
where $f_{H}=P_{M_H}f$ is the orthogonal
projection of $f$ onto $M_{H}$. Further, for $1 \leq s \leq 3$, define
\begin{align}\label{eq:fs}
    f_{(s)}:=\sum_{H\subseteq G: |V(H)|=s} f_{H} .
\end{align}
The smallest positive $d$ such that $f_{(d)}\neq 0$ is called the {\it
  principal degree} of $f$.   The asymptotic distribution of $N_{\triangle}(T_n)$
depends on the principal degree of $f$ and the subgraphs that appear in the decomposition.

For any graph $G\subseteq K_n$, the orthogonal projection %$P_{L^2(G)}$ 
onto $L^2(G)=L^2(\mathcal F_G)$ equals the conditional expectation
$\E[\cdot\mid\mathcal F_G]$, that is, 
\begin{align}\label{pl1}
  P_{L^2(G)}=\E [ \cdot\mid\mathcal F_G ] . 
\end{align}
Moreover, by \eqref{eq:LGH}, we have
\begin{align}\label{pl2}
P_{L^{2}(G)}=\sum_{H\subseteq G} P_{M_H}. 
\end{align}
The above identities allow us to express any $P_{M_H}$ as a linear combination of conditional expectations. We will compute these for different choices of $H$ in the following sections.

\subsection{Proof of Theorem \ref{thm:N} (1)}

For $1 \leq s \leq 3$, denote by $K_{\{s\}}$ the graph containing a single isolated vertex $s.$ 

\begin{lem}
Let $f$ be as defined in \eqref{eq:f}. Then 
$$f_{K_{\{1\}}}=2t_{W}^\triangle(\eta_1)-\E[f],$$
where  $t_{W}^\triangle$ is as defined in \eqref{eq:g}. 
\label{lem:eta}
\end{lem}

\begin{proof} By \eqref{pl1} and \eqref{pl2}, 
\begin{equation}
    f_{K_{\{1\}}}=P_{M(K_{\{1\}})}f=P_{L^2(K_{\{1\}})}f-P_{M(\emptyset)}f=\E[f\mid \eta_1]-\E[f].
    \label{eq:proj_1k}
\end{equation}  
Now, recalling \eqref{eq:f} observe that 
\begin{align}
    \E[f\mid \eta_1]
    %&=\E[a_{12}a_{23}a_{31}+a_{21}a_{32}a_{13}\mid \eta_1]\notag\\
    %&=\E_{\eta_2,\,\eta_3}[\E[a_{12}a_{23}a_{31}+a_{21}a_{32}a_{13}\mid \eta_1,\,\eta_2,\,\eta_3]]\notag\\
    &=\E_{\eta_2, \eta_3}[W(\eta_1,\,\eta_2)W(\eta_2,\,\eta_3)W(\eta_3,\,\eta_1)+W(\eta_2,\,\eta_1)W(\eta_3,\,\eta_2)W(\eta_1,\,\eta_3) ]\notag\\
    &=2\int_{[0,\,1]^2}W(\eta_1,\,y)W(y,\,z)W(z,\,\eta_1)\mathrm{d}y\mathrm{d}z=2t_{W}^\triangle(\eta_1) . 
    \label{eq:E_f}
\end{align}
Combining \eqref{eq:proj_1k} and \eqref{eq:E_f}  proves the result in Lemma \ref{lem:eta}. 
\end{proof}

From \eqref{eq:fs} and Lemma \ref{lem:eta}, we have that
$$f_{(1)}=2t_{W}^\triangle(\eta_1)+2t_{W}^\triangle(\eta_2)+2t_{W}^\triangle(\eta_3)-3\E[f].$$
This implies, since $\eta_1, \eta_2, \eta_3$ are i.i.d. $\mathrm{Unif}([0, 1])$, 
\begin{align*}
   \Var[f_{(1)}]= 12 \Var[t_{W}^\triangle(\eta_1)] & = 12 \left\{ \int_0^1 t_{W}^\triangle(x)^2 \mathrm d x  - \left( \int_0^1 t_{W} ^\triangle(x) \mathrm d x \right)^2 \right\}  \nonumber \\ 
   & = 12 \{ t(\cccc \,,\, W) - t( \cc, W)^2 \} , 
   \end{align*} 
   where the last step follows from \eqref{eq:ccjoin}. This shows that if $W$ is not $\triangle$-regular, that is,  $t_{W}^\triangle$ is non-constant, then $f_{(1)} \ne 0$, hence,  $f$ has principal degree $d=1$.  Then \cite[Theorem 1]{janson1991asymptotic} gives the result in \eqref{eq:irreg} with 
    \begin{align}\label{lvb2}
    \sigma_W^2 = \frac{1}{12}\text{Var}[f_{(1)}]  =  t(\cccc \,,\, W) - t( \cc, W)^2 , 
\end{align} 
as in \eqref{eq:irvariance}. 
%using \eqref{eq:variancef1} and \eqref{ttau} in the last step.
This completes the proof of Theorem \ref{thm:N} (1) when $W$ is not $\triangle$-regular.
In fact, \eqref{eq:irreg} also holds when $W$ is $\triangle$-regular, with
$f_{(1)}=0$ and $\sigma_W^2=0$. Although this case is not included in the
statement of \cite[Theorem 1]{janson1991asymptotic}, it follows from its proof, as
a consequence of \cite[Lemma 2]{janson1991asymptotic} (see also 
\cite[Corollary 11.36]{SJIII}). Consequently, the result in \eqref{eq:irreg} holds for any tournamenton $W$.

\subsection{Proof of Theorem \ref{thm:N} (2)}

We first prove \eqref{eq:reg} in Section \ref{sec:regpf}. Then in Section \ref{sec:uniformpf} we show that the degeneracy of the limiting distribution in \eqref{eq:reg} is equivalent to $W(x,\,y) = \frac{1}{2}$ almost everywhere.

\subsubsection{Proof of \eqref{eq:reg} } 
\label{sec:regpf}

In this case, $W$ is $\triangle$-regular, hence $f_{(1)}\equiv 0$.
Therefore, we consider $f_{(2)}$ (recall \eqref{eq:fs}) which can be written as 
\begin{align}\label{eq:f2}
 f_{(2)} = \sum_{1\leq s < t \leq 3} \left( f_{E_{\{s, t\}}} +f_{K_{\{s, t\}}} \right) , 
\end{align}
where $E_{\{s, t\}} = (\{s, t\}, \emptyset)$ is the graph with two vertices
$s$ and $t$ and no edge, and $K_{\{s, t\}} = (\{s, t\}, \{ (s, t) \} )$ is
the complete graph with vertices $s$ and $t$. Note that $\E[f_{(2)}]=0$, and hence, $\Var[f_{(2)}]=0$ if and only if $f_{(2)}=0$, almost everywhere.

\begin{lem}
Let $f$ be as in \eqref{eq:f}. Then the following hold: 
%The projections of the function $f$ in \eqref{eq:f} onto $M(K_{\{1\}}),\,M(E_{\{1,\,2\}}),$ and $M(K_{\{1,\,2\}})$ are as follows:
\begin{enumerate}
    \item[$(1)$] $f_{E_{\{1,\,2\}}}=2\triangle_{W}(\eta_1,\,\eta_2)-2t_{W}^\triangle(\eta_1)-2t_{W}^\triangle(\eta_2)+\E[f],$
    \item[$(2)$] $f_{K_{\{1,\,2\}}} = \bm 1\{\vartheta_{12}\le W(\eta_1,\,\eta_2)\} t_W( \lp, \eta_1, \eta_2 ) + \bm 1\{\vartheta_{21}\le W(\eta_2,\,\eta_1)\} t_W( \lp, \eta_2, \eta_1 ) - 2\triangle_{W}(\eta_1,\,\eta_2) ,$
\end{enumerate}
where the functions $t_{W}^\triangle$ and $\triangle_{W}$ are defined in \eqref{eq:g} and \eqref{eq:W_tr}, respectively.
\label{lem:proj}
\end{lem}

\begin{proof} By \eqref{pl2} and \eqref{eq:proj_1k},  
\begin{align}
    f_{E_{\{1,\,2\}}}&=P_{M(E_{\{1,\,2\}})}f=P_{L^2(E_{\{1,\,2\}})}f-P_{M(K_{\{1\}})}f-P_{M(K_{\{2\}})}f-P_{M(\emptyset)}f\notag\\
    &=\E[f\mid \eta_1,\,\eta_2]-(\E[f\mid \eta_1]-\E[f])-(\E[f\mid \eta_2]-\E[f])-\E[f]\notag\\
    &=\E[f\mid \eta_1,\,\eta_2]-\E[f\mid \eta_1]-\E[f\mid \eta_2]+\E[f] .
    \label{eq:proj_2e}
\end{align}
Note that,  
\begin{align}
    \E[f\mid \eta_1,\,\eta_2]
    %&=\E[a_{12}a_{23}a_{31}+a_{21}a_{32}a_{13}\mid \eta_1,\,\eta_2]\notag\\
    %&=\E_{\eta_3}[\E[a_{12}a_{23}a_{31}+a_{21}a_{32}a_{13}\mid \eta_1,\,\eta_2,\,\eta_3]]\notag\\
    &=\E_{\eta_3}[W(\eta_1,\,\eta_2)W(\eta_2,\,\eta_3)W(\eta_3,\,\eta_1)+W(\eta_2,\,\eta_1)W(\eta_3,\,\eta_2)W(\eta_1,\,\eta_3)]\notag\\
    &=W(\eta_1,\,\eta_2)\int_0^1W(\eta_2,\,z)W(z,\,\eta_1)\mathrm{d}z+W(\eta_2,\,\eta_1)\int_0^1W(z,\,\eta_2)W(\eta_1,\,z)\mathrm{d}z\notag\\
    &=W(\eta_1,\,\eta_2)t_W( \lp, \eta_1, \eta_2 )+W(\eta_2,\,\eta_1)t_W( \lp, \eta_2, \eta_1 )  \nonumber \\ 
    & = 2\triangle_{W}(\eta_1,\,\eta_2)  ,  
    \label{eq:E_f3}
\end{align}
where $\triangle_{W}$ is defined in \eqref{eq:W_tr}. Now, plugging \eqref{eq:E_f3} and \eqref{eq:E_f} in \eqref{eq:proj_2e}, the first part of Lemma \ref{lem:proj} follows.

For the second part, by \eqref{pl1} and \eqref{pl2}, 
\begin{align}
    f_{K_{\{1,\,2\}}}&=P_{M(K_{\{1,\,2\}})}f=P_{L^2(K_{\{1,\,2\}})}f-P_{M(E_{\{1,\,2\}})}f-P_{M(K_{\{1\}})}f-P_{M(K_{\{2\}})}f-P_{M(\emptyset)}f\notag\\
    &=P_{L^2(K_{\{1,\,2\}})}f-P_{L^2(E_{\{1,\,2\}})}f  \nonumber \\ 
    & =\E[f\mid \eta_1,\,\eta_2,\,\vartheta_{12}]-\E[f\mid \eta_1,\,\eta_2].
    \label{eq:proj_2k}
\end{align}
Now, note that 
\begin{align}
    & \E[f\mid \eta_1,\,\eta_2,\,\vartheta_{12}]  \nonumber \\ 
    %\E[a_{12}a_{23}a_{31}+a_{21}a_{32}a_{13}\mid \eta_1,\,\eta_2,\,\vartheta_{12}]\notag\\
    %&=a_{12}\E_{\eta_3}[\E[a_{23}a_{31}\mid \eta_1,\,\eta_2,\,\eta_3,\,\vartheta_{12}]]+a_{21}\E_{\eta_3}[\E[a_{32}a_{13}\mid \eta_1,\,\eta_2,\,\eta_3,\,\vartheta_{12}]]\notag\\
    &= \bm 1\{\vartheta_{12}\le W(\eta_1,\,\eta_2)\}  \E_{\eta_3}[W(\eta_2,\,\eta_3)W(\eta_3,\,\eta_1)]+ \bm 1\{\vartheta_{21}\le W(\eta_2,\,\eta_1)\} \E_{\eta_3}[W(\eta_3,\,\eta_2)W(\eta_1,\,\eta_3)]\notag\\
    & = \bm 1\{\vartheta_{12}\le W(\eta_1,\,\eta_2)\} \int_0^1W(\eta_2,\,z)W(z,\,\eta_1)\mathrm{d}z + \bm 1\{\vartheta_{21}\le W(\eta_2,\,\eta_1)\} \int_0^1W(z,\,\eta_2)W(\eta_1,\,z)\mathrm{d}z\notag\\
    &=\bm 1\{\vartheta_{12}\le W(\eta_1,\,\eta_2)\} t_W( \lp, \eta_1, \eta_2 ) + \bm 1\{\vartheta_{21}\le W(\eta_2,\,\eta_1)\} t_W( \lp, \eta_2, \eta_1 ) . 
    \label{eq:E_f2}
\end{align}
Plugging \eqref{eq:E_f2} and \eqref{eq:E_f3} in \eqref{eq:proj_2k}, the second part of Lemma \ref{lem:proj} follows.
\end{proof}

If $f_{(2)}\neq0$, then $f$ has principal degree 2, and we can apply
\cite[Theorem 2]{janson1991asymptotic},
which shows that 
\begin{equation}
    \frac{N_{\triangle}(T_n) - 2\binom{n}{3} t(\cc, W) }{n^{2}}\dto \tau_W \cdot Z+\sum_{\lambda\in\Lambda}\lambda(Z_\lambda^2-1) , 
    \label{eq:regpf}
\end{equation}
where $Z, \{Z_\lambda\}_{\lambda\in\Lambda}$ are i.i.d. $N(0, 1)$, $\Lambda$ is the multiset of (non-zero) eigenvalues of a certain integral operator $T$, and 
\begin{align}
    \tau_W^2&=\frac{1}{2}\E\left[f_{K_{\{1,\,2\}}}^2\right] \nonumber \\ 
    & =\frac{1}{2}\E\left[ \left (\E[ f\mid \eta_1,\,\eta_2,\,\vartheta_{12} ] -\E[ f\mid \eta_1,\,\eta_2] ]\right)^2 \right] \tag*{ (by Lemma \ref{lem:proj}) } \nonumber \\
    &=\frac{1}{2}\E\left[\Var[\bm 1\{\vartheta_{12}\le W(\eta_1,\,\eta_2)\} t_W( \lp, \eta_1, \eta_2 )+\bm 1\{\vartheta_{21}\le W(\eta_2,\,\eta_1)\} t_W( \lp, \eta_2, \eta_1 )\mid \eta_1,\,\eta_2] \right ] \tag*{(by \eqref{eq:E_f2})} \nonumber  \\
    &=\frac{1}{2}\E \left [\Var[\bm 1\{\vartheta_{12}\le W(\eta_1,\,\eta_2)\} (t_W( \lp, \eta_1, \eta_2 )-t_W( \lp, \eta_2, \eta_1 ))+t_W( \lp, \eta_2, \eta_1 )\mid \eta_1,\,\eta_2] \right ] \nonumber \\
    &=\frac{1}{2}\E[W(\eta_1,\,\eta_2)W(\eta_2,\,\eta_1)(t_W( \lp, \eta_1, \eta_2 )-t_W( \lp, \eta_2, \eta_1 ))^2]  ,  
    \label{eq:tWpf}
\end{align} 
as in \eqref{eq:tau_W}. Finally, we compute the Hilbert--Schmidt operator $T$ as defined in
\cite[Theorem 2]{janson1991asymptotic}. Note in our case this operator is defined on the space $M_{K_{\{1\}}}$. 
Recall that $M_{K_{\{1\}}}\subset L^2(K_{\{1\}})$, where $L^2(K_{\{1\}})$ is the space of all
square integrable random variables of the form $g(\eta_1)$.
We may identify $L^2(K_{\{1\}})$ with $L^2([0, 1])$, and then 
\eqref{eq:LGH} yields the orthogonal decomposition
\begin{align}\label{eq:decomp}
    L^{2}([0,1])=M_{K_{\{1\}}} \bigoplus M_{\emptyset}, %\bigoplus M_{K_{\{1\}}}, 
\end{align}
where $M_{\emptyset}$ is the one-dimensional space of all constants. 
Hence, $M_{K_{\{1\}}}$ is identified with the subspace of $L^2([0, 1])$ orthogonal to
constants, that is, $M_{K_{\{1\}}}=\{g\in L^2([0, 1]):\int_0^1 g(x) \mathrm dx = 0\}$. Then, taking $g,h\in M_{K_{\{1\}}}\subset L^2([0, 1])$, the
definitions given in \cite[Theorem 2]{janson1991asymptotic} yield
\begin{align}\label{eq:Tgh}
\langle Tg,h\rangle &=
 \frac{1}{2}\mathbb{E}\left[fg(\eta_{1})h(\eta_{2})\right] . 
%\nonumber \\ &
%=\frac{1}{2(|V(H)|-2)!}\mathbb{E}\left[f_{K_{\{1\}}j}g(\eta_{1})h(U_{2})\right] 
\end{align}   
%Recall the operator $T_{\triangle_W}$ defined on $L^2([0, 1])$ by \eqref{eq:TW} and \eqref{eq:WH}. 

\begin{lem}\label{lemma:operatorW} 
If $W$ is $\triangle$-regular, then  the operator $T$ on $M_{K_{\{1\}}}$ defined in \eqref{eq:Tgh}  equals the operator $T_{\triangle_{W}}$ restricted to the space
$M_{K_{\{1\}}}$. Moreover, then the multiset of non-zero eigenvalues of\/ $T$ is equal to $\mathrm{Spec}^{-}(\triangle_{W})$. 
\end{lem}

\begin{proof}
%Finally, we compute the Hilbert Schmidt operator $T$ defined in  \cite[Theorem 2]{janson1991asymptotic}. The domain of $T$ in our case simplifies to $M_{K_{\{1\}}}.$ Further, for $\phi_1,\,\phi_2\in M_{K_{\{1\}}},$ the definition in Theorem 2 of \cite{janson1991asymptotic} simplifies to
Note that we can replace $f$ by $\E[ f \mid \eta_1,\eta_2]$ in \eqref{eq:Tgh}. Hence, 
\begin{align}  
    \langle Tg,\, h\rangle&= \frac{1}{2}\E[\E[ f\mid \eta_1,\,\eta_2] g (\eta_1) h(\eta_2)] \nonumber \\
    &=\int_{[0,\,1]^2}\left(\frac{1}{2}\E[f\mid \eta_1=x,\,\eta_2=y]\right) g (x) h (y)\mathrm{d}x\mathrm{d}y \nonumber \\
    &=\int_{[0,\,1]^2}\triangle_{W}(x,\,y) g (x) h (y)\mathrm{d}x\mathrm{d}y \tag*{ (by \eqref{eq:E_f3}) } \nonumber \\ 
    & = \langle T_{\triangle_W} g,h \rangle , 
    \label{eq:TWgh}  
\end{align}
for $g,h\in M_{K_{\{1\}}}$. 
%Since $\triangle_{W}$ defined on $[0,\,1]^2$ is symmetric, it makes $T$ a symmetric Hilbert Schmidt operator. This completes the proof of Theorem \ref{thm:reg}.
% 
Furthermore, since $W$ is $\triangle$-regular, recall from the discussion after \eqref{eq:Wdegree} that $t(\cc, W)$ is an eigenvalue of $T_{\triangle_W}$ with corresponding eigenfunction $\phi\equiv \bm{1}$. Hence, 
\begin{align}
  \label{TWH1}
T_{\triangle_W} \bm{1} =t(\cc, W)=t(\cc, W) \cdot \bm{1} .  
\end{align}
Hence, $T_{\triangle_W}$ maps the space $M_\emptyset$ of constant functions into
itself. By \eqref{eq:decomp}, $M_{K_{\{1\}}}$ is the  orthogonal complement of $M_\emptyset$, and thus, since $T_{\triangle_W}$ is a symmetric operator, 
$T_{\triangle_W}$ also maps $M_{K_{\{1\}}}$ into itself.
Hence, both $T$ and $T_{\triangle_W}$ map $M_{K_{\{1\}}}$ into itself, and thus
\eqref{eq:TWgh} shows that $T=T_{\triangle_W}$ on $M_{K_{\{1\}}}$. %Finally, recall that $\Lambda $ in \eqref{beet1} is the multiset of non-zero eigenvalues of $T$, which we just have shown equals the multiset of eigenvalues of $T_{\triangle_W}$ on $M_{K_{\{1\}}}$. 
Moreover, on $M_\emptyset$, 
$T_{\triangle_W}$ has the single eigenvalue $t(\cc, W)$ by \eqref{TWH1}. 
Hence, $\Spec(T) = \Spec^-(\triangle_W)$. 
\end{proof}

Combining \eqref{eq:regpf}, \eqref{eq:tWpf}, and Lemma \ref{lemma:operatorW} completes the proof of  \eqref{eq:reg}, when $W$ is $\triangle$-regular and $f_{(2)} \ne 0$. Further, if $f_{(2)}=0$ almost everywhere, then the conclusion of \cite[Theorem 2]{janson1991asymptotic} continues to hold (with a trivial limit 0), again as a consequence of  \cite[Lemma 2]{janson1991asymptotic} (more generally, \cite[Theorem 11.35]{SJIII}.)

\subsubsection{Degeneracy of the Distribution in \eqref{eq:reg}} 
\label{sec:uniformpf}

We begin by providing a characterization when the second projection $f_{(2)}$ is degenerate  (recall \eqref{eq:f2}).

\begin{lem} 
Suppose $W$ is $\triangle$-regular. Then $f_{(2)}=0$ almost surely if and only if $\tau_W^2 = 0$ and 
\begin{align}\label{eq:coregularW}
\triangle_{W}(x, y)=\frac{W(x, y) t_{W}(\lp, x, y) + W(y, x) t_{W}(\lp, y, x) }{2} \text{ is constant} ,  
\end{align}
almost everywhere on $[0,\,1]^2.$ 
\label{lem:f2}
\end{lem}

\begin{proof} From Lemma \ref{lem:proj}, 
\begin{align*}
    & f_{E_{\{1,\,2\}}}+f_{K_{\{1,\,2\}}}  \nonumber \\ 
    &=\bm 1\{\vartheta_{12}\le W(\eta_1,\,\eta_2)\}t_W( \lp, \eta_1, \eta_2 )+a_{21}t_W( \lp, \eta_2, \eta_1 )-2t_{W}^\triangle(\eta_1)-2t_{W}^\triangle(\eta_2)+\E(f)\\
    &=\bm 1\{\vartheta_{12}\le W(\eta_1,\,\eta_2)\}t_W( \lp, \eta_1, \eta_2 )+a_{21}t_W( \lp, \eta_2, \eta_1 )-\E[f].  \nonumber   
\end{align*}
since $\triangle$-regularity of $W$ implies that $2t_{W}^\triangle(\eta_1)=2t_{W}^\triangle(\eta_2)=2\E[t_{W}^\triangle(\eta_2)]=\E[f]$, almost everywhere. Thus, recalling \eqref{eq:f2}, 
\begin{align*}
& f_{(2)} \nonumber \\ 
& =\sum_{1\le s<t\le 3}\left(f_{E_{\{s,\,t\}}}+f_{K_{\{s,\,t\}}}\right) \nonumber \\ 
& =\sum_{1\le s<t\le 3}\left(\bm 1\{\vartheta_{st}\le W(\eta_s,\,\eta_t)\} t_W(\lp, \eta_s, \eta_t)+\bm 1\{\vartheta_{ts}\le W(\eta_t,\,\eta_s)\}  t_W(\lp, \eta_t, \eta_s)-\E[f]\right). 
\end{align*} 
Since $\E[ f_{(2)} ]  = 0$,  $ f_{(2)}$ is degenerate if and only if $\Var[f_{(2)}]=0.$ To compute the variance, note that 
\begin{equation}
    \Var(f_{(2)})=\E[\Var[f_{(2)}\mid \eta_1,\,\eta_2,\,\eta_3]]+ \Var[\E[f_{(2)}\mid \eta_1,\,\eta_2,\,\eta_3]].
    \label{eq:var_f2}
\end{equation}
%For $\Var[f_{(2)}]$ to be zero, both the terms in \eqref{eq:var_f2} must vanish. 
For the first term above, 
\begin{align}\label{eq:varf2Itemp}
    & \E[\Var[f_{(2)}\mid \eta_1,\,\eta_2,\,\eta_3]] \nonumber \\ 
    %&= \E\left[\Var\left(\sum_{1\le s<t\le 3}(\bm 1\{\vartheta_{st}\le W(\eta_s,\,\eta_t)\} t_W(\lp, \eta_s, \eta_t)+\bm 1\{\vartheta_{ts}\le W(\eta_t,\,\eta_s)\} t_W(\lp, \eta_t, \eta_s)-\E[f])\mid \eta_1,\,\eta_2,\,\eta_3\right)\right]\\
    =&\E\left[\Var\left[\sum_{1\le s<t\le 3}\bm 1\{\vartheta_{st}\le W(\eta_s,\,\eta_t)\}( t_W(\lp, \eta_s, \eta_t)- t_W(\lp, \eta_t, \eta_s))\mid \eta_1,\,\eta_2,\,\eta_3\right ] \right] 
    \end{align}
    Since the terms $\{\bm 1\{\vartheta_{st}\le W(\eta_s,\,\eta_t)\}:1\le s<t\le3\}$ are conditionally independent given $\eta_1,\eta_2,\eta_3,$ continuing from \eqref{eq:varf2Itemp}, we obtain
    \begin{align}\label{eq:varf2I}
    & \E[\Var[f_{(2)}\mid \eta_1,\,\eta_2,\,\eta_3]] \nonumber \\
    =&\sum_{1\le s<t\le 3}\E[W(\eta_s,\,\eta_t)W(\eta_t,\,\eta_s)( t_W(\lp, \eta_s, \eta_t)- t_W(\lp, \eta_t, \eta_s))^2]=6\tau_W^2,
\end{align}
where $\tau_W^2$ is defined in \eqref{eq:tau_W}. The second term in \eqref{eq:var_f2}, 
\begin{align}\label{eq:varf2II}
    \Var[\E[f_{(2)}\mid \eta_1,\,\eta_2,\,\eta_3]]
    %&\Var\left[\E\left(\sum_{1\le s<t\le 3}(\bm 1\{\vartheta_{st}\le W(\eta_s,\,\eta_t)\} t_W(\lp, \eta_s, \eta_t)+\bm 1\{\vartheta_{ts}\le W(\eta_t,\,\eta_s)\} t_W(\lp, \eta_t, \eta_s)-\E[f])\mid \eta_1,\,\eta_2,\,\eta_3\right)\right]\\
    =&\Var\left[ 2 \sum_{1\le s<t\le 3}\triangle_{W}(\eta_s,\,\eta_t)-3\E[f]\right] \nonumber \\ 
    & = 4\Var[\triangle_{W}(\eta_1,\,\eta_2)+\triangle_{W}(\eta_2,\,\eta_3)+\triangle_{W}(\eta_3,\,\eta_1)].
\end{align} 
Note that \eqref{eq:varf2I} is zero equivalent to $\tau_W^2=0$, and \eqref{eq:varf2II} is zero is equivalent to $\triangle_{W}$ being a constant function almost everywhere on $[0,\,1]^2$. Combining the above with \eqref{eq:var_f2} completes the proof of Lemma \ref{lem:f2}. 
\end{proof}

Recall (from the discussion following \eqref{eq:regularconstant}) that a tournamenton $W$ satisfying \eqref{eq:coregularW} is referred to as $\triangle$-coregular. Further, note that $\tau_W^2 = 0$ is equivalent to the $\triangle$-{\it cosymmetric} condition \eqref{eq:varianceregular}. 
Therefore, the non-degeneracy result about the limiting distribution in  Theorem \ref{thm:N} (2) can be equivalently restated as follows:

\begin{proposition} 
Suppose $W$ is $\triangle$-cosymmetric and $\triangle$-coregular. Then  $W$ is  weakly equivalent  to either the transitive tournamenton $W_{\mathrm{tr}}$ or the uniform tournamenton $W_{1/2}$. 
\label{prop:half}
\end{proposition}

\noindent{\it Proof}. Let $W_p$ be the Condorcet tournamenton, as in \eqref{eq:randomW}. Since $W$ is $\triangle$-coregular, by Theorem \ref{thm:W}, there exists a measure preserving transformation $\psi: [0, 1] \rightarrow [0, 1]$ such that 
\begin{align}\label{eq:Wtransform}
W (x, y) = W_p^\psi(x, y) = p \bm 1\{\psi(x) < \psi(y) \} + (1- p) \bm 1\{\psi(x) > \psi(y) \} , 
\end{align}
for almost every $x, y \in [0, 1]$ and for some $p \in [0, 1]$. Now, there are 2 possiblities: 

\begin{itemize} 

\item $p \in \{0, 1\}$: Then $W$  is weakly equivalent to $W_{\mathrm{tr}}$.   

\item $p \in (0, 1)$: Then the $\triangle$-cosymmetric condition implies (recall \eqref{eq:varianceregular}) that 
\begin{align}\label{eq:dWsymmetric}
d_W^\uparrow(x)=d_W^\uparrow(y) ,
\end{align} 
for almost every $x, y \in [0, 1]$. From \eqref{eq:Wtransform}, we have 
$d_W^\uparrow(x)=p+(1-2p) \psi(x)$. 
Hence, \eqref{eq:dWsymmetric} can hold only when $p=\frac12$. In this case, $W$ is weakly equivalent to the uniform tournamenton. 
\hfill $\Box$
%Then $d_W^\uparrow(x) = p + (1-2p) x$. Hence, the $\triangle$-cosymmetry condition (recall \eqref{{eq:varianceregular}}) implies that either one the followitournamenton
\end{itemize}

Note that Theorem \ref{thm:N} assumes $\zeta(W) < 1$, which rules the case of the transitive tournamenton (recall Remark \ref{remark:transitiveW}). The non-degeneracy of the limiting distribution in Theorem \ref{thm:N} (2) is then an immediate consequence of Proposition \ref{prop:half}.

\subsection{Proof of Theorem \ref{thm:N} (3)}

The proof in this case follows from the classical result of \cite{moran1947method}, which was based on the method of moments. Here, for the sake of completeness, we present an alternate proof using the generalized $U$-statistic framework. To this end, observe that $W(x,\,y)=\frac{1}{2}$ implies both $f_{(1)}=0$ (by \eqref{lvb2}) and $f_{(2)}=0$ (by Lemma \ref{lem:f2}). Thus,  the principal degree of $f$ is 3, and we have to compute the projections of $f$ on the graphs of $3$ vertices. Note that there are $4$ non-isomorphic graphs with $\{1, 2, 3\}$ as the vertex set: 
\begin{itemize}

\item the empty graph with no edge (denote by $E_{\{1,\,2,\,3\}}$), 

\item the graph containing the single $(1, 2)$ and an isolated vertex 3 (denoted by $E_{\{12,\,3\}}$), 

\item the graph containing the edges $\{(1,\,2) , (2,\,3) \}$ (denoted by $E_{\{12,\,23\}}$),  

\item the complete graph with all three edges $\{ (1,\,2), (2,\,3), (1,\,3)\}$ (denoted by $K_{\{1, 2, 3\}}$). 

\end{itemize}

\begin{lem}\label{lem:projectionW} Let $f$ be as in \eqref{eq:f}. Then the following hold: 
\begin{itemize} 
\item[$(1)$] $f_{E_{\{1,\,2,\,3\}}} = f_{E_{\{12,\,3\}}} = 0$.
\item[$(2)$]  $f_{E_{\{12,\,23\}}} = \frac{ \bm 1\{\vartheta_{12}\le W(\eta_1,\,\eta_2)\} \bm 1\{\vartheta_{23}\le W(\eta_2,\,\eta_3)\} + \bm 1\{\vartheta_{21}\le W(\eta_2,\,\eta_1)\} \bm 1\{\vartheta_{32}\le W(\eta_3,\,\eta_2)\} }{2} -\frac{1}{4}$. 
\item[$(3)$] $f_{K_{\{1, 2, 3\}}}= 0$.
\end{itemize}
\end{lem}

\begin{proof} 
Note that when $W\equiv\frac{1}{2},$ then by Lemma \ref{lem:eta} and Lemma \ref{lem:proj}, we have, $f_{K_{\{s\}}}=f_{E_{\{s,\,t\}}}=f_{K_{\{s,\,t\}}}=0$, almost surely,  for every $1\le s, t\le 3$. Hence, from \eqref{pl2}, 
\begin{align*}
    f_{E_{\{1,\,2,\,3\}}} =P_{M(E_{\{1,\,2,\,3\}})}f & =P_{L^2(E_{\{1,\,2,\,3\}})}f-P_{M(\emptyset)}f-\sum_{s=1}^3P_{M(K_{\{s\}})}f-\sum_{1\le s<t\le3}P_{M(E_{\{s,\,t\}})}f\\
    &=\E[f\mid \eta_1,\,\eta_2,\,\eta_3]-\E[f]=0.
\end{align*}
Moreover,
\begin{align*}
    f_{E_{\{12,\,3\}}}&=P_{M(E_{\{12,\,3\}})}f\\
    &=P_{L^2(E_{\{12,\,3\}})}f-P_{M(\emptyset)}f-\sum_{s=1}^3P_{M(K_{\{s\}})}f-\sum_{1\le s<t\le3}P_{M(E_{\{s,\,t\}})}f-P_{M(K_{\{1,\,2\}})}f\\
    &=\E[f\mid \eta_1,\,\eta_2,\,\eta_3,\,\vartheta_{12}]-\E[f] \\ 
    & =\frac{ \bm 1\{\vartheta_{12}\le W(\eta_1,\,\eta_2)\}}{4}+\frac{ \bm 1\{\vartheta_{21}\le W(\eta_2,\,\eta_1)\}}{4}-\frac{1}{4}=0.
\end{align*}
This completes the proof of (1). 

For (2), note that 
\begin{align*}
    & f_{E_{\{12,\,23\}}} \\ 
    &=P_{M(E_{\{12,\,23\}})}f\\
    &=P_{L^2(E_{\{12,\,23\}})}f-P_{M(\emptyset)}f-\sum_{s=1}^3P_{M(K_{\{s\}})}f-\sum_{1\le s<t\le3}P_{M(E_{\{s,\,t\}})}f-P_{M(K_{\{1,\,2\}})}f-P_{M(K_{\{2,\,3\}})}f\\
    &=\E[f\mid \eta_1,\,\eta_2,\,\eta_3,\,\vartheta_{12},\,\vartheta_{23}]-\E[f] \nonumber \\ 
    & = \frac{ \bm 1\{\vartheta_{12}\le W(\eta_1,\,\eta_2)\} \bm 1\{\vartheta_{23}\le W(\eta_2,\,\eta_3)\} + \bm 1\{\vartheta_{21}\le W(\eta_2,\,\eta_1)\} \bm 1\{\vartheta_{32}\le W(\eta_3,\,\eta_2)\} }{2} -\frac{1}{4} . 
\end{align*} 
This completes the proof of (2). 

Using the result (2), recalling \eqref{eq:Nuvw}, and applying Lemma \ref{lem:N1} gives, 
\begin{align}\label{eq:f1223}
    f_{E_{\{12,\,23\}}}+f_{E_{\{23,\,31\}}}+f_{E_{\{31,\,12\}}}&=f(\eta_1,\,\eta_2,\,\eta_3, \,\vartheta_{12},\,\vartheta_{23},\,\vartheta_{31})-\frac{1}{4}  .  
\end{align}
From this and using \eqref{pl2} it follows that $f_{K_{\{1,\,2,\,3\}}} = 0$. This proves (3). 
\end{proof}

From \eqref{eq:f1223} and recalling \eqref{eq:Nuvw} it follows that 
$$\sum_{1\le u<v< w\le n} \left( f_{E_{\{uv,\,vw\}}}+f_{E_{\{vw,\,wu\}}}+f_{E_{\{wu,\,uv\}}} \right) %=\sum_{1\le u<v< w\le n}Q(u,\,v,\,w)-\binom{n}{3}\frac{1}{4}
=N_{\triangle}(T_n)-\E[N_{\triangle}(T_n)].$$
%by Lemma \ref{lem:N1}. This means the only non-zero projection of $f$ has to be $f_{E_{\{12,\,23\}}}+f_{E_{\{23,\,31\}}}+f_{E_{\{31,\,12\}}}.$ And all other projections are zero including $f_{K_{\{1,\,2,\,3\}}}.$ 
Hence, $f$ has principal degree $d=3$ and the principal subgraphs in $\Gamma_3$ are $E_{\{12,\,23\}}, E_{\{23,\,31\}}, E_{\{31,\,12\}}$, all of which are isomorphic and connected. Therefore, applying \cite[Theorem 1]{janson1991asymptotic} implies, 
\begin{align*}
\frac{N_{\triangle}(T_n)-\frac{1}{4} \binom{n}{3} }{n^{\frac{3}{2}}}\dto N\left(0,\, \kappa^2\right) , 
\end{align*} 
where 
\begin{align*}
    \kappa^2  & =\frac{1}{6}\Var[f_{E_{\{12,\,23\}}}+f_{E_{\{23,\,31\}}}+f_{E_{\{31,\,12\}}}]  \\ 
    & =\frac{1}{2}\Var[f_{E_{\{12,\,23\}}}]\\
    %&=\frac{1}{2}\Var\left[\frac{a_{12}a_{23}}{2}+\frac{a_{21}a_{32}}{2}-\frac{1}{4}\right] \\ 
    %& =\frac{1}{8}\Var(a_{12}a_{23}+a_{21}a_{32}) \\ 
    & =\frac{1}{8}\left[\E[(\bm 1\{\vartheta_{12}\le W(\eta_1,\,\eta_2)\} \bm 1\{\vartheta_{23}\le W(\eta_2,\,\eta_3)\} + \bm 1\{\vartheta_{21}\le W(\eta_2,\,\eta_1)\} \bm 1\{\vartheta_{32}\le W(\eta_3,\,\eta_2)\})^2]-\frac{1}{4}\right]\\
    %&=\frac{1}{8}\left[\E[a_{12}a_{23}+a_{21}a_{32}]-\frac{1}{4}\right] 
    &=\frac{1}{32} , 
\end{align*}
where the second equality follows from the orthogonality of the projection terms. This completes the proof.

\section{Proof of Theorem \ref{thm:W} } 
\label{sec:Wrandompf}

For $x , y \in [0, 1]$, define $J(x, y):=W(x,y)-\frac12$.  Since $W(y, x) = 1- W(x, y)$, we have $J(y, x) = - J(x, y)$ is an antisymmetric function. Define the out-degree function of $J$ as: 
$$d_J^{\uparrow}(x):=\int_0^1 J(x, y) \mathrm{d} y =d_W^{\uparrow}(x)-\frac12.$$
Now, consider the Hilbert--Schmidt operator on $L^2([0,1])$, 
\begin{align}\label{eq:Tf}
(T_J[f])(x):=\int_0^1 J(x, y)f(y) \mathrm{d} y . 
\end{align}
%From \eqref{eq:Tf}, we have $ T_J \bm{1}=d_J^{\uparrow}$.  
For all $x\in[0,1]$, define $J_{x}:[0,1]\rightarrow[-\frac{1}{2}, \frac{1}{2}]$ as
\begin{align}\label{eq:defJx}
    J_{x}(y):=J(x, \, y).
\end{align}
We regard $J_x$ as an element of $L^{2}([0,1])$. Note that this means, in particular, that $J_x=J_y$ means $J(x,z)=J(y,z)$, for almost every\ $z$. Since $J(x,y)$ is measurable and bounded, it is well known that the map $x \rightarrow J_x$ is a measurable, and (Bochner) integrable, map $[0, 1] \to L^2([0, 1])$ (by \cite[Lemma III.11.16(b)]{dunford1988linear}).
The Lebesgue differentiation theorem holds for Bochner
integrable Banach space value functions (by \cite[Section 5.V]{bochner1933integration}) 
hence, almost every $x\in[0, 1]$ is a Lebesgue point of $x \rightarrow  J_x$. We will use $\langle \cdot, \cdot \rangle$ for the inner product on $L^2([0, 1])$. We begin with expressing $\triangle_W$ in the above notation. 

\begin{lem} Let $\triangle_W$ be as defined in \eqref{eq:W_tr}. Then the following hold: 
\begin{equation}
    \triangle_W(x,y) = \frac{1}{8} - \frac{1}{2} \langle J_x, J_y\rangle + \frac12J(x, y) (d_J^{\uparrow}(y)-d_J^{\uparrow}(x) )  , 
    \label{eq:DeltaExpansion}
\end{equation}
where $\langle J_x, J_y\rangle   =\int_0^1 J(x,z)J(y,z) \mathrm{d} z$. 
\end{lem}

\begin{proof} We first compute $t_W(\lp, x,y)$ by substituting $W(y,z)=\frac12+J(y,z)$ and $ W(z,x)=\frac12+J(z,x)=\frac12-J(x,z)$ in \eqref{eq:rpxy}: 
\begin{align*}
    t_W(\lp, x,y)=& \int_0^1
    \left(\frac12+J(y,z)\right)
    \left(\frac12-J(x,z)\right) \mathrm{d} z \\
    &=\frac14
    +\frac12\int_0^1 J(y,z) \mathrm{d} z
    -\frac12\int_0^1 J(x,z) \mathrm{d} z
    -\int_0^1 J(y,z)J(x,z) \mathrm{d} z \\
    &=\frac14+\frac12 ( d_J^{\uparrow}(y)-d_J^{\uparrow}(x) ) - \langle J_x, J_y\rangle .
\end{align*}
Similarly, $t_W(\lp, y,x)=\frac14+\frac12\{d_J^{\uparrow}(x)-d_J^{\uparrow}(y)\} - \langle  J_x, J_y\rangle$. Combining the above with \eqref{eq:Wp} we obtain the result in \eqref{eq:DeltaExpansion}.  
\end{proof}

Now, since $W$ is $\triangle$-coregular, then there exists a constant $c \in [0, 1]$ such that
$\triangle_W(x,y)=c$, for almost every $(x, y) \in [0, 1]^2$. Therefore, by \eqref{eq:DeltaExpansion},
\begin{equation}
    - \langle J_x, J_y\rangle + J(x, y) (d_J^{\uparrow}(y)-d_J^{\uparrow}(x)) = 2c-\frac14 .
    \label{eq:basicEquation}
\end{equation}
Integrating the first term in \eqref{eq:basicEquation} over $[0,1]^2$ gives, 
\begin{align}
    \int_{[0,1]^2} \langle J_x, J_y\rangle ~ \mathrm{d} x \mathrm{d} y
    &=\int_{[0,1]^2}\int_0^1 J(x,z)J(y,z) \mathrm{d} z \mathrm{d} x \mathrm{d} y \nonumber \\
    &=\int_0^1
      \left(\int_0^1 J(x,z) \mathrm{d} x\right)
      \left(\int_0^1 J(y,z) \mathrm{d} y\right)
       \mathrm{d} z \nonumber \\ 
      & = \int_0^1 d_J^{\uparrow}(z)^2 \mathrm{d} z =  \| d_J^{\uparrow} \|_2^2 , 
       \label{eq:intG}
\end{align}
where the last step uses, since $J$ is antisymmetric, $\int_0^1 J(x,z) \mathrm{d} x=-\int_0^1 J(z,x) \mathrm{d} x=-d_J^{\uparrow}(z)$ and similarly $\int_0^1 J(y,z) \mathrm{d} y=-d_J^{\uparrow}(z)$. Also, integrating the second term in \eqref{eq:basicEquation} over $[0,1]^2$ gives, 
\begin{align}
    \int_{[0,1]^2}J(x, y) (d_J^{\uparrow}(y)-d_J^{\uparrow}(x)) \mathrm{d} x \mathrm{d} y  &=\int_0^1 d_J^{\uparrow}(y)\left(\int_0^1J(x, y) \mathrm{d} x\right) \mathrm{d} y
      -\int_0^1 d_J^{\uparrow}(x)\left(\int_0^1J(x, y) \mathrm{d} y\right) \mathrm{d} x \notag \\
    &=-\int_0^1 d_J^{\uparrow}(y)^2 \mathrm{d} y-\int_0^1 d_J^{\uparrow}(x)^2 \mathrm{d} x  \nonumber \\ 
    & =-2  \| d_J^{\uparrow} \|_2^2 .  
    \label{eq:intFDiff}
\end{align}  
Combining \eqref{eq:basicEquation}, \eqref{eq:intG}, and \eqref{eq:intFDiff}, gives $2c -\frac{1}{4} =  - 3 \| d_J^{\uparrow} \|_2^2 $ and, hence, 
\begin{equation}
     \langle J_x, J_y\rangle = J(x, y) (d_J^{\uparrow}(y)-d_J^{\uparrow}(x)) + 3 \| d_J^{\uparrow} \|_2^2  ,
    \label{eq:FxyJ}
\end{equation} 
for almost every $(x, y) \in [0, 1]^2$. 
%Denote $\| d_J^{\uparrow} \|_2^2  = \int_0^1d_J^{\uparrow}(x)^2 \mathrm{d} x$. 
Using this identity, in the next lemma we show that the $L^2$ norm of the function $J_x$ is constant for almost every $x \in [0, 1]$. Towards this, denote by $S_{\eqref{eq:FxyJ}}$ the set of $(x, y) \in [0, 1]$ for which \eqref{eq:FxyJ} holds.  Further, $\bar S_{\eqref{eq:FxyJ}} = \{x \in [0, 1]: (x, y) \in S_{\eqref{eq:FxyJ}} \text{ for almost every } y \in [0, 1]\}$. Since $S_{\eqref{eq:FxyJ}}$ has full measure in $[0, 1]$, by Fubini's theorem $\bar S_{\eqref{eq:FxyJ}}$ has full measure in $[0, 1]$. 

\begin{lem} \label{lm:J}
For almost every $x \in [0, 1]$, 
\begin{equation}
    \|J_x\|_2^2=  3  \| d_J^{\uparrow} \|_2^2  . 
    \label{eq:J}
\end{equation} 
\end{lem}

\begin{proof}  
Let $x \in \bar S_{\eqref{eq:FxyJ}}$ be a Lebesgue point of $x \rightarrow d_J^{\uparrow}(x)$ and also a Lebesgue point of the map $x \rightarrow  J_x$. Fix $\varepsilon > 0$ and denote by $B_\varepsilon = [x-\varepsilon, x + \varepsilon] \cap [0, 1]$. Denote by $|B_\varepsilon|$ the Lebesgue measure of the set $B_\varepsilon$. Then 
$$\frac{1}{|B_\varepsilon|} \int_{B_\varepsilon} \left \langle J_x, J_y \right \rangle \mathrm d y =   \left \langle J_x, \frac{1}{|B_\varepsilon|} \int_{B_\varepsilon} J_y \right \rangle \mathrm d y \rightarrow \|J_x\|_2^2 , $$  
as $\varepsilon \rightarrow 0$. Further, since $|J(x, y)|\le \frac{1}{2}$,
$$ \frac{1}{|B_\varepsilon|} \left| \int_{B_\varepsilon} J(x,y) (d_J^{\uparrow}(y)-d_J^{\uparrow}(x) )  \mathrm{d} y \right|  \le \frac{1}{ |B_\varepsilon| } \int_{B_\varepsilon} \left| d_J^{\uparrow}(y)-d_J^{\uparrow}(x) \right| \mathrm{d} y  \rightarrow 0 , $$
as $\varepsilon \rightarrow 0$. Combining the above with \eqref{eq:FxyJ} shows that \eqref{eq:J} holds for almost every $x \in [0, 1]$. 
\end{proof}

Note that if $\| d_J^{\uparrow} \|_2^2 = 0$, then Lemma \ref{lm:J} implies, $J = 0$ almost everywhere, hence $W\equiv \frac{1}{2}$ (which is the Condorcet tournamenton with $p = \frac{1}{2}$). Henceforth, we will assume $\| d_J^{\uparrow} \|_2^2 >0$. We now compute the moments $d_J^{\uparrow}$ and identify its distribution.

\begin{lem}\label{lm:TJr} 
Let $T_J$ be as defined in \eqref{eq:Tf} and denote $a^2 := 3  \| d_J^{\uparrow} \|_2^2 > 0 $. Fix an integer $r \geq 0$. Then, for almost every $x \in [0, 1]$, 
\begin{equation}
    T_J [(d_J^{\uparrow})^r] (x)  = 
    \begin{cases}
        \dfrac{ d_J^{\uparrow}(x)^{r+1}}{r+1} & \text{ when } r \text{ is even} , \\ 
        \dfrac{ d_J^{\uparrow}(x)^{r+1}- a^{r+1} }{r+1} & \text{ when } r \text{ is odd} . 
    \end{cases}
    \label{eq:TJr}
\end{equation}
Moreover, 
\begin{equation}
\int_0^1 d_J^{\uparrow}(x)^r \mathrm{d} x = 
\begin{cases}
        \frac{a^r}{r+1} & \text{ when } r \text{ is even},\\ 
        0 & \text{ when } r \text{ is odd} . 
    \end{cases} 
    \label{eq:momentFormula}
\end{equation}
Consequently, if $\eta \sim\Unif([0,1])$, then $d_J^{\uparrow}(\eta)\sim \Unif([-a,a])$.   
\end{lem}

\begin{proof} For $r \geq 0$, denote 
\begin{align}\label{eq:h}
h_r(x) :=  T_J [(d_J^{\uparrow})^{r}](x) = \int_0^1 J(x, y) d_J^{\uparrow}(y)^{r} \mathrm d y 
\end{align}
and $\mu_r:= \int_0^1 d_J^{\uparrow}(y)^{r} \mathrm{d} y$. Fix $x \in \bar S_{\eqref{eq:FxyJ}}$ and integrating both sides of \eqref{eq:FxyJ} with respect to the function $(d_J^{\uparrow})^{r-1}$ we get, 
\begin{align}\label{eq:dhxy}
    & \int_0^1 \langle J_x, J_y\rangle  d_J^{\uparrow}(y)^{r-1} \mathrm d y  - \int_0^1 J(x, y) (d_J^{\uparrow}(y)-d_J^{\uparrow}(x)) d_J^{\uparrow}(y)^{r-1} \mathrm d y - a^2\mu_{r-1} = 0. 
\end{align} 
Observe that 
\begin{align*}
      \int_0^1 \langle J_x, J_y\rangle  d_J^{\uparrow}(y)^{r-1} \mathrm d y   & =   \int_0^1 \left( \int_0^1 J(x,z)J(y, z)  \mathrm{d} z \right) d_J^{\uparrow}(y)^{r-1} \mathrm d y  \nonumber \\ 
    & = -  \int_0^1 \left( \int_0^1 J(x,z)J(z, y)  \mathrm{d} z \right) d_J^{\uparrow}(y)^{r-1} \mathrm d y \nonumber \\ 
    & =  - \int_0^1 J(x,z) \left( \int_0^1 J(z, y) d_J^{\uparrow}(y)^{r-1} \mathrm d y \right)   \mathrm{d} z  \nonumber \\ 
    & = -  \int_0^1 J(x,z) h_{r-1} (z) \mathrm{d} z = - T_J [h_{r-1}] (x) . 
\end{align*}  
Hence, \eqref{eq:dhxy} can be rewritten as: 
\begin{equation}
T_J [h_{r-1}] (x) + h_r (x)   -  d_J^{\uparrow} (x) h_{r-1} (x)  + a^2 \mu_{r-1} = 0 , 
    \label{eq:recurrence}
\end{equation}
for $x \in \bar S_{\eqref{eq:FxyJ}}$. 

Now, we proceed to prove \eqref{eq:TJr} and \eqref{eq:momentFormula} simultaneously by induction on $r$. 
Note that  $T_J [\bm 1] (x) = d_J^{\uparrow}(x)$. Hence, \eqref{eq:TJr} and \eqref{eq:momentFormula} hold for $r=0$. Fix $r \geq 1$, and assume that \eqref{eq:TJr} and \eqref{eq:momentFormula} hold for all non-negative integers up to $r-1$. Then we consider the following 2 cases for \eqref{eq:TJr}:

\begin{itemize} 

\item {\it $r$ is even}:  Then $r-1$ is odd and by the induction hypothesis, $h_{r-1}(x)=\frac{1}{r} (d_J^{\uparrow}(x)^r-a^r)$ almost everywhere. 
Therefore, 
$$T_J[h_{r-1}](x) = \frac{1}{r}  \int_0^1 J(x, y) (d_J^{\uparrow}(y)^r-a^r) \mathrm d y =  \frac{1}{r} \left( h_r(x) -a^r d_J^{\uparrow}(x) \right) , $$ 
almost everywhere. Further, since $r-1$ is odd, by the induction hypothesis, $\mu_{r-1}=0$. Plugging these into \eqref{eq:recurrence} gives, 
\begin{align*}
   0  & =\frac{1}{r}\left( h_r(x) -a^r d_J^{\uparrow}(x) \right)+h_r(x) - \frac{1}{r} \left( d_J^{\uparrow}(x)^{r+1}-a^r d_J^{\uparrow}(x)  \right) \nonumber \\ 
    & =\left(1+\frac1 r\right)h_r(x) -\frac{1}{r} d_J^{\uparrow}(x)^{r+1}(x) .
\end{align*}
Hence, $h_r(x) =\frac{1}{r+1} d_J^{\uparrow}(x)^{r+1}$, almost everywhere, as required. 

\item {\it $r$ is odd}:  Then $r-1$ is even and by the induction hypothesis, $h_{r-1}(x) =\frac{1}{r} d_J^{\uparrow}(x)^r$, almost everywhere. Therefore, 
$$T_J[h_{r-1}](x) = \frac{1}{r} \int_0^1 J(x, y) d_J^{\uparrow}(y)^r \mathrm d y = \frac{1}{r} h_r (x) , $$
almost everywhere. Also, since $r-1$ is even, induction hypothesis gives, $\mu_{r-1}=\frac{1}{r} a^{r-1}$. Plugging this in \eqref{eq:recurrence},
\begin{align*}
    0  &=\left(1+\frac1 r\right)h_r(x) -\frac{1}{r} \left( d_J^{\uparrow}(x)^{r+1}-a^{r+1} \right) .
\end{align*}
Thus,  $h_r(x) =\frac{1}{r+1} (d_J^{\uparrow}(x)^{r+1}-a^{r+1})$, almost everywhere, as required. 
\end{itemize}

Next, from \eqref{eq:h},  
\begin{equation}
    \int_0^1 h_r(x) \mathrm{d} x = \int_0^1 \int_0^1 J(x, y) d_J^{\uparrow}(y)^{r} \mathrm d y \mathrm d x = - \int_0^1d_J^{\uparrow}(y)^{r+1} \mathrm{d} y  =-\mu_{r+1}.
    \label{eq:intSk}
\end{equation}
Now, consider the following 2 cases for \eqref{eq:momentFormula}: 

\begin{itemize}

\item {\it $r$ is even}: Then \eqref{eq:TJr} and \eqref{eq:intSk} imply that $\frac{1}{r+1} \mu_{r+1} =-\mu_{r+1}$, that is, $\mu_{r+1}=0$. 

\item {\it $r$ is odd}: Then \eqref{eq:TJr} and \eqref{eq:intSk} give $\frac{1}{r+1} ( \mu_{r+1}-a^{r+1} )=-\mu_{r+1}$, and therefore $\mu_{r+1}=\frac{1}{r+2} a^{r+1} $. 
\end{itemize}

To complete the proof, observe that the moment sequence \eqref{eq:momentFormula} is exactly the moment sequence for the uniform distribution on $[-a,a]$. Since $d_J^{\uparrow}$ is bounded, the moments determine the distribution. Therefore,
$ d_J^{\uparrow}(\eta)\sim \Unif([-a,a])$, where $\eta\sim\Unif([0,1])$. 
\end{proof}

We now identify the function $J$ and, consequently, the tournamenton $W$. 

\begin{lem}\label{lm:q} Suppose $a> 0$. Then for almost every $(x, y) \in [0, 1]$, 
\begin{equation}
    J(x, y)=a\,\operatorname{sgn}(d_J^{\uparrow}(x)-d_J^{\uparrow}(y)) . 
    \label{eq:FIdentified}
\end{equation}
\end{lem}

\begin{proof}
First, observe that for every polynomial $q$, 
\begin{align}  \label{eq:polynomialIdentity} 
    T_J[q(d^{\uparrow}_J)](x) & = \int_0^1 J(x, y) q(d^{\uparrow}_J(y)) \mathrm d y \nonumber \\ 
    & = \frac{1}{2} \int_{-a}^a \operatorname{sgn}( d_J^{\uparrow}(x) - t)q(t) \mathrm{d} t  \nonumber \\ 
    & = \E [a \operatorname{sgn}( d_J^{\uparrow}(x) - d_J^{\uparrow}(\eta))q(d_J^{\uparrow}(\eta)) ] ,  
    \end{align}  
where $\eta \sim \Unif([0,1])$. The first equality in  \eqref{eq:polynomialIdentity} is from the definition of $T_J$ and the third equality is because $d_J^{\uparrow}(\eta)\sim \Unif([-a,a])$ (by Lemma \ref{lm:TJr}). Also, note that it suffices to verify the second equality for $q(t)=t^r$, for any integer $r \geq 0$. In this case,  
\begin{align*}
    \frac{1}{2} \int_{-a}^a \operatorname{sgn}(d_J^{\uparrow}(x) - t)t^r  \mathrm{d} t &=\frac{1}{2} \left(\int_{-a}^{d_J^{\uparrow}(x)} t^r \mathrm{d} t-\int_{d_J^{\uparrow}(x)}^a t^r \mathrm{d} t\right) =  T_J [(d_J^{\uparrow})^r] (x) , 
\end{align*} 
by recalling \eqref{eq:TJr}. Now, we extend \eqref{eq:polynomialIdentity} from polynomials to all functions in \(L^2([-a,a])\). For this, let \(\mathcal P_{\mathbb Q}\) be a countable dense subset of \(L^2([-a,a])\) consisting of
polynomials with rational coefficients. 
%(This is because, by the Weierstrass approximation theorem \cite[Chapter~7]{rudin1976principles}, together with the density of continuous functions in $L^2([-a,a])$ \cite[Proposition~7.9]{folland1999real}, the polynomials are dense in \(L^2([0,1])\). Further, since every polynomial can be uniformly approximated by polynomials with rational coefficients, the set of polynomials with rational coefficients is also dense in \(L^2([-a,a])\).
Since \eqref{eq:polynomialIdentity} holds for each \(q\in\mathcal P_{\mathbb Q}\) on a full-measure subset of $[0, 1]$, there exists a common full-measure set \(A_1 \subseteq[0,1]\) such that for all $x \in A_1$,  \eqref{eq:polynomialIdentity} holds for every \(q\in\mathcal P_{\mathbb Q}\). Finally, set $A = A_1 \cap \overline{S}_{\eqref{eq:J}}$, where $\overline{S}_{\eqref{eq:J}}$ is the full measure subset of $[0, 1]$ for which \eqref{eq:J} holds. 
Now, fix \(x\in A\) and consider the map $f\rightarrow \int_0^1 J(x,y)f(d_J^{\uparrow}(y))\,\mathrm{d}y$. This is a continuous linear functional on \(L^2([-a,a])\), because
$$\left|\int_0^1 J(x,y)f(d_J^{\uparrow}(y))\,dy\right| \le \|J_x\|_2 \left(\mathbb E[f(d_J^{\uparrow}(\eta))^2]\right)^{\frac{1}{2}}.$$
Similarly, $f\rightarrow\mathbb E[a\,\mathrm{sgn}\bigl(d_J^{\uparrow}(x)-d_J^{\uparrow}(\eta)\bigr)f(d_J^{\uparrow}(\eta)) ]$ is continuous on \(L^2([-a,a])\). Since the two functionals agree on the dense set
\(\mathcal P_{\mathbb Q}\), they agree on all of \(L^2([-a,a])\). Therefore, for every
\(f\in L^2([-a,a])\),
\begin{align*}%\label{eq:fJ} 
\mathbb E\!\left[J(x,\eta)f(d_J^{\uparrow}(\eta))\right] = \mathbb E\!\left[ a\,\mathrm{sgn}\bigl(d_J^{\uparrow}(x)-d_J^{\uparrow}(\eta)\bigr)f(d_J^{\uparrow}(\eta)) \right] ,  
    \end{align*}  
for all $f \in L^2([-a,a])$. This implies, by the definition of conditional expectation, for $x \in A$, 
\begin{equation}
    \mathbb E\left[J(x, \eta)\mid d_J^{\uparrow}(\eta)\right]
    =a\,\operatorname{sgn}(d_J^{\uparrow}(x)-d_J^{\uparrow}(\eta)),
    \label{eq:conditionalMean}
\end{equation}
where $\eta \sim \Unif([0,1])$. By \eqref{eq:J}, $\mathbb E[J(x, \eta)^2]=a^2$, for $x \in A$. Also, by \eqref{eq:conditionalMean},    
$$\mathbb E\left[ \left(\mathbb E[J(x, \eta)\mid d_J^{\uparrow}(\eta)]\right)^2 \right]    =a^2.$$ Therefore,
$$\mathbb E\!\left[ \Var\bigl[J(x,\eta)\mid d_J^{\uparrow}(\eta)\bigr] 
\right] = \mathbb E[J(x,\eta)^2] - \mathbb E\!\left[ \mathbb E[J(x,\eta)\mid d_J^{\uparrow}(\eta)]^2 \right] =0.$$
Thus, $J(x, y) = a\,\mathrm{sgn} (d_J^{\uparrow}(x)-d_J^{\uparrow}(y))$
almost surely, for every \(x\in A\). Since \(A\) has full measure, by Fubini's theorem the result in \eqref{eq:FIdentified} follows.  
\end{proof}

Note that, since $|J(x,y)|\le \frac{1}{2}$, Lemma~\ref{lm:J} gives $a\le \frac{1}{2}$. Now, define 
$$\phi(x):=\frac{1}{2} \left( 1 - \frac{d_J^{\uparrow}(x)}{a} \right)  .$$ Observe that 
$\phi(\eta)\sim \Unif([0,1])$, since $d_J^\uparrow(\eta)\sim \Unif([-a,a])$, by Lemma \ref{lm:TJr}. Hence, $\phi$ is measure-preserving. Further, 
$$d_J^{\uparrow}(x)>d_J^{\uparrow}(y)  \quad \Longleftrightarrow \quad \phi(x)<\phi(y).$$
Also, since $d_J^\uparrow(\eta)$ has a continuous distribution, the set $\{(x,y):\phi(x)=\phi(y)\}$ has measure zero. Thus, setting $p:=\frac{1}{2}+a\in[\frac{1}{2},1]$ and applying  \eqref{eq:FIdentified} gives,  
$$W(x,y) =
\frac{1}{2}+J(x,y)
=
\begin{cases}
p, & \phi(x)<\phi(y),\\
1-p, & \phi(x)>\phi(y),
\end{cases}
$$
for almost every $(x,y) \in [0, 1]^2$. Therefore, $W$ is equivalent, up to the measure-preserving transformation $\phi$, to the Condorcet tournamenton $W_p$ in \eqref{eq:randomW}. \hfill $\Box$

\section{Proof of Theorem \ref{thm:TnJ}}
\label{sec:TnJpf}

To begin with, suppose $\{\lambda_s\}_{s \geq 1}$ is an enumeration of $\Spec^-(\triangle_{W})$. Then denote the random variable in the RHS of \eqref{eq:TnJ} as:  
\begin{align}\label{eq:TnJpf}
X := \tau_W \cdot Z+\sum_{s=1}^\infty \lambda_s (Z_s^2-1), 
\end{align}
where $Z, \{Z_s\}_{s \geq 1}$ are i.i.d. $N(0,\,1)$ and $\tau_W^2$ is as defined in \eqref{eq:tau_W}. To prove the result in Theorem \ref{thm:TnJ} it suffices to show that the moment generating function of $\hat X_{\triangle}(T_n)|T_n$ converges to the moment generating function  of $X$ in a neighborhood of zero. Towards this, we first compute the moment generating function of $X$ in the following lemma. Throughout, for $r \geq 3$, the undirected cycle of length $r$ will be denoted by $C_r$. With slight abuse of notation, we denote the homomorphism density of $C_r$ in the graphon $\triangle_W$ (defined in \eqref{eq:Wp}) as: 
$$t(C_r,\,\triangle_{W})=\int_{[0,\,1]^r}\triangle_{W}(x_1,\,x_2)\triangle_{W}(x_2,\,x_3)\ldots \triangle_{W}(x_{r-1},\,x_r)\triangle_{W}(x_r,\,x_1)\prod_{s=1}^r\mathrm{d}x_s.$$
for $r\ge 3.$ Also, $\|\triangle_{W}\|_2 := (\int_{[0, 1]^2} \triangle_{W}(x_1,\,x_2) \mathrm d x_1 \mathrm d x_2)^{\frac{1}{2}}$ denotes the $L^2$ norm of $\triangle_{W}$.

\begin{lem}
Suppose $W$ is a $\triangle$-regular tournamenton and $X$ be as defined in \eqref{eq:TnJpf}. Denote by $M_X(t) := \E[e^{t X}]$ the moment generating function of $X$. Then, for all $|t|<\frac{1}{4}$, 
$$\log M_X(t)=\left(\frac{\tau_W^2}{2} + \| \triangle_{W} \|_2^2 -\gamma_W^2\right)t^2+\frac{1}{2}\sum_{r=3}^\infty\left(\frac{2^{r}(t(C_r,\,\triangle_{W}) - \gamma_W^{r} )}{r}\right)t^r ,$$ 
\label{lem:limit_MGF}
where $\tau_W^2$ is as defined in \eqref{eq:tau_W} and $\gamma_W =t(\cc, W) $. \end{lem}

\begin{proof} 
%Let $\{\lambda_s\}_{s \geq 1}$ be an enumeration of $\Spec^-(\triangle_{W}).$ 
Fix an integer $K \geq 1$ and define
$$X_K= \tau_W Z+\sum_{s=1}^K \lambda_s(Z_s^2-1), $$ 
where $Z, \{Z_s\}_{s \geq 1}$ are i.i.d. $N(0, 1)$. Then, because of  \eqref{eq:summable}, $X_K$ converges in $L^2$ to the random variable $X$, defined in \eqref{eq:TnJpf}, as $K \rightarrow \infty$. Hence, 
\begin{align}\label{eq:exptXK} 
e^{t X_K} \overset{P} \rightarrow e^{t X} , \quad \text{ for all } t \in \R,  
\end{align}
as $K \rightarrow \infty$. Furthermore, the moment generating function of $X_K$ exists for every $|t|<\frac{1}{2}$ and is given by, 
\begin{align*}
    \E[e^{tX_K} ] = \E\left[ e^{ t \tau_W Z+ t \sum_{s=1}^K \lambda_s (Z_s ^2-1)} \right] 
    %&=\exp\left(-t\sum_{s=1}^K \lambda_s\right)\E[\exp(tZ)]\prod_{s=1}^K\E\left[\exp(\lambda_s tZ_i^2)\right]\\
    &= e^{ \frac{\tau_W^2t^2}{2}} \prod_{s=1}^K e^{-t \lambda_s} \left( \frac{1}{1-2\lambda_s t} \right)^{\frac{1}{2}} , 
\end{align*} 
using the independence of $Z$ and the collection $\{Z_s\}_{s \geq 1}$. Now, using the bound $\log(1-x)+x\ge-x^2$, for $|x|< \frac{1}{2}$ gives the following. For $|t| < \frac{1}{4}$, 
\begin{align}\label{eq:exponentialXK} 
\log \E[ e^{tX_K} ] & =\frac{\tau_W^2t^2}{2} - \sum_{s=1}^K \left[\frac{\log(1-2\lambda_s t)+2\lambda_s t}{2}\right]  \nonumber \\ 
& \le\frac{\tau_W^2t^2}{2}+ 2 \sum_{s=1}^K \lambda_s^2t^2 \le\frac{\tau_W^2t^2}{2}+ 2 t^2 < \infty , 
\end{align} 
since $|2\lambda_s t|< \frac{1}{2}$, as $|t| < \frac{1}{4}$ and $|\lambda_s|\le1$ (recall  \eqref{eq:summable}).  Given $|t| < \frac{1}{4}$, choose $\varepsilon > 0$ such that $(1+\varepsilon) |t| < \frac{1}{4}$. Then, \eqref{eq:exponentialXK} implies, $\E[ e^{(1+\varepsilon) tX_K} ]  < \infty$. Hence, $\{e^{ tX_K}  : K \geq 1\}$ is uniformly integrable. Combining this with \eqref{eq:exptXK} gives, for $|t| < \frac{1}{4}$,
$$M_X(t) = \E[e^{t X}] = \lim_{K \rightarrow \infty} \E[e^{t X_K}] = e^{ \frac{\tau_W^2t^2}{2}} \prod_{s=1}^\infty  e^{-t \lambda_s}  \left( \frac{1}{1-2\lambda_s t} \right)^{\frac{1}{2}}.$$
Hence, for $|t| < \frac{1}{4}$ and by Fubini's theorem gives, 
\begin{align}
    \log M_X(t) &=\frac{\tau_W^2t^2}{2}+\frac{1}{2}\sum_{r=2}^\infty\frac{2^rt^r}{r}\sum_{s=1}^\infty\lambda_s ^r\notag\\
    &=\left(\frac{\tau_W^2}{2}+\sum_{s=1}^\infty\lambda_s ^2\right)t^2+\frac{1}{2}\sum_{r=3}^\infty\frac{2^rt^r}{r}\sum_{s=1}^\infty\lambda_s ^r . 
    \label{eq:lim_MGF_1}
\end{align} 
Now, if $W$ is $\triangle$-regular, then $t^{\triangle}_{W}(x) = t(\cc, W) = \gamma_W$, almost everywhere on $[0,\,1]$. Then $\gamma_W$ is an eigenvalue of the operator $T_{\triangle_W}$ derived from the kernel $\triangle_{W}$, defined in \eqref{eq:W_tr}. Since $\{\lambda_s \}_{s \geq 1}$ is an enumeration of $\Spec^-(\triangle_{W})$, which is the multiset of eigenvalues of $\triangle_{W}$ after decreasing the multiplicity of $\gamma_W$ by 1, we have for any integer $r \geq 3$, 
\begin{align}
    \gamma_W^r+\sum_{s=1}^\infty\lambda_s ^r=\sum_{\lambda\in\Spec(\triangle_{W})}\lambda^r = t(C_r,\,\triangle_{W})  ,   
    \label{eq:cycle}
\end{align}
where the second equality follows from Equation (7.22) in \cite{lovasz_book}. Also, by \eqref{eq:summable}, 
\begin{align}
    \gamma_W^2+\sum_{s=1}^\infty\lambda_s ^2&=\sum_{\lambda\in\Spec(\triangle_{W})}\lambda^2 = \| \triangle_W \|_2^2 .     
   \label{eq:doubleedge}
\end{align}
Combining \eqref{eq:lim_MGF_1}, \eqref{eq:cycle}, and \eqref{eq:doubleedge}, the result in Lemma \ref{lem:limit_MGF} follows. 
\end{proof}

Now, recall the definition of the random variable $\hat X_{\triangle}(T_n)$ from \eqref{eq:XTn}. Denote by $$M_{\hat X_{\triangle}(T_n)}(t) = \E[e^{ t \hat X_{\triangle}(T_n)}|T_n],$$ 
the moment generating function of $\hat X_{\triangle}(T_n)$ given the tournament $T_n$ (that is, the expectation is taken over the randomness of Gaussian multipliers $\{J_s\}_{1 \leq s \leq n}$). To complete the proof of Theorem \ref{thm:TnJ} it suffices to show the  following:

\begin{lem} 
For any $|t| < \frac{1}{4}$, 
$$M_{\hat X_{\triangle}(T_n)}(t) \rightarrow M_{X}(t), $$
almost surely. 
\label{lem:emp_MGF}
\end{lem}

\begin{proof}  Recalling \eqref{eq:XTn} and using the fact that the collection $\{J_s\}_{1 \leq s \leq n}$ are i.i.d. $N(0,\,1)$ gives, for $|t| < \frac{1}{4}$, 
\begin{align}\label{eq:MTn}
    M_{\hat X_{\triangle}(T_n)}(t) = \E[ e^{ t \hat X_\triangle(T_n)} |  T_n ] &=\E\left[ e^{ \sum_{s=2}^n \hat \lambda_s (J_s^2-1)t} \mid T_n\right] = e^{ - t\sum_{s=2}^n \hat \lambda_s } \prod_{s=2}^n \left( \frac{1}{1-2 \hat \lambda_s t} \right)^{\frac{1}{2}} . 
\end{align} 
Note that 
\begin{align*}
\sum_{s=2}^n \hat \lambda_s^2 \leq \sum_{s=1}^n \hat \lambda_s^2 = \frac{1}{n^2} \sum_{1 \leq u , v \leq n} \hat \triangle_{T_n}(u, v)^2 \leq \frac{1}{4} , 
\end{align*}
since $\hat \triangle_{T_n}(u, v) \leq \frac{1}{2}$, for all $1 \leq u, v \leq n$. 
This means $|\hat \lambda_s|< \frac{1}{2}$, for all $s \geq 1$, and the RHS of \eqref{eq:MTn} is well-defined for all $|t| \leq \frac{1}{4}$. Now, by Fubini's theorem and arguments as in \eqref{eq:lim_MGF_1} it follows that 
\begin{align}\label{eq:MTestimate} 
\log M_{\hat X_{\triangle}(T_n)}(t)= t^2 \sum_{s=2}^n \hat \lambda_s^2  +\frac{1}{2} \sum_{r=3}^\infty\left(\frac{2^r \sum_{s=2}^n \hat \lambda_s^r }{r}\right)t^r . 
\end{align}
for $|t| \leq \frac{1}{4}$. Now, define the empirical graphon associated with the matrix $\hat \triangle_{T_n}$ as follows: For $x, y \in [0, 1]$ 
\begin{align}\label{eq:triangleWxy}
\bar{\triangle}_{T_n}(x, y) = \hat \triangle_{T_n}( \lceil n x \rceil, \lceil n y \rceil ). 
\end{align}
We show in Lemma \ref{lm:Wtriangleconvergence} that $\bar \triangle_{T_n}$ converges to $\triangle_W$ in the cut-metric, as $n \rightarrow \infty$. Consequently, from \cite[Chapter 11]{lovasz_book} we have the following, as $n \rightarrow \infty$, almost surely: 
\begin{itemize} 
\item For $r \geq 3$, $\sum_{s=1}^n \hat \lambda_s^r = t(C_r, \bar \triangle_{T_n}) \rightarrow t(C_r, \triangle_{W})$. 
\item $\hat \lambda_1 \rightarrow \lambda_{\max}(\triangle_{W})$, where $\lambda_{\max}(\triangle_{W})$ is the largest eigenvalue of the operator $\triangle_{W}$. 
\end{itemize} 
Recall from the discussion after \eqref{eq:Wdegree} that, if $W$ is $\triangle$-regular, then $\lambda_{\max}(\triangle_{W})= t(\cc, W)$. Hence, for any fixed $K \geq 3$, applying the above results gives, 
\begin{align}\label{eq:MTK}
\lim_{n \rightarrow \infty}\sum_{r=3}^{K} \left(\frac{2^r \sum_{s=2}^n \hat \lambda_s^r }{r}\right)t^r = \sum_{r=3}^K \left(\frac{2^r (t(C_r, \triangle_{W} ) - \gamma_W^r)  }{r}\right)t^r , 
\end{align}
where $\gamma_W = t(\cc, W)$ as in Lemma \ref{lem:limit_MGF}. Furthermore, for $|t|< \frac{1}{4},$
$$\left| \sum_{r=3}^\infty \left(\frac{2^r (t(C_r, \triangle_{W} ) - \gamma_W^r)  }{r}\right)t^r \right| \le \sum_{r=3}^\infty\frac{1}{2^rr}<\infty. $$
Hence, from \eqref{eq:MTK}, for $|t|< \frac{1}{4},$
\begin{align}\label{eq:MTnK}
\lim_{K \rightarrow \infty} \lim_{n \rightarrow \infty} \sum_{r=3}^{K} \left(\frac{2^r \sum_{s=2}^n \hat \lambda_s^r }{r}\right)t^r = \sum_{r=3}^\infty \left(\frac{2^r (t(C_r, \triangle_{W} ) - \gamma_W^r)  }{r}\right)t^r . 
\end{align} 
Also, for $|t|<\frac{1}{4}$, 
\begin{align}\label{eq:MT}
\lim_{K \rightarrow \infty} \sup_{n \geq 1} \left| \sum_{r=K+1}^\infty \left(\frac{2^r \sum_{s=2}^n \hat \lambda_s^r }{r}\right)t^r \right| \le \lim_{K \rightarrow \infty} \sum_{r=K+1}^\infty\frac{1}{2^rr}  = 0 . 
\end{align} 
Combining \eqref{eq:MTnK} and \eqref{eq:MT} gives, the limit of the second term in 
\eqref{eq:MTestimate}, 
\begin{align}\label{eq:MTestimateII}
\lim_{n \rightarrow \infty}  \frac{1}{2} \sum_{r=3}^{\infty} \left(\frac{2^r \sum_{s=2}^n \hat \lambda_s^r }{r}\right)t^r =  \frac{1}{2} \sum_{r=3}^\infty \left(\frac{2^r (t(C_r, \triangle_{W} ) - \gamma_W^r)  }{r}\right)t^r . 
\end{align} 

Next, we consider the first term in \eqref{eq:MTestimate}. Towards this, note that 
\begin{align*} 
\sum_{s=1}^n \hat \lambda_s^2 & = \frac{1}{n^2} \sum_{1 \leq u , v \leq n} \hat \triangle_{T_n}(u, v)^2  \nonumber \\ 
& =  \frac{1}{n^2} \sum_{1 \leq u , v \leq n} \left[ \frac{a_{uv}}{2n}\sum_{w=1}^na_{vw}a_{wu}+\frac{a_{vu}}{2n}\sum_{w=1}^na_{uw}a_{wv}\right]^2 \nonumber \\ 
&  =  \frac{1}{4n^2} \sum_{1 \leq u , v \leq n} \left[ a_{uv}  \left( \frac{1}{n} \sum_{w=1}^n a_{vw} a_{wu} \right)^2 + a_{vu} \left( \frac{1}{n} \sum_{w=1}^n a_{uw}a_{wv} \right)^2 \right] , 
\end{align*}
since $a_{uv}^2 = a_{uv}$, $a_{vu}^2 = a_{vu}$, and $a_{uv}a_{vu} = 0$. Hence, 
\begin{align} 
\sum_{s=1}^n \hat \lambda_s^2 & \rightarrow \frac{1}{4}\E[W(\eta_1,\,\eta_2)(t_W( \lp, \eta_1, \eta_2 ))^2+W(\eta_2,\,\eta_1)(t_W( \lp, \eta_2, \eta_1 ))^2] \nonumber \\ 
    & = \frac{1}{4}\E[W(\eta_1,\,\eta_2)W(\eta_2,\,\eta_1)(t_W( \lp, \eta_1, \eta_2 )-t_W( \lp, \eta_2, \eta_1 ))^2] + \| \triangle_{W} \|_2^2  \nonumber \\ 
    & = \frac{\tau_W^2}{2} + \| \triangle_{W} \|_2^2 ,  \nonumber 
\end{align} 
where the second equality follows from recalling \eqref{eq:W_tr} and last equality is from \eqref{eq:tau_W}. Now, since $\hat \lambda_1 \rightarrow t(\cc, W) = \gamma_W$,  we have, 
\begin{align}\label{eq:MTestimateI}
\sum_{s=2}^n \hat \lambda_s^2 \rightarrow \frac{\tau_W^2}{2} + \| \triangle_{W} \|_2^2 - \gamma_W^2 , 
\end{align} 
as $n \rightarrow \infty$, almost surely. 
Combining \eqref{eq:MTestimate}, \eqref{eq:MTestimateII}, and \eqref{eq:MTestimateI},  Lemma \ref{lem:emp_MGF} follows. 
\end{proof}

\section{Proofs from Section \ref{sec:cW}}

\subsection{Proof of Proposition \ref{prop:H01r}}
\label{sec:H01rpf}

Recall the definition of $t^{\triangle}_W(x)$ from \eqref{eq:ccjoin}. By applying the Cauchy-Schwarz inequality, 
\begin{align*}
t(\cccc \,,\, W) - t(\cc, \, W)^2 = \int_0^1 t^{\triangle}_W(x)^2 \mathrm{d} x  - \left( \int_0^1 t^{\triangle}_W(x) \mathrm{d} x \right)^2 \geq 0, 
\end{align*} 
for any tournamenton $W$. Applying this for the empirical tournamenton $W^{T_{n}}$ gives, $$\hat \sigma^2_{T_n} = t(\cccc \,,\, W^{T_n}) - t(\cc, \, W^{T_n})^2 \geq 0.$$ Now, we consider the following 2 cases: 
    
    \begin{itemize} 
    
    \item \textit{$W$ is $\triangle$-regular:} Since $\hat \sigma^2_{T_n} \geq 0$, it is now enough to show that $\mathbb{E}\hat \sigma_{T_n}^2 =O(1/n)$. Towards that, note that,
    \begin{align}\label{eq:exprERHWn}
        & \E \hat \sigma^2_{T_n}  = \E [t(\cccc \,,\, W^{T_n}) ] - \E[t(\cc, \, W^{T_n})^2 ] . 
    \end{align}
    Also, observe that for any tournamenton $W$, $$t(\cc \bigsqcup \cc,W) = t(\cc , W)^2,$$ where $\cc \bigsqcup \cc$ is the disjoint union of 2 copies of $\cc$. Then, by \cite[Lemma 2.4 ]{lovasz2006limits} (adapted in a straightforward manner from the graphon to the tournamenton setting), 
        \begin{align*}
      \left| \E[t(\cc, \, W^{T_n})^2 ] - t(\cc,W)^2 \right | =  \left| \E[t(\cc \bigsqcup \cc, \, W^{T_n}) ] - t(\cc \bigsqcup \cc, W) \right | \leq \frac{c}{n} ,  
    \end{align*}
    for some constant $c > 0$. 
    Similarly,  
    \begin{align*}
      \left| \E[t(\cccc, \, W^{T_n}) ] - t(\cccc, \, W) \right | \leq \frac{c}{n} . 
    \end{align*}
    Substituting these bounds in \eqref{eq:exprERHWn} gives,  
    \begin{align*}
        \mathbb{E}\hat \sigma^2_{T_n} \leq t(\cccc \,,\, W) - t(\cc, \, W)^2  +  \frac{2 c}{n}  = O \left( \frac{1}{n} \right ) ,  
    \end{align*} 
    since $\sigma_W^2 = t(\cccc \,,\, W) - t(\cc, \, W)^2 = 0$, when $W$ is $\triangle$-regular.

    \item \textit{$W$ is not $\triangle$-regular:} In this case, by arguments as in \cite[Corollary 10.4]{lovasz_book} it follows that $\hat \sigma^2_{T_n} \overset{P}{\rightarrow} \sigma_W^2$. Since $\sigma_W^2>0$, whenever $W$ is not $\triangle$-regular (recall Theorem \ref{thm:N} (1)), the result in Proposition \ref{prop:H01r} (2) follows. \hfill $\Box$ 
    \end{itemize}

    \subsection{Proof of Proposition \ref{prop:Huniform}} 
\label{sec:Hunifpf}

Denote by $([n])_4$ the collection of 4-tuples $\bm s = (s_1, s_2, s_3, s_4) \in [n]^4$ with distinct indices. For $\bm s \in ([n])_4$, define $a_{\bm s} = a_{s_1s_2} a_{s_1s_3} a_{s_1s_4} a_{s_2s_3} a_{s_2 s_4} a_{s_3s_4} $. 
Then recalling \eqref{eq:trTn} gives, 
$$\hat t_{\Tr_4}(T_n) = \frac{1}{(n)_4} \sum_{\bm s \in ([n])_4} a_{\bm s} . $$
Hence, 
\begin{align}
\mathrm{Var}[ \hat t_{\Tr_4}(T_n) ] = \frac{1}{(n)_4^2 } \sum_{\bm s, \bm s' \in ([n])_4} \mathrm{Cov}[ a_{\bm s}, a_{\bm s'} ] . 
\label{eq:varianceTn}
\end{align}
Note that $\mathrm{Cov}[a_{\bm s},a_{\bm s'}]=0$ whenever $\bm s$ and $\bm s'$ have no index in common. Since the number of ways to choose $\bm s,\bm s' \in ([n])_4$ such that they share a common index is $O(n^7)$, it follows from \eqref{eq:varianceTn} that
$\mathrm{Var}[\hat t_{\Tr_4}(T_n)]=O(\frac{1}{n})$. Now, since $\E[\hat t_{\Tr_4}(T_n)]=t(\Tr_4,W)$, applying Chebyshev's inequality shows that $\hat t_{\Tr_4}(T_n)$ converges in probability to $t(\Tr_4,W)$. Moreover, $t(\Tr_4,W) \ne \frac{1}{64}$ under $H_1$ in \eqref{eq:H01uniform}, since $\Tr_4$ is quasirandom-forcing (recall Definition~\ref{defn:TW}) by \cite{coregliano2017density}. This completes the proof of Proposition~\ref{prop:Huniform} (2).

To prove Proposition \ref{prop:Huniform} (1), note that under uniformity, that is, $W \equiv \frac{1}{2}$, one has $\mathrm{Cov}[a_{\bm s}, a_{\bm s'}] = 0$, whenever $|\bm s \cap \bm s' | \leq 1$. Since the number of ways to choose $\bm s, \bm s' \in ([n])_4$ such that they share at least 2 common indices is $O(n^6)$, we have from \eqref{eq:varianceTn} that $\mathrm{Var}[\hat t_{\Tr_4}(T_n)] = O(\frac{1}{n^2})$, under $H_0$ as in \eqref{eq:H01uniform}.  This establishes the result  in Proposition \ref{prop:Huniform} (1).     \hfill $\Box$

\subsection{Proof of Theorem \ref{thm:Ln}}
\label{sec:Lnpf}

Define the following events:  
$$\mathcal E_{\mathrm {reg}} = \left\{ \hat \sigma^2_{T_n} \leq \frac{1}{\sqrt n} \right\} \quad \text{ and }  \quad  \mathcal E_{\mathrm {unif}}  = \left\{ \left| \hat t_{\Tr_4}(T_n) - \frac{1}{64} \right| \leq \frac{1}{\sqrt n} \right\}.$$
Now, we consider the following cases: 
\begin{itemize}
\item $W$ is not $\triangle$-regular: Then, by Proposition \ref{prop:H01r}, $\P(\mathcal E_{\mathrm {reg}}) = o(1)$. Hence, recalling \eqref{eq:irreg_ci}, 
\begin{align} 
\P(\zeta(W) \in L_n) & = \P( \{ \zeta(W) \in L_n \} \cap \mathcal E_{\mathrm {reg}}^c ) + o(1)  \nonumber \\ 
& =  \P\left( \sqrt{n} \left| \frac{\hat \zeta_n - \zeta(W) }{24 \hat{\sigma}_{T_n}} \right| \leq  z_{\alpha/2} \cap \mathcal E_{\mathrm {reg}}^c \right) +  o(1) \nonumber \\ 
& =  \P\left( \sqrt{n} \left| \frac{\hat \zeta_n - \zeta(W) }{24 \hat{\sigma}_{T_n}} \right| \leq  z_{\alpha/2}  \right) +  o(1) \rightarrow 1- \alpha , \nonumber 
\end{align}
by Theorem \ref{thm:N} (1), Proposition \ref{prop:H01r}, and Slutsky's theorem. 
\item $W$ is $\triangle$-regular and $W \ne \frac{1}{2}$ on a set of positive measure. Then, by Propositions \ref{prop:H01r} and \ref{prop:Huniform}, $\P(\mathcal E_{\mathrm {reg}} \cap \mathcal E_{\mathrm {unif}}^c) \rightarrow 1$. Hence, recalling \eqref{eq:reg_aprx_ci},  
\begin{align} 
& \P(\zeta(W) \in L_n) \nonumber \\ 
& = \P( \{ \zeta(W) \in L_n \} \cap ( \mathcal E_{\mathrm {reg}} \cap \mathcal E_{\mathrm {unif}}^c ) ) + o(1)  \nonumber \\ 
& =  \P\left( \left\{ \hat q_{1-\alpha/2, T_n} \leq n \left( \frac{ \zeta(W) - \hat \zeta_n }{24} \right) \leq  \hat q_{\alpha/2, T_n} \right\} \cap ( \mathcal E_{\mathrm {reg}} \cap \mathcal E_{\mathrm {unif}}^c ) \right) +  o(1) \nonumber \\ 
& =  \P\left( \left\{ \hat q_{1-\alpha/2, T_n} \leq n \left( \frac{\zeta(W) - \hat \zeta_n }{24} \right) \leq  \hat q_{\alpha/2, T_n} \right\}  \right) +  o(1) .  \nonumber 
\end{align}
Now, by Theorem \ref{thm:TnJ} and Polya's theorem, $\hat q_{\alpha, T_n}|T_n \overset{P}\longrightarrow q_{\alpha}$, where $q_{\alpha}$ is the $(1-\alpha)$-th quantile of the distribution in \eqref{eq:reg}. Then Theorem \ref{thm:N} (2) and Slutsky's theorem implies, $\P(\zeta(W) \in L_n) \rightarrow 1 - \alpha$. 
\item  $W = \frac{1}{2}$ almost everywhere. Then, $\P(\mathcal E_{\mathrm {reg}} \cap \mathcal E_{\mathrm {unif}}) \rightarrow 1$ and 
\begin{align*} 
\P(\zeta(W) \in L_n) & = \P( \{ \zeta(W) = 0 \} \cap (\mathcal E_{\mathrm {reg}} \cap \mathcal E_{\mathrm {unif}}) ) + o(1)  \nonumber \\ 
& = \P_{W = \frac{1}{2} }( \mathcal E_{\mathrm {reg}} \cap \mathcal E_{\mathrm {unif}}) + o(1) \rightarrow 1, 
\end{align*} 
 by Proposition \ref{prop:Huniform}. 
\end{itemize}

\small

\subsection*{Acknowledgements} The authors were supported by NSF CAREER grant DMS 2046393 and a Sloan Research Fellowship. The authors also thank Huy Pham for helpful discussions.

\subsection*{AI Disclosure} GPT-5.5 Pro was used to assist with the proof of Theorem \ref{thm:W}. Specifically, the idea of using moments/polynomials in Lemmas \ref{lm:TJr} and  \ref{lm:q} was suggested by GPT-5.5 Pro. The authors reviewed and rewrote the proofs with additional details.

\bibliography{ref.bib,bibliography.bib}
\bibliographystyle{abbrvnat}

\normalsize  

\appendix

\section{Properties of Degree-Regular Tournamentons}
\label{sec:degreepf}

In this section we collect a few basic facts about degree-regular tournamentons (recall Remark \ref{remark:degree}). 

\begin{lem} 
If $W$ is a degree-regular tournamenton, then $W$ is also $\triangle$-regular. In particular, $t^{\triangle}_{W}(x)= \frac{1}{8}$, almost everywhere. 
\label{lem:deg_reg}
\end{lem}

\begin{proof} 
Suppose $W$ is degree regular. Then we know from Remark \ref{remark:degree} that 
$d^{\uparrow}_W(x) = d^{\downarrow}_W(x) = \frac{1}{2}$, for almost every $x \in [0,\,1]$. Then,
\begin{align}\label{eq:Wregular}
    t^{\triangle}_{W}(x) &=\int_{[0,\,1]^2}W(x,\,y)W(y,\,z)W(z,\,x)\mathrm{d}y\mathrm{d}z \nonumber \\
    &=\int_{[0,\,1]^2}(1-W(y,\,x))(1-W(z,\,y))(1-W(x,\,z))\mathrm{d}y\mathrm{d}z . 
    %&=1-(1-c)-\frac{1}{2}-c+(1-c)^2+c^2+c(1-c)-t^{\triangle}_{W}(x),
\end{align}
Note, $\int_{0}^1W(y,\,x) \mathrm{d} y = d^{\downarrow}_W(x) = \frac{1}{2}$, $\int_{0}^1 W(x,\,z) \mathrm{d} z = d^{\uparrow}_W(x) = \frac{1}{2}$, almost everywhere, and $\int_{[0,\,1]^2}W(y,\,z)\mathrm{d}y\mathrm{d}z  =\frac{1}{2}$. 
 Moreover, for almost every $x \in [0, 1]$, 
 \begin{align*}
 \int_{[0,\,1]^2}W(y,\,x) W(z,\,y) \mathrm{d}y\mathrm{d}z & =  \int_{[0,\,1]^2} W(y,\,x) d^{\downarrow}_W(y) \mathrm d y \mathrm d z = \tfrac{1}{2} d^{\downarrow}_W(x) = 
\tfrac{1}{4}  ,  \\ 
 \int_{[0,\,1]^2} W(z,\,y) W(x,\,z) \mathrm{d}y\mathrm{d}z & =  \int_{[0,\,1]^2} d^{\uparrow}_W(z) W(x,\,z) \mathrm d y \mathrm d z = \tfrac{1}{2} d^{\uparrow}_W(x) = \tfrac{1}{4} , \\ 
 \int_{[0,\,1]^2} W(y,\,x) W(x,\,z) \mathrm{d}y\mathrm{d}z & =   d^{\uparrow}_W(x) d^{\downarrow}_W(x) = \tfrac{1}{4} . 
\end{align*} 
Also, relabeling the variables gives, $\int_{[0,\,1]^2} W(y,\,x) W(z,\,y) W(x,\,z) \mathrm{d}y\mathrm{d}z =  t^{\triangle}_{W}(x).$
Combining the above with \eqref{eq:Wregular} gives, $t^{\triangle}_{W}(x) = \frac{1}{8}$, for almost every $x \in [0, 1]$. 
\end{proof}

\begin{remark} Note that the converse of the result in Lemma \ref{lem:deg_reg} does not hold, in general. For instance, the Condorcet tournamenton is $\triangle$-regular (recall Remark~\ref{remark:conditional}), but it is not degree-regular when $p \neq \frac{1}{2}$. In particular, for $W$ as in \eqref{eq:Tn}, the out-degree function 
$$d_W^{\uparrow}(x)=p+x(1-2p), \quad \text{ for } x\in[0,1] , $$ 
which is non-constant, for $p \ne \frac{1}{2}$. 
\label{remark:randomdegree}
\end{remark}

\begin{lem} Let $t_{W}(\lp, x, y)$ be the kernel defined in \eqref{eq:rpxy}. Then, for $x, y \in [0, 1]$,  
\begin{align}\label{eq:rplpxy}
t_{W}(\lp, x, y) - t_{W}(\lp, y, x) = d^{\uparrow}_W(y)-d^{\uparrow}_W(x) , 
\end{align} 
where $d^{\uparrow}_W$ is the out-degree function defined in \eqref{eq:directeddegree}. Hence, $t_{W}(\lp, x, y)$ is symmetric if and only if $W$ is a degree-regular tournamenton.
\label{lem:rp}
\end{lem}

\begin{proof} Fix $x, y \in [0, 1]$. Then note that 
\begin{align*}
    t_{W}(\lp, x, y) =\int_0^1W(y,\,z)W(z,\,x)\mathrm{d}z & =\int_0^1(1-W(z,\,y))(1-W(x,\,z))\mathrm{d}z\\
    &=1-d^{\downarrow}_W( y )-d^{\uparrow}_W(x)+t_{W}(\lp, y, x) \\ 
    & =d^{\uparrow}_W(y)-d^{\uparrow}_W(x)+t_{W}(\lp, y, x).
\end{align*} 
This proves \eqref{eq:rplpxy}. 
\end{proof}

\section{Convergence of $\bar{\triangle}_{T_n}$}   
\label{sec:Wtriangleconvergence}

Recall from \eqref{eq:triangleWxy}, 
$\bar{\triangle}_{T_n}(x, y) = \hat \triangle_{T_n}( \lceil n x \rceil, \lceil n y \rceil )$, for $x, y \in [0, 1]$. This is a symmetric function from $[0, 1]^2 \rightarrow [0, 1]$, hence, it is  a graphon. In this section, we will show that $\bar{\triangle}_{T_n}$ converges in the cut-metric to the graphon $\triangle_{W}$ (recall \eqref{eq:Wp}). 
    
    \begin{defn}\label{defn:Wconvergence} \cite[Chapter 8]{lovasz_book}
    The {\it cut-distance} between two bounded functions $W_1, W_2 : [0, 1]^2 \rightarrow \mathbb R$ is 
    \begin{align}\label{eq:W12}
    ||W_1-W_2||_{\square}:=\sup_{f, g: [0, 1] \rightarrow [0, 1]} \left|\int_{[0, 1]^2} \left(W_1(x, y)-W_2(x, y)\right) f(x) g(y) \mathrm dx \mathrm dy \right|. 
    \end{align} 
    The {\it cut-metric} between $W_1, W_2$ is defined as,  
    \begin{align*}
    \delta_{\square}(W_1, W_2):= \inf_{\psi}||W_1^{\psi}-W_2||_{\square}, 
    \end{align*} 
    with the infimum taken over all measure-preserving bijections $\psi: [0, 1] \rightarrow [0, 1]$, and  $W_1^\psi(x, y):= W_1(\psi(x), \psi(y))$, for $x, y \in [0, 1]$. 
    \end{defn}

With the above definition, we can now prove the following result:

\begin{lem} Let be $\bar{\triangle}_{T_n}$ and $\triangle_{W}$ as defined in  \eqref{eq:triangleWxy} and \eqref{eq:Wp}, respectively. Then 
$$\delta_{\square}(\bar{\triangle}_{T_n}, \triangle_{W}) \rightarrow 0, $$
$n \rightarrow \infty$, almost surely.  
\label{lm:Wtriangleconvergence}
\end{lem}

\begin{proof}
Let $W^{T_n}$ be the empirical tournament associated with $T_n$ (as in \eqref{eq:Tn}). 
Denote 
\begin{align*}
    \triangle_{W^{T_n}}(x,\,y) :=\frac{W^{T_n}(x, y) t_{W^{T_n}}(\lp, x, y) + W^{T_n}(y, x) t_{W^{T_n}}(\lp, y, x) }{2} . 
\end{align*} 
Note that $\bar{\triangle}_{T_n}(x, y) = \triangle_{W^{T_n}}(x,\,y)$, for almost every $(x, y) \in [0, 1]^2$.   Hence, to prove the result we have to show that,  given a tournamenton $U$, the map $U \rightarrow  \triangle_{U}(x,\,y)$, where 
\begin{align}
    \triangle_{U}(x,\,y) :=\frac{U(x, y) t_{U}(\lp, x, y) + U(y, x) t_{U}(\lp, y, x) }{2} . 
    \label{eq:Uxy}
\end{align} 
is continuous in the cut-metric.  To this end, first note that a straightforward adaption of the  sampling lemma for graphons \cite[Lemma 10.16]{lovasz_book} to the setting of tournamentons, gives $\delta_{\square}(W^{T_n}, W) \rightarrow 0$, almost surely. Now, choose a sequence of measure-preserving transformations $\{\psi_n\}_{n \geq 1}$ such that $\|W_{T_n}^{\psi_n}-W\|_\square\to 0$. Note that the functional $U \rightarrow \triangle_U$ is equivariant under relabeling, that is, $\triangle_{U^\psi}=(\triangle_U)^\psi$. Thus, it suffices to prove that
\begin{align}\label{eq:Un}
\|U_n-U\|_\square\to 0 \quad\Longrightarrow\quad \|\triangle_{U_n}-\triangle_U\|_\square\to 0. 
\end{align}

We begin by considering the first term in \eqref{eq:Uxy}. Fix $f, g: [0, 1] \rightarrow [0, 1]$. Then by a telescoping argument and triangle inequality, 
\begin{align}
& \left| \int_{[0, 1]^2} \left( U_n(x, y) t_{U_n}(\lp, x, y) - U(x, y) t_{U}(\lp, x, y) \right) f(x) g(y) \,\mathrm{d}x\,\mathrm{d}y \right| \leq R_1 + R_2 + R_3 , \nonumber 
\end{align}
where 
\begin{align} 
R_1 &= \left|\int_{[0, 1]^3} (U_n(x,y)-U(x,y)) U_n(y,z)U_n(z,x) f(x) g(y) \,\mathrm{d}x\,\mathrm{d}y\,\mathrm{d}z \right| , \nonumber \\ 
R_2 & = \left| \int_{[0, 1]^3}
U(x,y) (U_n(y,z)-U(y,z)) U_n(z,x) f(x) g(y) \,\mathrm{d}x\,\mathrm{d}y\,\mathrm{d}z \right| , \nonumber \\ 
R_3 & = \left| \int_{[0, 1]^3}
U(x,y)U(y,z) (U_n(z,x)-U(z,x)) f(x) g(y) \,\mathrm{d}x\,\mathrm{d}y\,\mathrm{d}z \right| . \nonumber 
\end{align} 
Recalling \eqref{eq:W12}, note that each of the above terms are bounded by $\|U_n-U\|_\square$.
The same bound holds for the second term in \eqref{eq:Uxy}. Hence, for  $f, g: [0, 1] \rightarrow [0, 1]$, 
$$\left| \int_{[0, 1]^2} \left( \triangle_{U_n}(x,\,y)  - \triangle_{U}(x,\,y) \right)  f(x) g(y) \,\mathrm{d}x\,\mathrm{d}y \right| \leq 3 \|U_n-U\|_\square.$$
This proves \eqref{eq:Un} and completes the proof of Lemma \ref{lm:Wtriangleconvergence}.  \end{proof}

\end{document}